\documentclass[10pt, a4paper]{amsart}

\pdfoutput=1

\usepackage{mathrsfs, amsmath, amsthm, amssymb, stmaryrd}%
\usepackage{tikz}%
\usepackage[linktocpage]{hyperref}%
\usepackage[all]{xy}%
\xyoption{2cell}%
\UseAllTwocells%

\DeclareMathOperator{\id}{id}

\DeclareMathOperator{\op}{op}
\DeclareMathOperator{\Lan}{Lan}
\DeclareMathOperator{\Mod}{\mathbf{Mod}}
\DeclareMathOperator{\Rep}{Rep}
\DeclareMathOperator{\Vect}{\mathbf{Vect}}

\DeclareMathOperator{\fp}{fp}
\DeclareMathOperator{\Hom}{Hom}
\DeclareMathOperator{\End}{End}
\DeclareMathOperator{\Spec}{Spec}

\DeclareMathOperator{\rk}{rk}
\DeclareMathOperator{\Sym}{Sym}
\DeclareMathOperator{\sgn}{sgn}

\DeclareMathOperator{\ev}{ev}

\DeclareMathOperator{\iso}{iso}

\DeclareMathOperator{\Sh}{Sh}
\DeclareMathOperator{\Rex}{\mathbf{Rex}}

\DeclareMathOperator{\Ind}{Ind}

\DeclareMathOperator{\Aff}{\mathbf{Aff}}

\DeclareMathOperator{\QCoh}{\mathbf{QCoh}}
\DeclareMathOperator{\VB}{\mathbf{VB}}

\DeclareMathOperator{\BGL}{\mathrm{BGL}}

\DeclareMathOperator{\fpqc}{\mathit{fpqc}}

\DeclareMathOperator{\Aut}{Aut}
\DeclareMathOperator{\lax}{lax}
\DeclareMathOperator{\Tors}{Tors}

\DeclareMathOperator{\colim}{colim}

\DeclareMathOperator{\Lex}{\mathbf{Lex}}

\DeclareMathOperator{\Set}{\mathbf{Set}}

\DeclareMathOperator{\Comod}{\mathbf{Comod}}

\DeclareMathOperator{\Nat}{\mathrm{Nat}}
\DeclareMathOperator{\Fun}{\mathrm{Fun}}
\DeclareMathOperator{\LF}{\mathrm{LF}}

\DeclareMathOperator{\CAlg}{\mathrm{CAlg}}

\newcommand{\ca}[1]{\mathscr{#1}}

\newcommand{\Prs}[1]{\mathcal{P}\ca{#1}}

\DeclareMathOperator{\U}{O}

\newcommand{\ubar}[1]{\underline{#1\mkern-4mu}\mkern4mu }

\newcommand{\ten}[1]{\mathop{{\otimes}_{#1}}}

\newcommand{\pb}[1]{\mathop{{\times}_{#1}}}
\newcommand{\po}[1]{\mathop{{+}_{#1}}}

\newcommand{\defl}{\mathrel{\mathop:}=}

\theoremstyle{plain}
\newtheorem{thm}{Theorem}[subsection]
\newtheorem{prop}[thm]{Proposition}
\newtheorem{lemma}[thm]{Lemma}
\newtheorem{cor}[thm]{Corollary}

\theoremstyle{definition}
\newtheorem{example}[thm]{Example}
\newtheorem{rmk}[thm]{Remark}
\newtheorem{dfn}[thm]{Definition}
\newtheorem{notation}[thm]{Notation}

\newtheoremstyle{citing}{}{}{\itshape}{}{\bfseries}{.}{ }{\thmnote{#3}}
\theoremstyle{citing}

\newtheoremstyle{citingdfn}{}{}{}{}{\bfseries}{.}{ }{\thmnote{#3}}
\theoremstyle{citingdfn}

\numberwithin{equation}{section}


\author{Daniel Sch\"appi}
\thanks{The author gratefully acknowledges support through the Swiss National Foundation Fellowship P2SKP2\_148479}
\address{School of Mathematics and Statistics, University of
  Sheffield, Sheffield, S3 7RH, UK}
\email{d.schaeppi@sheffield.ac.uk}

\title{Constructing colimits by gluing vector bundles}

\begin{document}

\begin{abstract}
 As already observed by Gabriel, coherent sheaves on schemes obtained by gluing affine open subsets can be described by a simple gluing construction. An example due to Ferrand shows that this fails in general for pushouts along closed immersions, though the gluing construction still works for flat coherent sheaves.

 We show that by further restricting this gluing construction to vector bundles, we can construct pushouts along arbitrary morphisms (and more general colimits) of certain algebraic stacks called Adams stacks. The proof of this fact uses generalized Tannaka duality and a variant of Deligne's argument for the existence of fiber functors which works in arbitrary characteristic.

 We use this version of Deligne's existence theorem for fiber functors as a novel way of recognizing stacks which have atlases. It differs considerably from Artin's algebraicity results and their generalizations: rather than studying conditions on the functor of points which ensure the existence of an atlas, our theorem identifies conditions on the category of quasi-coherent sheaves of the stack which imply that an atlas exists.
\end{abstract}

\maketitle

\tableofcontents

\section{Introduction}\label{section:introduction}

\subsection{Overview}
 Gluing constructions are central to modern algebraic geometry: general schemes are obtained by gluing open affine subschemes along their intersections. This construction is a special case of a colimit. However, in most categories arising in algebraic geometry general colimits rarely exist. Results in the literature on pushouts of schemes \cite{FERRAND} and of algebraic stacks \cite{RYDH_PUSHOUTS} have restrictive hypotheses on the morphisms involved. More importantly, there are examples of diagrams which do not have a colimit in the category of schemes (see for example MathOverflow \url{http://mathoverflow.net/questions/9961}).

 From a categorical perspective, it is often convenient to work in a category where \emph{arbitrary} colimits of small diagrams exist, that is, in a \emph{cocomplete} category. Examples of such arising in algebraic geometry are the category of \emph{all} sheaves (respectively the 2-category of \emph{all} stacks) for the various Grothendieck topologies on the category of (affine) schemes. The drawback of these categories is that they also contain objects that are quite ``far removed'' from geometry. From this perspective it would be preferable to have a category of \emph{algebraic} spaces or \emph{algebraic} stacks which is cocomplete.

 Moreover, colimits of general schemes often exhibit rather pathological behaviour (for example, the affine plane with double origin, or the scheme without closed points constructed in \cite{SCHWEDE}). For practical purposes it would be desireable to have a category of ``nice'' algebraic spaces or stacks (for example stacks satisfying some separation axioms).

 The main result of our paper is that the 2-category of Adams stacks has all (bicategorical) colimits. Recall that an \emph{Adams stack} over a commutative ring $R$ is a stack on the $\fpqc$-site of affine schemes over $R$ which is quasi-compact, has an affine diagonal, and which has the \emph{resolution property}: all finitely presentable quasi-coherent sheaves are quotients of vector bundles. These are precisely the stacks associated to Adams Hopf algebroids (see \cite[Theorem~1.3.1]{SCHAEPPI_STACKS}). 

 In general, Adams stacks need not be Artin stacks. The reason is simply that, for Adams stacks, there are no finiteness conditions imposed on the source and target morphisms of the corresponding affine groupoids. Thus various examples arising in algebraic topology, such as the moduli stack of formal groups, are Adams stacks but not Artin stacks. Using the language of Adams Hopf algebroids and their comodules (instead of stacks and their categories of quasi-coherent sheaves), these have been studied extensively by algebraic topologists, see for example \cite{ADAMS_STABLE, LANDWEBER, GOERSS_HOPKINS, HOVEY}.

 On the other hand, certainly not every Artin stack is an Adams stack. Firstly, all Adams stacks are assumed to be quasi-compact and to have an affine diagonal. Moreover, as Burt Totaro pointed out, there is a quasi-compact Artin stack with affine diagonal which does \emph{not} have the resolution property. Namely, the stack constructed in \cite[Example~3.12]{EDIDIN_HASSETT_KRESCH_VISTOLI} does not have the resolution property by \cite[Theorem~1.1]{TOTARO}. 

 The main result of our paper is the following. We prove it in \S \ref{section:colimits_of_adams_stacks}.

 \begin{thm}\label{thm:complete_cocomplete}
 Let $R$ be a commutative ring. The 2-category $\ca{AS}$ of Adams stacks over $R$ has all bicategorical limits and colimits.
\end{thm}

 This shows that the class of all Adams stacks is more convenient to work with than the class of all Artin stacks.

 One immediate consequence of the above result is that for almost any reasonable stack $X$ on the $\fpqc$-site $\Aff_R$, there exists an Adams stack $X^{\prime}$ together with a morphism $X \rightarrow X^{\prime}$ which is universal among morphisms to Adams stacks (see Theorem~\ref{thm:universal_adams_stack} for details). The only assumption on $X$ we need to make is a set-theoretical size restriction: one has to be able to describe $X$ using a set of data (rather than a proper class).

 We illustrate with a few examples the type of construction made possible by Theorem~\ref{thm:complete_cocomplete}. One way to construct a new Adams stack from an existing one is by pinching substacks: given a closed substack $Z \subseteq X$, we can form the pushout
\[
 \xymatrix{Z \ar[d] \ar[r] & X \ar[d] \\ \ast \ar[r] & X \po{Z} \ast}
\]
 in the 2-category of Adams stacks, which we denote by $X \slash Z$ since it is obtained by formally identifying all the points of $Z$.

 Another construction we can perform is that of formally adding automorphisms to a given point of an Adams stack $X$ (over $R$). If $G$ is a flat affine group scheme over $R$ such that the classifying stack $BG$ of principal $G$-bundles is an Adams stack (for example, $G=\mathrm{GL}_n$, or $G$ is flat and $R$ is a Dedekind ring), then the pushout
\[
 \xymatrix{\ast \ar[r]^-{x} \ar[d] & X \ar[d] \\ BG \ar[r] & X \po{\ast} BG }
\]
 formally adds the affine group scheme $G$ to the automorphism group of the point $x$ of $X$. We can also form pushouts among arbitrary morphisms, for example, we can form the pushout of the fold map $\nabla \colon X + X \rightarrow X$ along itself, thereby obtaining an Adams stack $X \odot \Aut$ with a universal property dual to that of the inertia stack: giving a morphism from $X \odot \Aut$ to $Y$ amounts to giving a morphism $f \colon X \rightarrow Y$ together with an automorphsim $\varphi \colon f \Rightarrow f$. 

 Note that even though all Adams stack are assumed to be quasi-compact, Theorem~\ref{thm:complete_cocomplete} implies that the 2-category of Adams stacks has infinite coproducts.

 Our existence proof for colimits of Adams stacks is rather indirect. In a sequel we will give a more detailed analysis of the construction, which will allow us to identify a number of colimit diagrams in the 2-category of Adams stacks. This will extend some recent results due to Bhatt \cite{BHATT} about colimits of algebraic spaces to Adams stacks.

\subsection{Method of proof}
 To prove Theorem~\ref{thm:complete_cocomplete}, we use the Tannaka duality results for Adams stacks established in \cite{SCHAEPPI_STACKS, SCHAEPPI_INDABELIAN, SCHAEPPI_GEOMETRIC}. One of the main results of these papers is that the 2-category of \emph{weakly Tannakian categories} (see Definition~\ref{dfn:weakly_tannakian}) and right exact symmetric strong monoidal functors between them is contravariantly equivalent to the 2-category of Adams stacks. Thus, in order to prove Theorem~\ref{thm:complete_cocomplete}, it suffices to show that the category of weakly Tannakian categories is bicategorically complete and cocomplete. In fact, we will prove the following slightly more general theorem. Recall that we call a symmetric monoidal $R$-linear category with finite colimits \emph{right exact symmetric monoidal} if the functors given by tensoring with a fixed object are right exact. We denote the category of right exact symmetric monoidal categories and right exact symmetric strong monoidal functors between them by $\ca{RM}$. The following theorem is a consequence of Theorem~\ref{thm:weakly_tannakian_coreflective} (where the construction of the coreflection mentioned below is made explicit).

\begin{thm}\label{thm:coreflective}
 The inclusion of the full sub-2-category of $\ca{RM}$ consisting of weakly Tannakian categories has a right biadjoint. In particular, the 2-category of weakly Tannakian categories has all bicategorical limits, and it is closed under all bicategorical colimits in $\ca{RM}$.
\end{thm}
 
 This theorem also provides a more conceptual explanation for previous results of the author, namely \cite[Theorem~6.6]{SCHAEPPI_INDABELIAN} (which states that weakly Tannakian categories are closed under finite coproducts in $\ca{RM}$), \cite[Theorem~4.2]{SCHAEPPI_GEOMETRIC} (closure of weakly Tannakian categories under bicategorical pushouts), and \cite[Theorem~6.11]{SCHAEPPI_GEOMETRIC} (closure under arbitrary bicategorical colimits if $R$ is a $\mathbb{Q}$-algebra).

 Similar techniques can be used to generalize a well-known construction for Tannakian categories to the weakly Tannakian case. Given a weakly Tannakian category $\ca{A}$ and a set $S$ of objects with duals in $\ca{A}$, we can form the weakly Tannakian category ``generated'' by $S$. The key difference to the Tannakian case is that this might not be a subcategory of $\ca{A}$. Instead it is merely equipped with a universal morphism to $\ca{A}$ (see Proposition~\ref{prop:stack_generated_by_vector_bundles} for details).

 The main difficulty in the construction of the universal weakly Tannakian category of Theorem~\ref{thm:coreflective} lies in the construction of a \emph{fiber functor}, that is, a faithful and exact symmetric strong monoidal functor to a category of modules over a commutative $R$-algebra. The existence of such a functor is part of the definition of a weakly Tannakian category. Note that fiber functors correspond to $\fpqc$-atlases of the stack associated to the weakly Tannakian category in question. To establish the existence of such a functor in the case of our universal construction, we use a variant of Deligne's argument for the existence of fiber functors for Tannakian categories in characteristic zero. Our variant works independently of the characteristic of the ground field (or ring), though the characterization is less tractable (see Theorem~\ref{thm:description}). Nevertheless, it suffices to prove Theorem~\ref{thm:complete_cocomplete}. 

 Using this new result we can also generalize one of the main results of \cite{SCHAEPPI_GEOMETRIC} which characterizes weakly Tannakian categories intrinsically in the case where $R$ is a $\mathbb{Q}$-algebra. More precisely, we can show that one of the conditions of \cite[Theorem~1.4]{SCHAEPPI_GEOMETRIC} is implied by the others. Recall that the \emph{rank} of an object $V$ with a dual in a symmetric monoidal $R$-linear category $\ca{A}$ is by definition the trace of the identity of $V$, that is, the composite
\[
 \xymatrix{\U \ar[r] & V \otimes V^{\vee} \cong V^{\vee} \otimes V \ar[r] & \U}
\]
 of the evaluation and the coevaluation in the endomorphism ring of the unit object $\U$ of $\ca{A}$. We write $\rk V$ for the rank of $V$. The \emph{$k$-th exterior power} $\Lambda^k V$ of $V$ is given by the splitting of the idempotent
\[
 \frac{1}{k!} \sum_{\sigma \in \Sigma_k} \sgn(\sigma) \sigma \colon V^{\otimes k} \rightarrow V^{\otimes k}
\]
 (where $\sigma$ denotes the automorphism of $V^{\otimes k}$ given by permuting the factors of the $k$-fold tensor product of $V$). 

 Finally, we call a finitely cocomplete $R$-linear category \emph{ind-abelian} if its category of ind-objects is abelian. All abelian categories are ind-abelian. An example of an ind-abelian category which is not abelian is given by the category of finitely presentable modules over a non-coherent ring. There is also an intrinsic characterization of ind-abelian categories, see \cite[\S 1]{SCHAEPPI_INDABELIAN}.

\begin{thm}\label{thm:existence_of_fiber_functor}
 Let $R$ be a $\mathbb{Q}$-algebra, and let $\ca{A}$ be a right exact symmetric monoidal $R$-linear category. Then $\ca{A}$ is weakly Tannakian (that is, it is equivalent to the category of finitely presentable quasi-coherent sheaves on an Adams stack over $R$) if and only if the following conditions hold:
\begin{enumerate}
 \item[(i)] The category $\ca{A}$ is ind-abelian;
\item[(ii)] For every object $A \in \ca{A}$ there exists an object $V \in \ca{A}$ with a dual and an epimorphism $V \rightarrow A$;
\item[(iii)] If $V \in \ca{A}$ has a dual and $\rk(V)=0$, then $V \cong 0$;
\item[(iv)] For all objects $V \in \ca{A}$ with a dual there exists a number $k \in \mathbb{N}$ such that $\Lambda^k(V) \cong 0$.
\end{enumerate}
\end{thm}

 One of the main difficulties we encounter in proving Theorem~\ref{thm:coreflective} is that we need to generalize various concepts from algebraic geometry (for example, locally free objects and locally split epimorphisms) to arbitrary right exact symmetric monoidal categories. In other words, we need to find definitions of these concepts which only use the structure present in a right exact symmetric monoidal category, and which coincide with the usual definition in the case where the symmetric monoidal category is the category of finitely presentable quasi-coherent sheaves of an Adams stack. Then we need to check that the weakly Tannakian categories corresponding to certain classifying stacks have a (dual) universal property among \emph{all} right exact symmetric monoidal categories. For example, in \S \ref{section:torsors} we will develop the theory of torsors for general symmetric monoidal categories in order to establish that $\QCoh(\BGL_d)$ classifies ``locally free objects'' in very general symmetric monoidal categories. 

 Finally, to check that our construction has the desired universal property, we use a recent result of Tonini which---suitably interpreted, see \S \ref{section:universal}---gives a ``generators and relations'' presentation of the category of quasi-coherent sheaves of an Adams stacks purely in terms of vector bundles. We will prove a slightly different version of Tonini's result which closely follows the proof of a similar result in \cite[\S 3.1]{BHATT} for categories of quasi-coherent sheaves on algebraic spaces. This gives us sufficient control over symmetric monoidal functors \emph{out} of such symmetric monoidal categories to establish the universal property of the coreflection of Theorem~\ref{thm:coreflective}.

\section*{Acknowledgments}
 Some of the results in \S \ref{section:torsors} grew out of discussions with Martin Brandenburg during a visit at the University of Western Ontario. I am grateful to both Martin and to Rick Jardine, who made this visit possible. I would also like to thank Fabio Tonini for sending me an early preprint of his recent paper \cite{TONINI}. His results are used to prove that our constructions actually satisfy the desired universal properties.

 I thank Bhargav Bhatt and Martin for pointing out some corrections and suggesting improvements.

 This paper was written at the University of Sheffield during a stay made possible by the Swiss National Foundation Fellowship P2SKP2\_148479.
\section{Categorical background}\label{section:background} 

\subsection{Tensor categories}
 Throughout we fix a commutative ring $R$. Unless stated otherwise, the categories and functors we consider are \emph{$R$-linear}. This means that the hom-sets between two objects $C$, $C^{\prime}$ of a category $\ca{C}$---which we denote by $\ca{C}(C,C^{\prime})$---are endowed with the structure of an $R$-module and that composition of morphisms is $R$-bilinear. Functors between $R$-linear categories are called $R$-linear if they preserve the $R$-module structure.

 An $R$-linear \emph{symmetric monoidal category} is a category $\ca{C}$ equipped with an $R$-bilinear functor
\[
 \otimes \colon \ca{C} \times \ca{C} \rightarrow \ca{C} \smash{\rlap{,}}
\]
 a \emph{unit object} $\U \in \ca{C}$ (which we also denote by $\U_{\ca{C}}$ if there is a potential for confusion), and natural isomorphisms
\[
 \alpha \colon (C \otimes C^{\prime}) \otimes C^{\prime\prime} \cong C \otimes (C^{\prime}\otimes C^{\prime\prime}) \text{,} \quad \sigma \colon C \otimes C^{\prime} \cong C^{\prime} \otimes C \quad \text{and} \quad 
C \otimes \U \cong C \cong \U \otimes C
\]
 for all objects $C, C^{\prime}, C^{\prime\prime} \in \ca{C}$. These are subject to axioms which ensure that any two morphisms built from these isomorphisms which induce the same permutation on objects coincide (see \cite[Theorem~4.2]{MACLANE_COHERENCE}). In particular, we often get a non-trivial action of the symmetric group $\Sigma_n$ on $n$ letters on the $n$-fold tensor product $X^{\otimes n}$ of an object $X \in \ca{C}$. The main example to keep in mind is the category of quasi-coherent sheaves on a scheme or stack $X$ over $R$, where $\U$ is the sheaf of rings $\U_X$, considered as a sheaf of modules over itself.

 A \emph{symmetric monoidal functor} $F \colon \ca{C} \rightarrow \ca{D}$ is an $R$-linear functor $F \colon \ca{C} \rightarrow \ca{D}$ equipped with natural transformations
\[
 \varphi^F_{C,C^{\prime}} \colon FC \otimes FC^{\prime} \rightarrow F(C\otimes C^{\prime}) \quad\text{and}\quad
 \varphi_0^F \colon \U_{\ca{D}} \rightarrow F\U_{\ca{C}}
\]
 subject to compatibility conditions. The functor $F$ is called symmetric \emph{strong} monoidal if $\varphi^F_{C,C^{\prime}}$ and $\varphi^F_0$ are isomorphisms. If $f \colon X \rightarrow Y$ is a morphism of quasi-compact quasi-separated schemes, then $f^{\ast} \colon \QCoh(Y) \rightarrow \QCoh(X)$ is symmetric strong monoidal and the direct image functor $f_{\ast} \colon \QCoh(X) \rightarrow \QCoh(Y)$ is always symmetric monoidal, but rarely strong. A natural transformation between symmetric monoidal functors $F,G$  is called \emph{symmetric monoidal} if it is compatible with the structure morphisms $\varphi^F$, $\varphi^F_0$ and $\varphi^G$, $\varphi^G_0$.

 A symmetric monoidal category is called \emph{closed} if the functors $C \otimes - \colon \ca{C} \rightarrow \ca{C}$ have a right adjoint for all $C \in \ca{C}$. In this case we usually write $[C,-]$ or $[C,-]_{\ca{C}}$ for the right adjoint, and we call $[C,C^{\prime}]$ the \emph{internal hom} of $\ca{C}$.

 An $R$-linear category is called \emph{cocomplete} if it has colimits of all small diagrams (that is, diagrams whose indexing category has a set of objects rather than a proper class). A \emph{tensor category} is a symmetric monoidal closed category which is also cocomplete. A \emph{tensor functor} is a symmetric strong monoidal functor which also has a right adjoint. Any such right adjoint inherits a unique structure of a symmetric monoidal functor in such a way that the unit and counit become symmetric monoidal natural transformations (this is a very general fact about categories equipped with algebraic structure, see \cite{KELLY_DOCTRINAL}). Note that this definition of tensor category differs slightly from the one used in \cite{SCHAEPPI_GEOMETRIC}, where we also assumed that the tensor category is locally finitely presentable (see below). The reason for this discrepancy is that the results of \S \ref{section:torsors} hold at this level of generality. Examples of tensor categories which are \emph{not} locally finitely presentable are given by categories of quasi-coherent sheaves on schemes and stacks which are not quasi-compact. We do not consider such objects in this article, but \S \ref{section:torsors} suggests that some of our results could be extended to this context.

 A \emph{commutative algebra} $A$ in a symmetric monoidal category $\ca{C}$ over $R$ consists of an object $A \in \ca{C}$, a \emph{unit} $\eta \colon \U \rightarrow A$ and a multiplication $\mu \colon A \otimes A \rightarrow A$ such that various diagrams expressing the fact that $\mu$ is a commutative multiplication and $\eta$ is a unit commute. An \emph{$A$-module} in $\ca{C}$ is an object $M \in \ca{C}$ with an \emph{action} $\alpha \colon A \otimes M \rightarrow M$ which is compatible with $\eta$ and $\mu$. The category of $A$-modules in $\ca{C}$ is denoted by $\ca{C}_A$. If $\ca{C}$ has coequalizers and tensoring with an object in $\ca{C}$ preserves them, then $\ca{C}_A$ is again a symmetric monoidal category over $R$, with tensor product given by the usual coequalizer involving the two $A$-actions.

 The ``free $A$-module functor'' $A \otimes - \colon \ca{C} \rightarrow \ca{C}_A$ is a symmetric strong monoidal functor: it is left adjoint to the forgetful functor $U \colon \ca{C}_A \rightarrow \ca{C}$. We often write this as functor as $(-)_A \colon \ca{C} \rightarrow \ca{C}_A$ and we call $M_A$ the \emph{base change} of $M \in \ca{C}$. Given a morphism $f\colon M \rightarrow N$ in $\ca{C}$ where $N$ is the underlying object of an $A$-module $(N,\alpha)$, the corresponding morphism of $A$-modules $M_A \rightarrow N$ is given by the composite
\[
 \xymatrix{A \otimes M \ar[r]^-{A \otimes f} & A \otimes N \ar[r]^-{\alpha} & N}
\]
 in $\ca{C}$.

 \subsection{Locally finitely presentable categories}
 An object $C$ of an $R$-linear category $\ca{C}$ is called \emph{finitely presentable} if the $R$-linear functor
\[
 \ca{C}(C,-) \colon \ca{C} \rightarrow \Mod_R
\]
 preserves filtered colimits. For example, the finitely presentable objects in $\Mod_R$ are precisely the finitely presentable $R$-modules. We denote the full subcategory of $\ca{C}$ consisting of the finitely presentable objects by $\ca{C}_{\fp}$.

 An $R$-linear category is called \emph{locally finitely presentable} (lfp for short) if it is cocomplete and there is a set $\ca{G} \subseteq \ca{C}_{\fp}$ of finitely presentable objects which is a strong generator (that is, $\ca{C}(G,f)$ is an isomorphism for all $G \in \ca{G}$ if and only if $f$ is an isomorphism). An important fact about lfp categories that we will frequently use is that filtered colimits in an lfp category are exact (that is, they commute with finite limits, see \cite[Korollar~7.12]{GABRIEL_ULMER}). In particular, an lfp abelian category is also a Grothendieck abelian category. There are many other technical conveniences of lfp categories. For example, a functor between lfp categories has a right adjoint if and only if it preserves all small colimits (such functors are called \emph{cocontinuous}).

 An \emph{lfp tensor category} is a tensor category which is also locally finitely presentable in a compatible way. More precisely, we ask that the unit object of an lfp tensor category is finitely presentable, and that $C \otimes C^{\prime}$ is locally presentable if both $C$ and $C^{\prime}$ are. Note that if $\ca{C}$ is an lfp tensor category and $A \in \ca{C}$ is a commutative algebra, then the category $\ca{C}_A$ of $A$-modules in $\ca{C}$ is an lfp tensor category as well. The images of the finitely presentable objects in $\ca{C}$ under the base change functor $(-)_A \colon \ca{C} \rightarrow \ca{C}_A$ give the desired generator of $\ca{C}_A$ consisting of finitely presentable objects. Moreover, $\ca{C}_A$ is abelian if $\ca{C}$ is since limits and colimits in $\ca{C}_A$ are computed as in $\ca{C}$.

 From the fact that filtered colimits commute with finite limits in the category $\Mod_R$ of $R$-modules it follows that $\ca{C}_{\fp}$ is always closed under finite colimits in $\ca{C}$. In particular, if $\ca{C}$ is an lfp tensor category, then $\ca{C}_{\fp}$ is a right exact symmetric monoidal category as defined in the introduction, that is, the functor $C \otimes - \colon \ca{C}_{\fp} \rightarrow \ca{C}_{\fp}$ is right exact. Conversely, if $\ca{A}$ is right exact symmetric monoidal, then its category $\Ind(\ca{A})$ of ind-objects (equivalently, the category $\Lex[\ca{A}^{\op},\Mod_R]$ of $R$-linear functors $\ca{A}^{\op} \rightarrow \Mod_R$ which send finite colimits to finite limits) is an lfp tensor category with the Day convolution tensor product (see \S \ref{section:day_convolution} below). However, not every tensor functor between lfp tensor categories arises from a right exact symmetric monoidal functors between the respective subcategories of finitely presentable objects. The reason is that a general tensor functor need not preserve finitely presentable objects.
\section{Quasi-coherent sheaves as a universal cocompletion}\label{section:universal}

\subsection{Overview}\label{section:universal:overview}

 In this section we prove a slight generalization of a recent result of Tonini about the category of quasi-coherent sheaves on an Adams stack. In order to state it, we need to introduce some notation. Given a full subcategory $\ca{A}$ of an $R$-linear category $\ca{C}$ and a set $\Sigma$ of colimit diagrams in $\ca{A}$, we write $\Lex_\Sigma[\ca{A}^{\op},\Mod_R]$ for the category of $R$-linear presheaves which send the colimit diagrams in $\Sigma$ to limit diagrams. If the colimits in $\Sigma$ are preserved by the inclusion $K \colon \ca{A} \rightarrow \ca{C}$ (that is, they consist of colimits in $\ca{C}$ which happen to lie in $\ca{A}$), then we can define a functor
\[
 \Hom_{\ca{A}}(K,-) \colon \ca{C} \rightarrow \Lex_{\Sigma}[\ca{A}^{\op},\Mod_R]
\]
 which sends an object $C \in \ca{C}$ to the presheaf $\ca{C}(K-,C)$.

 Now let $X$ be an Adams stack over $R$. We are mainly interested in the case where $\ca{C}=\QCoh(X)$ is the category of quasi-coherent sheaves on $X$ and $\ca{A}$ is a generator of $\ca{C}$. In \cite{TONINI}, Tonini proved several variants of the following theorem.

\begin{thm}[Tonini]\label{thm:tonini}
 Let $X$ be an Adams over $R$, and let $\ca{A} \subseteq \QCoh_{\fp}(X)$ be a subcategory of finitely presentable quasi-coherent sheaves on $X$ which generates $\QCoh(X)$. Let $\Sigma$ be the set of right exact sequences in $\QCoh(X)$ whose entries all lie in $\ca{A}$ and write $K \colon \ca{A} \rightarrow \QCoh(X)$ for the inclusion. Then
\[
 \Hom_{\ca{A}}(K,-) \colon \QCoh(X) \rightarrow \Lex_{\Sigma}[\ca{A}^{\op},\Mod_R]
\]
 is an equivalence of categories.
\end{thm}

 Up to a minor additional assumption on the generator $\ca{A}$ (closure under certain kernels, for example), this follows from \cite[Theorem~3.18 and Proposition~3.26]{TONINI}.

 The main case of interest for us is where the category $\ca{A}$ is the category $\VB^c(X)$ of vector bundles of constant rank. Note that Tonini proved much more general results, for stacks that need not be quasi-compact and for generators that need not consist of finitely presentable quasi-coherent sheaves. As we will see in \S \ref{section:tonini_new_proof} below, in the case of Adams stacks, his arguments can be slightly generalized. 

 If $\ca{A}$ contains the unit and is closed under taking tensor products, it is itself a symmetric monoidal category. For any such category $\ca{A}$, the category $\Prs{A}$ of enriched presheaves can be endowed with the Day convolution symmetric monoidal structure. If the monoidal structure interacts well with a given set of colimits $\Sigma$ in $\ca{A}$, then the monoidal structure on $\Prs{A}$ induces a symmetric monoidal structure on $\Lex_\Sigma[\ca{A}^{\op},\Mod_R]$ via Day reflection. Our goal in this section is to show that the equivalence of Theorem~\ref{thm:tonini} is compatible with this symmetric monoidal structure, and that $\Lex_{\Sigma}[\ca{A}^{\op},\Mod_R]$ has a universal property among all tensor categories over $R$.

\subsection{Day convolution and its universal property}\label{section:day_convolution}
 Let $(\ca{A},\otimes,\U)$ be a small symmetric monoidal $R$-linear category, and let $\Prs{A}$ denote the category of $R$-linear presheaves on $\ca{A}$. The \emph{Day convolution} symmetric monoidal structure
\[
 \bigl(\Prs{A},\ast,\ca{A}(-,\U)\bigr)
\]
 on $\Prs{A}$ is given by
\[
 (F \ast G)(A) \defl \int^{B,C} FB \otimes GC \otimes \ca{A}(A,B\otimes C)
\]
 and unit $\ca{A}(-,\U)$. The presheaf
\[
 [F,G](A) \defl \int_B [FB,G(A\otimes B)]
\]
 on $\ca{A}$ gives an internal hom for $\bigr(\Prs{A},\ast,\ca{A}(-,\U)\bigl)$, that is, we have isomorphisms $\Prs{A}(F\ast G,H) \cong \Prs{A}(F,[G,H])$, natural in $F$, $G$, and $H$.

 Now let $\Sigma$ be a set of colimit diagrams in $\ca{A}$. If there is no confusion possible, we will use the abbreviation $\Lex_\Sigma$ for $\Lex_\Sigma[\ca{A}^{\op},\Mod_R]$. The inclusion $\Lex_\Sigma \rightarrow \Prs{A}$ has a left adjoint $R \colon \Prs{A} \rightarrow \Lex_{\Sigma}$ by \cite[Theorem~6.11]{KELLY_BASIC}, that is, $\Lex_{\Sigma}$ is a reflective subcategory of $\Prs{A}$.

 Day's reflection theorem gives necessary and sufficient conditions such that a reflective subcategory of a symmetric monoidal category can be endowed with a compatible symmetric monoidal structre.

\begin{prop}\label{prop:day_reflection}
 Let $(\ca{A}, \otimes, \U )$ be a symmetric monoidal $R$-linear category, and let $\Sigma$ be a set of colimit diagrams in $\ca{A}$ which are preserved by tensoring with any fixed object in $\ca{A}$. Then $\Lex_\Sigma[\ca{A}^{\op},\Mod_R]$ has a unique symmetric monoidal structure such that the reflection $R \colon \Prs{A} \rightarrow \Lex_{\Sigma}$ is symmetric strong monoidal.
\end{prop}

\begin{proof}
 One of the equivalent conditions of Day's reflection theorem \cite{DAY_REFLECTION} is that, for all $F \in \Lex_\Sigma$ and all $G$ in a generating set of $\Prs{A}$, the presheaf $[F,G]\in \Prs{A}$ lies in $\Lex_\Sigma$. Thus we can take $F$ to be a representable presheaf $\ca{A}(-,C)$. The Yoneda lemma
\[
 [F,G]=\int_B [\ca{A}(B,C),G(- \otimes B)]\cong G(-\otimes C)
\]
 reduces the problem to cheking that $-\otimes C$ preserves colimits in $\Sigma$ for all $C \in \ca{A}$, which is precisely the assumption.
\end{proof}

 The Day convolution monoidal structure satisfies a universal property. In order to give its full statement, we need to recall the concept of left Kan extensions. This is a construction which allows one to extend a given $R$-linear functor $\ca{A} \rightarrow \ca{C}$ along another $R$-linear functor $K \colon \ca{A} \rightarrow \ca{B}$ (usually some kind of inclusion). It always exists as long as $\ca{A}$ is small and $\ca{C}$ has small colimits.

 The left Kan extension of $F$ along $K$ is denoted by $\Lan_K F \colon \ca{B}\rightarrow \ca{C}$, and it is characterized by the existence of bijections
\[
 \Nat(\Lan_K F,G) \cong \Nat(F,G\circ K)
\]
 which are natural in $G$.

\begin{example}\label{example:kan_ext_along_yoneda}
 We will frequently need to consider left Kan extensions long a Yoneda embedding $Y \colon \ca{A} \rightarrow \Prs{A}$. In this case, $\Lan_Y F$ is given by the ``functor tensor product''
\[
 \Lan_Y F(G)=F\ten{\ca{A}} G =\int^A FA \otimes GA \smash{\rlap{.}}
\]
 It has a right adjoint $\Hom_\ca{A}(F,-) \colon \ca{C} \rightarrow \Prs{A}$, which sends an object $C \in \ca{C}$ to the presheaf $\ca{C}(F-,C)$ on $\ca{A}$. In \cite{KELLY_BASIC}, the functor tensor product is denoted by $F \star G$, and its right adjoint is denoted by $\widetilde{F}$. In order to avoid confusion with the Day convolution monoidal structure we have decided to use the above notation instead.
\end{example}

 In order to state the universal property of the Day convolution monoidal structure, we introduce the following notation. Given symmetric monoidal $R$-linear categories $\ca{A}$ and $\ca{B}$, we write $\Fun_{\otimes} (\ca{A},\ca{B})$ for the category of symmetric strong monoidal $R$-linear functors and symmetric monoidal natural transformations between them. If $\Sigma$ is a class of colimit diagrams in $\ca{A}$, then we write $\Fun_{\Sigma,\otimes}(\ca{A},\ca{B})$ for the full subcategory of $\Fun_{\otimes}(\ca{A},\ca{B})$ consisting of symmetric strong monoidal functors which also preserve all colimits in $\Sigma$. We write $\Fun_{c,\otimes}(\ca{A},\ca{B})$ for the full subcategory of $\Fun_{\otimes}(\ca{A},\ca{B})$ consisting of symmetric strong monoidal functors whose underlying functor is a left adjoint.

\begin{thm}[Universal property of Day convolution]\label{thm:day_convolution_universal}
 Let $(\ca{A}, \otimes, \U)$ be a symmetric monoidal $R$-linear category. Then the Yoneda embedding gives a symmetric strong monoidal functor $(\ca{A},\otimes,\U) \rightarrow \bigl(\Prs{A},\ast,\ca{A}(-,\U)\bigr)$. If $\ca{C}$ is a tensor category over $R$ and $F \in \Fun_\otimes(\ca{A},\ca{C})$, then $\Lan_Y F$ is a tensor functor. Moreover, the functors
\[
 \Lan_Y \colon \Fun_{\otimes}(\ca{A},\ca{C}) \rightarrow \Fun_{c,\otimes}(\Prs{A},\ca{C})
\]
 and
\[ 
 (-)\circ Y \colon \Fun_{c,\otimes}(\Prs{A},\ca{C}) \rightarrow \Fun_{\otimes}(\ca{A},\ca{C})
\]
 (restriction along the Yoneda embedding) are mutually inverse equivalences.
\end{thm}

\begin{proof}
 This is part of \cite[Theorem~5.1]{IM_KELLY}.
\end{proof}

 The induced symmetric monoidal structure on $\Lex_\Sigma$ has a similar universal property. Note that the Yoneda embedding factors through $\Lex_\Sigma$ (since any representable functor sends colimits to limits). We denote the resulting embedding by $Z \colon \ca{A} \rightarrow \Lex_\Sigma$.

\begin{thm}[Universal property of {$\Lex_{\Sigma}[\ca{A}^{\op},\Mod_R]$}] \label{thm:lex_sigma_universal}
 Let $\ca{A}$ be a small symmetric monoidal $R$-linear category and let $\Sigma$ be a set of colimit diagrams in $\ca{A}$ which are all preserved by tensoring with objects in $\ca{A}$. If $\Lex_\Sigma$ is endowed with the symmetric monoidal closed structure of Proposition~\ref{prop:day_reflection}, then $Z \colon \ca{A} \rightarrow \Lex_\Sigma$ is a symmetric strong monoidal functor. If $\ca{C}$ is a tensor category over $R$ and $F \in \Fun_{\Sigma,\otimes}(\ca{A},\ca{C})$, then $\Lan_Z F \colon \Lex_\Sigma \rightarrow \ca{C}$ is a tensor functor. Moreover, the functors
\[
 \Lan_Z \colon \Fun_{\Sigma,\otimes}(\ca{A},\ca{C}) \rightarrow \Fun_{c,\otimes}(\Lex_{\Sigma},\ca{C})
\]
 and
\[ 
 (-)\circ Z \colon \Fun_{c,\otimes}(\Lex_\Sigma,\ca{C}) \rightarrow \Fun_{\Sigma,\otimes}(\ca{A},\ca{C})
\]
 are mutually inverse equivalences.
\end{thm}

 The proof of this can be reduced to some basic properties of Kan extensions. The key observation we need is that for any $R$-linear functor $F \colon \ca{A} \rightarrow \ca{C}$ which preserves colimits in $\Sigma$, the left Kan extension $\Lan_Z F \colon \Lex_{\Sigma} \rightarrow \ca{C}$ is a left adjoint, see Part~(v) of the lemma below. Since any object in $\Lex_{\Sigma}$ can be written as a colimit of representable presheaves in a canonical way, and colimits in $\ca{C}$ commute with tensor products, we get induced comparison morphisms between the objects obtained by first taking a tensor product in $\Lex_\Sigma$ and then applying the left adjoint $\Lan_Z$, and the object obtained by taking the tensor product in $\ca{C}$ of the images under $\Lan_Z$. So, the above theorem could in some sense be proved ``manually,'' that is, by explicitly constructing symmetric strong monoidal structures, and then checking that a plethora of diagrams actually commute. It is much more efficient to reduce the proof of Theorem~\ref{thm:lex_sigma_universal}---using various basic properties of left Kan extensions---to the case considered in Theorem~\ref{thm:day_convolution_universal}.

 The first fact which we use is that left Kan extensions are preserved by left adjoints, in the sense that $L \circ \Lan_Z F\cong \Lan_Z LF$ for any left adjoint $L$. This follows easily from the defining property of Kan extensions. We collect the necessary additional facts in the following lemma. In order to prove Theorem~\ref{thm:lex_sigma_universal}, we will apply it in the case where $\ca{E}=\Lex_{\Sigma}$.

\begin{lemma}\label{lemma:kan_extensions}
 Let $\ca{A}$ be a small $R$-linear category and let $\ca{E}$ be a reflective subcategory of $\Prs{A}$ such that the Yoneda embedding factors through $\ca{E}$. Write $I \colon \ca{E} \rightarrow \Prs{A}$ for the inclusion, $R \colon \Prs{A} \rightarrow \ca{E}$ for its reflection, and $Z \colon \ca{A} \rightarrow \ca{E}$ for the corestriction of the Yoneda embedding. Let $F \colon \ca{A} \rightarrow \ca{C}$ be an $R$-linear functor whose codomain is cocomplete. Then the following hold.
\begin{enumerate}
 \item[(i)] There is an isomorphism $\Lan_Z F \circ Z \cong F$ natural in $F$;
 \item[(ii)] There is an isomorphism $R \cong \Lan_Y Z$;
 \item[(iii)] There is an isomorphism $\Lan_Z F \cong \Lan_Y F \circ I$ natural in $F$;
 \item[(iv)] If $L \colon \ca{E} \rightarrow \ca{C}$ is a left adjoint, then $LR \cong \Lan_Y LZ$.
\item[(v)] Let $\ca{E}=\Lex_{\Sigma}$ for some set $\Sigma$ of colimit diagrams in $\ca{A}$. If $F$ preserves the colimits in $\Sigma$, then $\Lan_Z F \colon \Lex_\Sigma \rightarrow \ca{C}$ has a right adjoint. Moreover, $\Lan_Y F$ sends  the unit $\eta \colon \id \rightarrow IR$ of the adjunction $R \dashv I$ to an isomorphism. 
\end{enumerate}
\end{lemma}

\begin{proof}
 Part~(i) is true for left Kan extensions along any fully faithful functor such as $Z$ (see \cite[Proposition~4.23]{KELLY_BASIC}). To get the isomorphism in Part~(ii), note that the left adjoint $R$ preserves Kan extensions, so we have $R \circ \Lan_Y Y \cong \Lan_Y RY \cong \Lan_Y Z$. Thus it suffices to prove that $\Lan_Y Y$ is isomorphic to the identity. This is a consequence of the Yoneda lemma, see \cite[Formula~(4.30) and Theorem~5.1]{KELLY_BASIC}.

 Part~(iii) is an immediate consequence of the last part of \cite[Theorem~4.47]{KELLY_BASIC} since $I$ is fully faithful. To see that Part~(iv) holds, note that $R\cong \Lan_Y Z$ by Part~(ii). Since left adjoints preserve left Kan extensions, we get the desired isomorphism $LR\cong L \Lan_Y Z \cong \Lan_Y LZ$.

 It remains to check that Part~(v) holds. To see that $\Lan_Z F$ has a right adjoint, note that $\Lan_Z F \cong \Lan_Y F \circ I$ by Part~(iii), that is, $\Lan_Z F$ is simply the restriction of $\Lan_Y F$ to $\Lex_\Sigma$. But $\Lan_Y F =F\ten{\ca{A}}-$ has a right adjoint $\Hom_{\ca{A}}(F,-)$, which sends $C \in \ca{C}$ to the presheaf $\ca{C}(F-,C)$. To see that its restriction to $\Lex_\Sigma$ has a right adjoint, it suffices to check that the $\Hom_{\ca{A}}(F,-)$ factors through $\Lex_{\Sigma}$. In other words, we want to show that the presheaf $\ca{C}(F-,C)$ sends colimits in $\Sigma$ to limits for ever $C \in \ca{C}$. This follows immediately from the fact that $F$ preserves colimits in $\Sigma$.

 To see that $\Lan_Y F$ sends the unit $\eta \colon \id \rightarrow IR$ to an isomorphism, note that we have isomorphisms
\[
 \Lan_Z F \circ R \cong \Lan_Z F \circ \Lan_Y Z \cong \Lan_Y (\Lan_Z F \circ Z) \cong \Lan_Y F
\]
 by (ii), by (iv) applied to the left adjoint $\Lan_Z F$, and by (i) respectively. Thus it suffices to check that $R\eta$ is an isomorphism. But this follows from one of the triangle identities since the counit of the adjunction $R \dashv I$ is an isomorphism.
\end{proof}

\begin{proof}[Proof of Theorem~\ref{thm:lex_sigma_universal}]
 We first show that $\Lan_Z F$ is a symmetric strong monoidal left adjoint for any $F \in \Fun_{\Sigma, \otimes}(\ca{A},\ca{C})$. From Part~(v) of Lemma~\ref{lemma:kan_extensions} we know that it is a left adjoint. Moreover, by Part~(iii) of Lemma~\ref{lemma:kan_extensions}, it is isomorphic to $\Lan_Y F \circ I$. The functor $\Lan_Y F$ is symmetric strong monoidal by Theorem~\ref{thm:day_convolution_universal}, and the inclusion 
\[
I \colon \Lex_\Sigma \rightarrow \Prs{A} 
\]
 is symmetric \emph{lax} monoidal by construction of the Day reflection tensor product. The structure morphism in question is given by $\eta_{F\ast G} \colon IF \ast IG \rightarrow IR(IF\ast IG)$. Here we use that the tensor product on $\Lex_\Sigma$ is given by $F \mathop{\widehat{\ast}} G=R(IF\ast IG)$ (see \cite{DAY_REFLECTION}). It follows that the composite $\Lan_Y F \circ I$ is certainly lax monoidal again, and to check that it is strong monoidal it suffices to check that $\Lan_Y F$ sends the unit $\eta \colon \id \rightarrow IR$ to an isomorphism. This is proved in Part~(v) of Lemma~\ref{lemma:kan_extensions}. Thus we can endow $\Lan_Z F$ with a unique symmetric strong monoidal structure such that the isomorphism $\Lan_Z F \cong \Lan_Y F \circ I$ is symmetric monoidal.

 To get the desired equivalence between $\Fun_{\Sigma,\otimes}(\ca{A},\ca{C})$ and $\Fun_{c,\otimes}(\Lex_\Sigma,\ca{C})$, it remains to check that the two composites are naturally isomorphic to the identity functors. We have isomorphisms
\[
 \Lan_Z F \circ Z \cong \Lan_Y F \circ I \circ Z =\Lan_Y F \circ Y \cong F \smash{\rlap{,}}
\]
 where the last isomorphism is implicit in Theorem~\ref{thm:day_convolution_universal}. It is in particular symmetric monoidal. The first is symmetric monoidal by definition of the symmetric strong monoidal structure on $\Lan_Z F$ (see above). The equality $IZ=Y$ is symmetric monoidal since the unit of the reflection of a representable functor can be taken to be an equality. Thus we do indeed get an isomorphism $\Lan_Z F \circ Z \cong F$ in $\Fun_{\Sigma,\otimes}(\ca{A},\ca{C})$. It is easy to check that this isomorphism is natural in $F$.

 To see the other isomorphism, let $L \in \Fun_{c,\otimes}(\Lex_{\Sigma},\ca{C})$ be a symmetric strong monoidal left adjoint. Then we have the chain of isomorphisms
\begin{align*}
 \Lan_Z LZ & \cong \Lan_Y LZ \circ I \\ 
&\cong L \circ \Lan_Y Z \circ I \\
&\cong L \circ R \circ I \\
&\cong L \smash{\rlap{,}}
\end{align*}
 where the last isomorphism is given by the counit of the adjunction $R \dashv I$. It is straightforward to see that these are all natural in $L$. The first is symmetric monoidal by definition of the symmetric strong monoidal structure on $\Lan_Z F$ (see above), and the last is symmetric monoidal since the counit of the adjunction $R \dashv I$ is symmetric monoidal by definition. That the remaining isomorphisms are symmetric monoidal follows from Theorem~\ref{thm:day_convolution_universal}.
\end{proof}

\begin{rmk}
 Both Proposition~\ref{prop:day_reflection} and Theorem~\ref{thm:lex_sigma_universal} are true for categories enriched in an arbitrary cosmos\footnote{A \emph{cosmos} is a complete and cocomplete symmetric monoidal closed category.} $\ca{V}$. In fact, the proof of Theorem~\ref{thm:lex_sigma_universal} works verbatim at this level of generality (using the fact that \cite[Theorem~5.1]{IM_KELLY} is proved at the same level of generality).
\end{rmk}

\subsection{Locally finitely presentable abelian categories} \label{section:tonini_new_proof}

 In this section we will prove the version of Tonini's result stated in Theorem~\ref{thm:tonini}. In the special case we are interested in (that is, in the case were the stacks are quasi-compact and the generators consist of finitely presentable objects), our version of Tonini's theorem is a special case of the following result about arbitrary lfp abelian categories.

\begin{thm}\label{thm:lfp_abelian}
 Let $\ca{C}$ be a locally finitely presentable $R$-linear abelian category and let $\ca{A}$ be a generator of $\ca{C}$ consisting of finitely presentable objects that is closed under finite direct sums. Let $\Sigma$ be the set of sequences
\[
 \xymatrix{A_0 \ar[r] & A_1 \ar[r]& A_2 \ar[r] & 0}
\]
 in $\ca{A}$ which are right exact in $\ca{C}$. Let $K \colon \ca{A} \rightarrow \ca{C}$ denote the inclusion. Then
\[
 \Hom_{\ca{A}}(K,-) \colon \ca{C} \rightarrow \Lex_{\Sigma}[\ca{A}^{\op},\Mod_R] 
\]
and
\[
 \Lan_Y K \vert_{\Lex_{\Sigma}} \colon \Lex_{\Sigma}[\ca{A}^{\op},\Mod_R] \rightarrow \ca{C} 
\]
 are mutually inverse equivalences of categories.
\end{thm}

 Before giving its proof, we show how it implies Tonini's theorem.

\begin{proof}[Proof of Theorem~\ref{thm:tonini}]
 The category $\QCoh(X)$ is lfp and abelian: this follows for example from the fact that it is equivalent to the category of comodules of an Adams Hopf algebroid and \cite[Proposition~1.4.1]{HOVEY}.
\end{proof}

 In the case where the subcategory of finitely presentable objects happens to be abelian, we could give a proof of Theorem~\ref{thm:lfp_abelian} using \cite{BRANDENBURG_THESIS}[Propositions~2.5.1 and 3.1.21] using some well-known facts about categories of ind-objects. However, since we will need this in the context of Adams stacks which might not be coherent, we will need the full generality as stated above.

 Our proof of Theorem~\ref{thm:lfp_abelian} follows a similar proof given in \cite[\S 3.1]{BHATT}. In order to state the first lemma, we need to recall the following definition and result from \cite{SCHAEPPI_INDABELIAN}. Given a class of right exact sequences
\[
 \xymatrix{A \ar[r]^{p} & B \ar@{->>}[r]^{q} & C \ar[r] & 0 }
\]
 in an $R$-linear category $\ca{A}$, we write $R(\Sigma)$ for the set of all morphisms in $\ca{A}$ which are isomorphic to a morphism $q$ in one of the right exact sequences in $\Sigma$.

\begin{dfn}\label{dfn:ind_class}
 Let $\Sigma$ be a set of right exact sequences in an $R$-linear category $\ca{A}$. Then $\Sigma$ is called an \emph{ind-class} if the following two properties hold:
\begin{enumerate}
 \item[(i)] For any $q \colon B \rightarrow C$ in $R(\Sigma)$, the sequence
\[
 \xymatrix{B \ar@{->>}[r]^q & C \ar[r] & 0 \ar[r] & 0}
\]
 lies in $\Sigma$.

\item[(ii)]
 If 
\[
  \xymatrix{A \ar[r]^{p} & B \ar@{->>}[r]^{q} & C \ar[r] & 0 }
\]
 is a sequence in $\Sigma$ and $f \colon D \rightarrow B$ is a morphism with $qf=0$, then there exists a morphism $f^{\prime} \colon E \rightarrow A$ in $\ca{A}$ and a morphism $p^{\prime} \colon E \rightarrow D$ in $R(\Sigma)$ such that the diagram
\[
 \xymatrix{ E \ar@{->>}[r]^{p^{\prime}} \ar[d]_{f^{\prime}} & D \ar[d]^f \\ A \ar[r]_{p} & B }
\]
 is commutative.
\end{enumerate}
\end{dfn}

\begin{prop}[{\cite[Proposition~2.9]{SCHAEPPI_INDABELIAN}}]\label{prop:ind_class_implies_abelian}
 Let $\ca{A}$ be an $R$-linear category with finite direct sums, and let $\Sigma$ be an ind-class in $\ca{A}$. Then $\Lex_{\Sigma}$ is a locally finitely presentable abelian category. Moreover, for any finitely presentable object $F \in \Lex_{\Sigma}$ there exists a morphism $f \colon A \rightarrow A^{\prime}$ in $\ca{A}$ and a right exact sequence
\[
 \xymatrix{\ca{A}(-,A) \ar[r]^-{\ca{A}(-,f)} & \ca{A}(-,A^{\prime}) \ar[r] & F \ar[r] & 0 }
\]
 in $\Lex_{\Sigma}$.
\end{prop}

\begin{proof}
 The first claim is proved in Part~(i) of \cite[Proposition~2.9]{SCHAEPPI_INDABELIAN}. The second claim follows from Part~(iii) of that proposition and the fact that $\ca{A}$ has finite direct sums.
\end{proof}

\begin{lemma}\label{lemma:lex_abelian}
 In the situation of Theorem~\ref{thm:lfp_abelian}, $\Sigma$ is an ind-class. In particular, $\Lex_{\Sigma}$ is abelian.
\end{lemma}

\begin{proof}
 Recall that $\Sigma$ is the set of left exact sequences in $\ca{C}$ which happen to lie in the subcategory $\ca{A}$. Part~(i) of the definition of ind-class is therefore immediate.

 To see Part~(ii), note that we can always find a finitely presentable object $E_0 \in \ca{C}$, together with a morphism $f_0 \colon E_0 \rightarrow A$ and an epimorphism $p_0 \colon E_0 \rightarrow D$ such that the diagram
\[
 \xymatrix{E_0 \ar[d]_{f_0} \ar@{->>}[r]^{p_0} & D \ar[d]^{f} \\
A \ar[r]_{p} & B}
\]
 is commutative (this follows for example from the proof of \cite[Proposition~2.2]{SCHAEPPI_INDABELIAN}). Since $\ca{A}$ is a generator and $E_0$ is finitely presentable, we can write $E_0$ as a finite colimit of objects of $\ca{A}$. Closure of $\ca{A}$ under finite direct sums therefore implies that we can find an object $E \in \ca{A}$ together with an epimorphism $p_1 \colon E \rightarrow E_0$. We claim that the composite $p=p_0 p_1 \colon E \rightarrow D$ actually lies in $R(\Sigma)$, that is, that there exists a right exact sequence
\[
 \xymatrix{A \ar[r] & E \ar@{->>}[r]^{p} & D \ar[r] & 0 }
\]
 in $\ca{C}$ such that $A \in \ca{A}$.

 To see this, simply note that the kernel of an epimorphism between finitely presentable objects is finitely generated (that is, an epimorphic quotient of a finitely presentable object), see for example \cite[Lemma~2.1]{SCHAEPPI_INDABELIAN}. Since the kernel is a quotient of a finitely presentable object, which in turn is a quotient of an object $A \in \ca{A}$ by the above argument, we do get the desired right exact sequence in $\Sigma$ exhibiting $p$ as an element of $R(\Sigma)$.
\end{proof}

\begin{lemma}\label{lemma:epis_in_lfp_abelian}
 Let $\ca{A} \subseteq \ca{C}$ be as in Theorem~\ref{thm:lfp_abelian} and let $p \colon C^{\prime} \rightarrow C$ be an epimorphism in $\ca{C}$. Then for any morphism $f \colon A \rightarrow C$ with $A \in \ca{A}$ there exists an epimorphism $p^{\prime} \colon A^{\prime} \rightarrow A$ in $\ca{C}$ with $A^{\prime} \in \ca{A}$ and a morphism $A^{\prime} \rightarrow C^{\prime}$ such that the diagram
\[
 \xymatrix{A^{\prime} \ar[r] \ar@{->>}[d]_{p^{\prime}} & C^{\prime} \ar@{->>}[d]^{p} \\
A \ar[r]^-{f} & C}
\]
 is commutative.
\end{lemma}

\begin{proof}
 This fact was used in the proof of \cite[Proposition~5.8]{SCHAEPPI_GEOMETRIC}. The following argument is taken from the proof given there. 

 The pullback $P \rightarrow A$ is an epimorphism since $\ca{C}$ is abelian. The assumption that $\ca{A}$ is a generator implies that we can find an epimorphism
\[
 \textstyle\bigoplus\nolimits_{i \in I} A_i \rightarrow P
\]
 with $A_i \in \ca{A}$ (and $I$ a possibly infinite set). We can write $\bigoplus_{i \in I} A_i$ as filtered colimit of the objects $\bigoplus_{j \in J} A_j \in \ca{A}$ where $J \subseteq I$ is finite. It suffices to check that the composite
\[
 \textstyle\bigoplus\nolimits_{j \in J} A_j \rightarrow P \rightarrow A
\]
 is still an epimorphism for some finite $J \subseteq I$. But since filtered colimits in $\ca{C}$ are exact, $A$ is the directed union of the images of these morphisms. The assumption that $A$ is finitely presentable therefore implies that the identity of $A$ factors through one of these inclusions, hence that the corresponding morphism $\bigoplus_{j \in J} A_j \rightarrow P$ is an epimorphism.
\end{proof}

\begin{lemma}\label{lemma:hom_cocontinuous}
 In the situation of Theorem~\ref{thm:lfp_abelian}, the functor
\[
 \Hom_{\ca{A}}(K,-) \colon \ca{C} \rightarrow \Lex_{\Sigma}[\ca{A}^{\op},\Mod_R]
\]
 is cocontinuous.
\end{lemma}

\begin{proof}
 It suffices to show that $\Hom_{\ca{A}}(K,-)$ preserves filtered colimits and finite colimits. Since the class $\Sigma$ consists of finite diagrams, filtered colimits in $\Lex_{\Sigma}$ are computed as in the presheaf category $\Prs{A}$. To show the first claim it therefore suffices to check that $\ca{C}(A,-)$ preserves filtered colimits for each $A \in \ca{A}$, which follows from the assumption that $\ca{A} \subseteq \ca{C}_{\fp}$.

 It remains to check that $\Hom_{\ca{A}}(K,-)$ is right exact. Since it is a right adjoint, we know that it is left exact. Both the domain and the codomain $\Lex_{\Sigma}$ of $\Hom_{\ca{A}}(K,-)$ are abelian (see Lemma~\ref{lemma:lex_abelian}), so it suffices to check that $\Hom_{\ca{A}}(K,-)$ preserves epimorphisms.

 Fix an epimorphism $p \colon C^{\prime} \rightarrow C$ in $\ca{C}$. From the Yoneda Lemma we know that
\[
 \textstyle\bigoplus\nolimits_{f \colon A \rightarrow C} \ca{A}(-,A) \rightarrow \ca{C}(K-,C)=\Hom_{\ca{A}}(K,C)
\]
 is an epimorphism in $\Prs{A}$ (and therefore also in $\Lex_{\Sigma}$). For each such morphism $f$ we can find an object $A_f \in \ca{A}$, an epimorphism $A_f \rightarrow A$ in $\ca{C}$, and a morphism $A_f \rightarrow C^{\prime}$ such that the diagram
\[
 \xymatrix{A_f \ar[r] \ar@{->>}[d] & C^{\prime} \ar@{->>}[d]^{p} \\ A \ar[r]^{f} & C}
\]
 is commutative (see Lemma~\ref{lemma:epis_in_lfp_abelian}). From the definition of $\Sigma$ it follows that the morphism $\ca{A}(-,A_f) \rightarrow \ca{A}(-,A)$ is an epimorphism in $\Lex_{\Sigma}$ for each $f \colon A \rightarrow C$. Since we just observed that the morphisms
\[
\ca{C}(K-, f) \colon \ca{A}(-,A) \rightarrow \ca{C}(K-,C) 
\]
 are jointly epimorphic in $\Lex_{\Sigma}$ it follows that
\[
 \ca{C}(K-, p) \colon \ca{C}(K-,C^{\prime}) \rightarrow \ca{C}(K-,C)
\]
 is an epimorphism in $\Lex_{\Sigma}$, as claimed.
\end{proof}

 With these ingredients in place we can now prove Theorem~\ref{thm:lfp_abelian} (following the outline given above).

\begin{proof}[Proof of Theorem~\ref{thm:lfp_abelian}]
 By Lemma~\ref{lemma:hom_cocontinuous}, the right adjoint
\[
 \Hom_{\ca{A}}(K,-) \colon \ca{C} \rightarrow \Lex_{\Sigma}
\]
 of $\Lan_Y K \vert_{\Lex_{\Sigma}}$ is cocontinuous. To check that the unit and counit of this adjunction are isomorphisms, it therefore suffices to check that their components at a generating set of the respective categories are isomorphisms. From the Yoneda lemma it follows that the natural transformation
\[
 K_{-,A} \colon \ca{A}(-,A) \rightarrow \Hom_{\ca{A}}(K-,KA)
\]
 exhibits $KA$ as $\Lan_Y K\bigl(\ca{A}(-,A)\bigr)$. Thus the component of the unit of the above adjunction at each element of the generating set $\ca{A}(-,A)$, $A \in \ca{A}$ of $\Lex_{\Sigma}$ is indeed an isomorphism.

 To conclude the proof it suffices to check that the counit
\[
\varepsilon_A \colon \Lan_Y K \bigl( \Hom_{\ca{A}}(K-,A) \bigr) \rightarrow A 
\]
 is an isomorphism for all $A \in \ca{A}$. Using the triangle identities, this will follow if the objects $A$ lie in the essential image of $\Lan_Y K \vert_{\Lex_{\Sigma}}$. This, in turn, follows from the isomorphism $\Lan_Y K Y \cong K$ (see Part~(i) of Lemma~\ref{lemma:kan_extensions}).
\end{proof}

 Note that Theorem~\ref{thm:lfp_abelian} also recovers a special case of a result of Day and Street: in \cite[Theorem~2 and Example~(3)]{DAY_STREET_GENERATORS}, they showed that every generator $\ca{A}$ of a Grothendieck abelian category $\ca{C}$ is dense; meaning that the canonical functor identifies $\ca{C}$ with a reflective subcategory of the presheaf category $\Prs{A}$. Our result only shows that many generators of the (smaller) class of lfp abelian tensor categories are dense. The advantage of Theorem~\ref{thm:lfp_abelian} is that it explicitly identifies the reflective subcategory of $\Prs{A}$ that is equivalent to $\ca{C}$.

 Combining the two main results of this section we obtain the following corollary.

\begin{cor}\label{cor:C_universal_tensor}
 Let $\ca{C}$ be an lfp abelian tensor category over $R$. Let $\ca{A}$ be a generator of $\ca{C}$ which consists of finitely presentable objects and which is closed under finite direct sums and finite tensor products. Let $\Sigma$ be the set of right exact sequences in $\ca{C}$ which lie in $\ca{A}$, and write $K \colon \ca{A} \rightarrow \ca{C}$ for the inclusion. 

 Then the restriction functor
\[
 (-)\circ K \colon \Fun_{c,\otimes}(\ca{C},\ca{D}) \rightarrow \Fun_{\Sigma,\otimes}(\ca{A},\ca{D})
\]
 is an equivalence for any tensor category $\ca{D}$ over $R$.
\end{cor}

\begin{proof}
 By Theorem~\ref{thm:lex_sigma_universal}, it suffices to show that $\ca{C}$ is equivalent to $\Lex_{\Sigma}$ as a symmetric monoidal category. From Theorem~\ref{thm:lfp_abelian} we know that the restriction of $\Lan_Y K \colon \Prs{A} \rightarrow \ca{C}$ to $\Lex_{\Sigma}$ is an equivalence. But this restriction is precisely the left Kan extension of $K$ along the corestricted Yoneda embedding $Z \colon \ca{A} \rightarrow \Lex_\Sigma$ by Part~(iii) of Lemma~\ref{lemma:kan_extensions}. Since $K$ is symmetric strong monoidal, this Kan extension is symmetric strong monoidal as well by Theorem~\ref{thm:lex_sigma_universal}. Thus we do indeed have an equivalence $\ca{C} \simeq \Lex_{\Sigma}$ of symmetric monoidal $R$-linear categories. By construction, it is compatible with $K$ and the corestricted Yoneda embedding $Z$ (see Part~(i) of Lemma~\ref{lemma:kan_extensions}), hence restriction along $K$ does give the desired equivalence.
\end{proof}

\subsection{Vector bundles on Adams stacks}

 We are particularly interested in applying Corollary~\ref{cor:C_universal_tensor} to the case where $\ca{C}$ is the category $\QCoh(X)$ of quasi-coherent sheaves on an Adams stack $X$, and the objects of $\ca{A}$ are vector bundles (that is, objects with duals). The following result is well-known for Artin stacks. Since we are dealing with Adams stacks, we provide a short proof that works at this level of generality.

\begin{prop}\label{prop:vb_c_generator}
 Let $X$ be an Adams stack over $R$, and let $V \in \QCoh(X)$ be a vector bundle on $X$. Then there exists a vector bundle $W$ on $X$ such that $V \oplus W$ has constant rank $d$ for some $d \in \mathbb{N}$.
\end{prop}

\begin{proof}
 Fix a faithfully flat affine covering $f \colon \Spec(B) \rightarrow X$. Then $f^{\ast} (V)$ is a finitely generated projective $B$-module. It's rank function is locally constant, so we can cover $\Spec(B)$ by finitely many affine open subsets $U_i=\Spec(A_i)$ such that the pullback of $V$ to $U_i$ is free of rank $d_i \in \mathbb{N}$, and such that $d_1 > d_2 > \ldots d_k\geq 0$. Replacing $B$ with this covering if necessary, we can in fact assume that $f$ is a morphism $\Spec(\prod_{i=1}^k A_i) \rightarrow X$. We thus have $f^{\ast}(V) \cong (A_1^{d_1}, \ldots, A_k^{d_k})$.

 Now consider the exterior power $\Lambda^{d_1} V$, that is, the image of the operator
\[
 \sum_{\sigma \in \Sigma_{d_1}} \sgn(\sigma) \sigma \colon V^{\otimes d_1} \rightarrow V^{\otimes d_1} 
\]
 on the $d_1$-fold tensor product of $V$ with itself. Since each summand of $f^{\ast} (V)$ is free (and $f^{\ast}$ is exact), this pullback of $\Lambda^{d_1} V$ coincides with the usual exterior power of $f^{\ast} (V)$. Here we are using the fact that---for free modules over \emph{any} commutative ring---the exterior power can be computed as the image of the above ``alternator.'' Thus we have isomorphisms
\[
 f^{\ast} (\Lambda^{d_1}V) \cong (\Lambda^{d_1} A_1^{d_1}, \Lambda^{d_1} A_2^{d_2}, \ldots, \Lambda^{d_1} A_k^{d_k})\cong (A_1,0,\ldots,0)
\]
 of $\prod_{i=1}^k A_i$-modules.

 Now consider the kernel $W_1$ of the coevaluation $\U_X \rightarrow (\Lambda^{d_1} V)^{\vee} \otimes \Lambda^{d_1} V$. Its image under the exact functor $f^{\ast}$ is the annihilator of $f^{\ast} \Lambda^{d_1}V \cong (A_1, 0,\ldots, 0)$. Thus $f^{\ast} W_1 \cong (0,A_2,\ldots,A_k)$ is finitely generated and projective, hence $W_1$ is a vector bundle. It also follows that the pullback of $V \oplus W_1^{\oplus d_1-d_2}$ to $\Spec(A_1 \times A_2)$ is free of rank $d_1$, and its pullback to $\Spec(A_i)$, $2 < i \leq k$ is free of rank strictly less than $d_1$. Iterating this procedure, we can construct vector bundles $W_i$ on $X$ with $f^{\ast} W_i \cong (0,\ldots,0, A_{i+1}, \ldots,A_k)$ for $i=2, \ldots k-1$. Using these as direct summands, we can build a vector bundle $W$ such that $f^{\ast}(V \oplus W)$ is free of rank $d_1$.
\end{proof}

 The following result will be crucial in establishing the universal properties of the various Adams stacks we construct in \S \ref{section:limits}. Conceptually, it shows that we can understand tensor functor out of $\QCoh(X)$---and therefore, by Tannaka duality, morphisms from algebraic stacks to $X$---if we have a good understanding of the category of vector bundles on $X$.

\begin{thm}\label{thm:adams_vb_universal}
 Let $X$ be an Adams stack over $R$. Let $\ca{A} \subseteq \VB(X)$ be a generator of the category $\QCoh(X)$ which is closed under finite direct sums and finite tensor products. Let $\Sigma$ be the set of right exact sequences in $\QCoh(X)$ which lie in $\ca{A}$. Then restriction along the inclusion $\ca{A} \rightarrow \QCoh(X)$ induces an equivalence
\[
 \Fun_{c,\otimes}\bigl(\QCoh(X),\ca{D}\bigr) \rightarrow \Fun_{\Sigma,\otimes}(\ca{A},\ca{D})
\]
 of categories for tensor categories $\ca{D}$ over $R$. Moreover, if $\ca{A}$ is closed under taking duals and kernels of locally split epimorphisms, and $\ca{D}$ is abelian, then a functor $F \in \Fun_{\otimes}(\ca{A}, \ca{D})$ preserves right exact sequences in $\Sigma$ if and only if it sends locally split epimorphisms to epimorphisms in $\ca{D}$.
\end{thm}

\begin{proof}
 The conditions of Corollary~\ref{cor:C_universal_tensor} are satisfied, so the first half of the statement follows from that.

 It remains to check the claim about $F \in \Fun_{\otimes} (\ca{A},\ca{D})$. First note that the kernel $i \colon K \rightarrow A$ of a locally split epimorphism $p \colon A \rightarrow B$ is a locally split monomorphism. Its dual $i^{\vee} \colon A^{\vee} \rightarrow K^{\vee}$ is therefore a locally split epimorphism. The assumption on $F$ implies that $F(i^{\vee})\cong (Fi)^{\vee}$ is an epimorphism. Its dual
\[
 [(Fi)^{\vee},\U] \colon [(FA)^{\vee},\U] \rightarrow [(FK)^{\vee},\U]
\]
  (where $\U$ denotes the unit of $\ca{C}$ and $[-,-]$ the internal hom) is therefore a monomorphism. Indeed, the contravariant hom-functor $[-,\U]$ turns epimorphisms into monomorphisms since $\ca{C}(A,[-,\U]) \cong \ca{C}(-,[A,\U])$. But the dual of $(Fi)^{\vee}$ is isomorphic to $Fi$. Since any epimorphism in $\QCoh(X)$ whose target is a vector bundle is locally split, any right exact sequence in $\Sigma$ factors as a sequence
\[
 \xymatrix{A_0 \ar@{->>}[r] & K \ar[r]^{i} & A_1 \ar@{->>}[r] & A_2 \ar[r] & 0}
\]
 where all the unlabeled arrows are locally split epimorphisms and $i$ is a locally split monomorphism. By assumption, the object $K$ lies in $\ca{A}$, so the above argument shows that $Fi$ is a monomorphism and the images of all the other morphisms in the above sequence are epimorphisms by assumption. Since $\ca{D}$ is abelian it follows that $F$ preserves the right exact sequence in question.

 Conversely, for any locally split epimorphism $p \colon A \rightarrow B$, the sequence
\[
 \xymatrix{K \ar[r] & A \ar@{->>}[r]^{p} & B \ar[r] & 0}
\]
 where $K$ denotes the kernel of $p$ lies in $\Sigma$. Thus, if $F$ preserves right exact sequences in $\Sigma$, it sends locally split epimorphisms to epimorphisms.
\end{proof}

 If we take $\ca{A}$ to be the category of all vector bundles of constant rank we get the following corollary.

\begin{cor}\label{cor:vb_density_presentation}
 Let $X$ be an Adams stack over $R$, and let $\VB^c(X)$ be the full subcategory of $\QCoh(X)$ of vector bundles of constant rank. Let $\ca{C}$ be an abelian tensor category. Let $\Fun_{\ell,\otimes}\bigl(\VB^c(X),\ca{C}\bigr)$ denote the full subcategory of $\Fun_{\otimes}\bigl(\VB^c(X),\ca{C}\bigr)$ which consists of functors which send locally split epimorphisms to epimorphisms. 

 Then restriction along the embedding $\VB^c(X) \rightarrow \QCoh(X)$ induces an equivalence
\[
 \Fun_{c,\otimes}\bigl(\QCoh(X),\ca{C}\bigr) \rightarrow \Fun_{\ell,\otimes}\bigl(\VB^c(X),\ca{C}\bigr)
\]
 of categories.
\end{cor}

\begin{proof}
 By Proposition~\ref{prop:vb_c_generator}, the category $\VB^c(X)$ is a generator of $\QCoh(X)$. Since the rank of the kernel of an epimorphism $p \colon A \rightarrow B$ is $\rk A - \rk B$, the rank of the kernel is constant if both $\rk A$ and $\rk B$ are constant. Thus $\VB_c(X)$ is a generator which satisfies the conditions of Theorem~\ref{thm:adams_vb_universal}, and the claim follows from its conclusion.
\end{proof}

\section{Torsors in general tensor categories}\label{section:torsors}

\subsection{Adams algebras}
 One of the difficulties we face in \S \ref{section:limits} is the construction of exact tensor functors. Our main source of such is the class of \emph{free module} or \emph{base change} functors
\[
 A\otimes - \colon \ca{C} \rightarrow \ca{C}_A
\]
 for faitfhfully flat commutative algebras. Under mild assumptions, there are characterizations of such algebras in terms of exact sequences and flat objects. The problem with this characterization is that it does not interact well with tensor functors, since they are in general not left exact. Thus it is not clear if tensor functors preserve faithfully flat algebras in general.

 The notion of an \emph{Adams algebra} which we are about to introduce does not suffer from this drawback. Implicitly, it lies at the heart of the proof of \cite[Theorem~1.3.3]{SCHAEPPI_STACKS} and much of \cite{SCHAEPPI_GEOMETRIC}. Since it also plays a key role in the proofs of Theorems~\ref{thm:universal_torsor} and \ref{thm:description}, we have decided to move this supporting character to center stage.

\begin{dfn}\label{dfn:adams_algebra}
 Let $(\ca{C},\otimes,\U)$ be an abelian tensor category with exact filtered colimits and let $(A,\mu,\eta)$ be a commutative algebra in $\ca{C}$. Let $f_{ij} \colon A_i \rightarrow A_j$ be a filtered diagram of objects with duals, together with morphisms $\eta_i, f_i$ making the triangles
\[
\vcenter{ \xymatrix@C=20pt{& \U \ar[rd]^{\eta} \ar[ld]_{\eta_i} \\ A_i \ar[rr]_{f_i} && A}}
 \quad \text{and} \quad
\vcenter{ \xymatrix@C=20pt{A_i \ar[rd]_{f_i} \ar[rr]^{f_{ij}} & & A_j \ar[ld]^{f_j} \\ & A}}
\]
 commutative. 

 We say that the morphisms $\eta_i \colon \U  \rightarrow A_i$ (and, implicitly, $f_i$ and $f_{ij}$) \emph{exhibit $A$ as Adams algebra} if the cocone $f_i \colon A_i \rightarrow A$ exhibits $A$ as filtered colimit of the $A_i$, and the dual morphism $\eta_i^{\vee} \colon A_i^{\vee} \rightarrow \U^{\vee} \cong \U$ is an epimorphism for all $i$. An \emph{Adams algebra} in $\ca{C}$ is a commutative algebra $A$ such that there exist $(A_i,\eta_i)$ as above which exhibit $A$ as Adams algebra.
\end{dfn}

 If an algebra $A$ is already known to be faithfully flat, then there is a simpler criterion for $A$ to be an Adams algebra, at least under some mild assumptions on the tensor category $\ca{C}$. This allows us to give lots of examples of Adams algebras. The following argument can essentially be found in the proof of \cite[Theorem~1.3.2]{SCHAEPPI_STACKS}.

\begin{prop}\label{prop:faithfully_flat_implies_adams_criterion}
 Let $(\ca{C},\otimes, \U)$ be an abelian tensor category such that $\U$ is finitely presentable. Then a faithfully flat algebra $(A,\mu,\eta)$ is Adams if and only if the underlying object $A \in \ca{C}$ can be written as filtered colimit of objects with duals.
\end{prop}

\begin{proof}
 The ``only if'' part is immediate. Thus let $f_i \colon A_i \rightarrow A$ exhibit $A$ as colimit of the filtered diagram $f_{ij} \colon A_i \rightarrow A_j$ of objects with duals. Let $\ca{I}$ be its indexing category. The assumption that $\U$ is finitely presentable implies that the unit $\eta \colon \U \rightarrow A$ factors through some $f_{i_0}$. Since $\ca{I}$ is filtered, the functor $i_0 \slash \ca{I} \rightarrow \ca{I}$ which sends the object $i_0 \rightarrow i$ to $i$ is cofinal. Thus $A$ is also the filtered colimit of the composite diagram $i_0 \slash \ca{I} \rightarrow \ca{I} \rightarrow \ca{C}$. This new diagram has an initial object, and $\eta$ factors through the initial vertex by construction. Thus we have obtained a new filtered diagram $i \mapsto A_i$, together with $\eta_i \colon \U \rightarrow A_i$ such that the diagrams in Definition~\ref{dfn:adams_algebra} commute.

 To conclude the proof, we only need to show that $\eta_i^{\vee}$ is an epimorphism. Since $A$ is faithfully flat, it suffices to check that $A \otimes \eta_i^{\vee}$ is an epimorphism. But $A\otimes \eta_i$ is a split monomorphism, with splitting $\mu \circ A \otimes f_i$. Thus its dual is a split epimorphism, so the $(A_i,\eta_i)$ do indeed exhibit $A$ as Adams algebra.
\end{proof}

 As a consequence we get the following class of examples, which explains our choice of name for these objects.

\begin{example}\label{example:adams_hopf_alg_is_adams_alg}
 If $(A,\Gamma)$ is a flat commutative Hopf algebroid, then $\Gamma$, considered as a comodule over $(A,\Gamma)$, is an Adams algebra if and only if $(A,\Gamma)$ is an Adams Hopf algebroid. 

 Note that if $R=k$ is a field and $A=k$ (that is, $\Gamma$ is a commutative Hopf algebra in $\Vect_k$), then $\Gamma$, considered as a comodule over itself, is always an Adams algebra. This follows from the well-known fact that any comodule is the union of its finite dimensional sub-comodules.
\end{example}

 Combininig this with \cite[Theorem~1.3.1]{SCHAEPPI_STACKS}, we see that examples of Adams algebras abound in algebraic geometry.

\begin{prop}\label{prop:adams_affine_algebra}
 Let $X$ be an Adams stack and let $f\colon X_0 \rightarrow X$ be a faithfully flat morphism with affine domain. Then $f^{\ast} (\U_{X_0})$ is an Adams algebra in $\QCoh(X)$.
\end{prop}

\begin{proof}
 Since $f$ is faithfully flat, $X$ is the stack associated to the flat affine groupoid $(X_0,X_0 \pb{X} X_0)=\bigr(\Spec (A),\Spec(\Gamma)\bigl)$. Under the equivalence between the category of quasi-coherent sheaves and comodules of the corresponding Hopf algebroid, the algebra $f^{\ast}(\U_{X_0})$ corresponds to $\Gamma$, considered as a comodule over $(A,\Gamma)$. The claim follows since $\Gamma$ is a filtered colimit of objects with duals if the category of comodules is generated by objects with duals (see \cite[Theorem~1.3.1]{SCHAEPPI_STACKS}).
\end{proof}

 Adams algebras satisfy the following convenient stability properties.

\begin{prop}\label{prop:adams_stable}
 Let $F \colon \ca{C} \rightarrow \ca{D}$ be a tensor functor between abelian tensor categories with exact filtered colimits. Then $F$ preserves Adams algebras. Moreover, Adams algebras are closed under finite tensor products. 
\end{prop}

\begin{proof}
 If the $\eta_i \colon \U \rightarrow A_i$ exhibit $A$ as Adams algebra, then the $F \eta_i \colon F\U \rightarrow FA_i$ exhibit $FA$ as Adams algebra since tensor functors preserve filtered colimits, duals, and epimorphisms.
 
 If, in addition, the $\eta_j^{\prime} \colon \U \rightarrow A_j^{\prime} $ exhibit $A^{\prime}$ as Adams algebra, then the $\eta_i \otimes \eta_j^{\prime}$ exhibit $A\otimes A^{\prime}$ as Adams algebra.
\end{proof}

 It is perhaps less clear that Adams algebras are faithfully flat.

\begin{prop}\label{prop:adams_faithfully_flat}
 Let $\ca{C}$ be an abelian tensor category with exact filtered colimits. Then all Adams algebras in $\ca{C}$ are faithfully flat.
\end{prop}

\begin{proof}
 Let $\eta_i \colon \U \rightarrow A_i$ exhibit $A$ as Adams algebra. First note that $A \cong \colim A_i$ is flat as filtered colimit of objects with duals. To see that $A\otimes X \cong 0$ implies $X \cong 0$, note that we have natural isomorphisms in the commutative diagram
\[
 \xymatrix{ \U\otimes X \ar[r]^-{\cong} \ar[d]_{\eta_i \otimes X} & \U^{\vee \vee} \otimes X \ar[r]^-{\cong} \ar[d]^{\eta_i^{\vee\vee}\otimes X} & [\U^{\vee},X] \ar[d]^{[\eta_i^{\vee},X]} \\ A_i \otimes X \ar[r]^-{\cong} &  A_i^{\vee\vee} \otimes X \ar[r]^-{\cong} & [A_i^{\vee},X]}
\]
 where $[-,-]$ denotes the internal hom of $\ca{C}$. Since $\ca{C}(C,[-,X])\cong \ca{C}(-,[C,X])$, the internal hom functor sends the epimorphism $\eta_i^{\vee}$ to a monomorphism. Thus $\eta_i \otimes X$ is a monomorphism, and since filtered colimits are exact, it follows that $\eta \otimes X \colon X \rightarrow A \otimes X$ is a monomorphism. This shows that $A$ is indeed faithfully flat.
\end{proof}

 Using this, we can in fact extend the preservation result of Proposition~\ref{prop:adams_stable} to tensor categories which are not necessarily abelian.

\begin{dfn}\label{dfn:faithfully_flat}
 Let $\ca{C}$ be a (not necessarily abelian) tensor category. Then a commutative algebra $A$ in $\ca{C}$ is called \emph{faithfully flat} if $A\otimes -$ preserves all the finite limits that exist in $\ca{C}$, and $A\otimes -$ is \emph{conservative}: whenever $A \otimes f$ is an isomorphism, so is $f$.
\end{dfn}

 Using the fact that a morphism in an abelian category is an isomorphism if and only if it is both an epimorphism and a monomorphism, it is not hard to see that this coincides with the usual definition for abelian tensor categories. Note that this also coincides with the notion of faithfully flat algebra in \cite{BRANDENBURG_THESIS}[Definition~4.6.4]. One implication follows from \cite{BRANDENBURG_THESIS}[Remark~4.6.5], while the other follows since a conservative functor creates all the limits it preserves.

 In order to extend the preservation result of Proposition~\ref{prop:adams_stable}, we need to recall the concept of a split equalizer. It will also play a crucial role in the next section.
 
 \begin{dfn}\label{dfn:split_equalizer}
  A diagram
\[
\turnradius{5pt}
\xymatrix{
X \ar[r]^{s} & Y \ar@<0.5ex>[r]^u \ar@<-0.5ex>[r]_v \ar `u[l] `[l]+<0.4pt,10pt>_p [l]+<0pt,6pt> &  Z \ar `u[l] `[l]+<2.4pt,10pt>_{q} [l]+<2pt,6pt>\\
}
\]
 is called a \emph{split equalizer} if $us=vs$, $ps=\id_E$, $qu=\id$, and $sp=qv$.
 \end{dfn}

 The equations immediately imply that $s$ is an equalizer of $u$ and $v$. Thus split equalizers are examples of \emph{absolute} equalizers: they are preserved by any functor whatsoever (since the equations above are preserved).

\begin{lemma}\label{lemma:strong_adams}
 Let $\ca{C}$ be an abelian tensor category with exact filtered colimits, and let $\eta_i \colon \U \rightarrow A$ exhibit $A$ as Adams algebra. Then
\[
 \xymatrix{ A_i^{\vee} \otimes A_i^{\vee} \ar@<0.5ex>[r]^-{A_i^{\vee} \otimes \eta_i^{\vee}} \ar@<-0.5ex>[r]_-{\eta_i^{\vee} \otimes A_i^{\vee}} & A_i^{\vee} \ar[r]^-{\eta_i^{\vee}} & \U }
\]
 is a coequalizer diagram for all $i$ in the indexing category.
\end{lemma}

\begin{proof}
 Since $A$ is faithfully flat (see Proposition~\ref{prop:adams_faithfully_flat}), it suffices to show that the image of the above diagram under $A\otimes -$ is a coequalizer diagram in the category $\ca{C}_A$ of $A$-modules. But this diagram is isomorphic to the image of
\[
 \xymatrix{A \ar[r]^-{A \otimes \eta_i} & A \otimes A_i \ar@<0.5ex>[r]^-{ A \otimes  \eta_i \otimes A_i} \ar@<-0.5ex>[r]_-{A \otimes A_i \otimes \eta_i} & A \otimes A_i \otimes A_i }
\]
 under the internal hom functor $[-,A]_{\ca{C}_A} \colon \ca{C}_A^{\op} \rightarrow \ca{C}_A$. Thus it suffices to check that the above is a split equalizer. Let $\overline{f_i} \colon A\otimes A_i \rightarrow A$ be the morphism of $A$-modules which corresponds to the morphism $f_i \colon A_i \rightarrow A$ in $\ca{C}$, that is, $\overline{f_i}=\mu \circ A \otimes f_i$. Then $\overline{f_i} \colon A\otimes A_i \rightarrow A$ and $\overline{f_i}\otimes A_i \colon A\otimes A_i \otimes A_i \rightarrow A\otimes A_i $ exhibit the above diagram as split equalizer. 
\end{proof}

 The above result suggests a definition of Adams algebras in tensor categories that are not abelian.

 \begin{dfn}\label{dfn:general_adams_algebra}
 Let $(\ca{C},\otimes,\U)$ be a (not necessarily abelian) tensor category and let $(A,\mu,\eta)$ be a commutative algebra in $\ca{C}$. Let $f_{ij} \colon A_i \rightarrow A_j$ a filtered diagram of objects with duals, together with morphisms $\eta_i, f_i$ making the triangles
\[
\vcenter{ \xymatrix@C=20pt{& \U \ar[rd]^{\eta} \ar[ld]_{\eta_i} \\ A_i \ar[rr]_{f_i} && A}}
 \quad \text{and} \quad
\vcenter{ \xymatrix@C=20pt{A_i \ar[rd]_{f_i} \ar[rr]^{f_{ij}} & & A_j \ar[ld]^{f_j} \\ & A}}
\]
 commutative. 

 We say that the morphisms $\eta_i \colon \U  \rightarrow A_i$ (and, implicitly, $f_i$ and $f_{ij}$) \emph{exhibit $A$ as Adams algebra} if the cocone $f_i \colon A_i \rightarrow A$ exhibits $A$ as filtered colimit of the $A_i$, and the diagram
\[
  \xymatrix{ A_i^{\vee} \otimes A_i^{\vee} \ar@<0.5ex>[r]^-{A_i^{\vee} \otimes \eta_i^{\vee}} \ar@<-0.5ex>[r]_-{\eta_i^{\vee} \otimes A_i^{\vee}} & A_i^{\vee} \ar[r]^-{\eta_i^{\vee}} & \U }
\]
 is a coequalizer diagram for each $i$. An \emph{Adams algebra} in $\ca{C}$ is a commutative algebra such that there exist $(A_i,\eta_i)$ as above which exhibit $A$ as Adams algebra.
\end{dfn}

 Note that by Lemma~\ref{lemma:strong_adams}, a commutative algebra in an abelian tensor category with exact filtered colimits is Adams in the sense of Definition~\ref{dfn:adams_algebra} if and only if it is Adams in the sense of Definition~\ref{dfn:general_adams_algebra}.

 There are some similarities between Adams algebras and the special descent algebras introduced in \cite{BRANDENBURG_THESIS}[Definition~4.10.6]: they have analogous stability properties. However, the latter involves a condition that certain cokernels admit duals, which is often harder to check in practice (cf. Proposition~\ref{prop:locally_split_criteria} below). In an abelian category, every special descent algebra is an Adams algebra by Proposition~\ref{prop:faithfully_flat_implies_adams_criterion}. It is not clear if the converse holds.

 If $\ca{C}$ is a cocomplete $R$-linear category, we say that $\ca{C}$ has \emph{exact filtered colimits} if filtered colimits in $\ca{C}$ commute with finite limits (that is, with those finite limits that exist in $\ca{C}$). This is for example the case if $\ca{C}$ is lfp (see \cite[Korollar~7.12]{GABRIEL_ULMER}).

 \begin{prop}\label{prop:general_adams_faithfully_flat}
  Let $\ca{C}$ be a (not necessarily abelian) tensor category, and let $A \in \ca{C}$ be an Adams algebra in the sense of Definition~\ref{dfn:general_adams_algebra}. If filtered colimits in $\ca{C}$ are exact, then $A$ is faithfully flat (in the sense of Definition~\ref{dfn:faithfully_flat}).
 \end{prop}

\begin{proof}
 Clearly $A \cong \colim A_i$ is flat as filtered colimit of objects with duals. It remains to check that $f \colon X \rightarrow Y$ is an isomorphism if $A\otimes f$ is.

 To see this, we will show that both rows in the commutative diagram
\[
 \xymatrix@C=50pt{
 X \ar[d]_f \ar[r]^-{\eta \otimes X} & A \otimes X \ar@<0.5ex>[r]^-{A \otimes \eta \otimes X} \ar@<-0.5ex>[r]_-{\eta \otimes A \otimes X} \ar[d]_{A \otimes f} & A \otimes A \otimes X \ar[d]^{A \otimes A\otimes f} \\
 Y \ar[r]_-{\eta \otimes Y} & A \otimes Y \ar@<0.5ex>[r]^-{A \otimes \eta \otimes Y} \ar@<-0.5ex>[r]_-{\eta \otimes A \otimes Y} & A \otimes A \otimes Y }
\]
 are equalizer diagrams. This implies the claim: since $A\otimes f$ is an isomorphism, so is $A \otimes A\otimes f$, and therefore $f$ by the universal property of equalizers.

 Since filtered colimits in $\ca{C}$ are exact, it suffices to check that
\[
 \xymatrix@C=50pt{X \ar[r]^-{\eta_i \otimes X} & A_i \otimes X \ar@<0.5ex>[r]^-{A_i \otimes \eta_i \otimes X} \ar@<-0.5ex>[r]_-{\eta_i \otimes A_i \otimes X} & A_i \otimes A_i \otimes X}
\]
 is an equalizer diagram for all $i$ in the indexing category. But this diagram is isomorphic to
\[
 \xymatrix@C=50pt{[\U^{\vee},X] \ar[r]^-{[\eta_i^{\vee}, X]} & [A_i^{\vee},  X] \ar@<0.5ex>[r]^-{[\eta_i^{\vee} \otimes A_i^{\vee}, X]} \ar@<-0.5ex>[r]_-{[A_i^{\vee} \otimes \eta_i^{\vee} , X]} & [ A_i^{\vee} \otimes A_i^{\vee}, X ]}
\]
 since we have natural isomorphisms $M \otimes X \cong M^{\vee\vee} \otimes X \cong [M^{\vee},X]$ for any object $M$ with dual. This last diagram is obtained by applying the contravariant internal hom functor $[-,X]$ to the coequalizer diagram in Definition~\ref{dfn:general_adams_algebra}. The claim follows since $[-,X]$ sends colimits to limits.
\end{proof}

\begin{prop}\label{prop:adams_strongly_stable}
 Tensor functors preserve Adams algebras in the sense of Definition~\ref{dfn:general_adams_algebra}. In particular, if $F \colon \ca{C} \rightarrow \ca{D}$ is a tensor functor, $A \in \ca{C}$ is an Adams algebra, and filtered colimits in $\ca{D}$ are exact, then $FA \in \ca{D}$ is faithfully flat (in the sense of Definition~\ref{dfn:faithfully_flat}).
\end{prop}

\begin{proof}
 The second claim follows from the first and Proposition~\ref{prop:general_adams_faithfully_flat}. The first claim is immediate from the fact that tensor functors preserve duals and coequalizers.
\end{proof}

 Proposition~\ref{prop:adams_strongly_stable} gives us a powerful tool for constructing a faithfully flat algebra in a general tensor category: we just need to show that it is the image of an Adams algebra under a tensor functor. The important part is that no exactness properties are required of the tensor functor. We will use this in the next section to study torsors of Adams Hopf algebras in general tensor categories.

 The concept of an Adams algebra also gives us a dual way of detecting faithfully flat algebras: if we have a candidate of morphisms $\eta_i \colon \U \rightarrow A_i$ which might exhibit $A$ as Adams algebra, then it suffices to check that their image under some tensor functor $F$ exhibits $FA$ as an Adams algebra. This holds under fairly mild assumptions on a tensor functor involved: it need not be exact, faithful, or even conservative.

\begin{prop}\label{prop:adams_detection}
 Let $\ca{C}$, $\ca{D}$ be tensor categories, and let $A \in \ca{C}$ be a commutative algebra. Let $F \colon \ca{C} \rightarrow \ca{D}$ be a tensor functor which reflects cokernels. Then a collection of morphisms $\eta_i \colon \U \rightarrow A_i$ (together with $f_i \colon A_i \rightarrow A$ and $f_{ij} \colon A_i \rightarrow A_j$) as in Definition~\ref{dfn:general_adams_algebra} exhibit $A$ as Adams algebra if and only if the $F\eta_i$ exhibit $FA$ as Adams algebra in $\ca{D}$.

 If both $\ca{C}$ and $\ca{D}$ are abelian, the same is true for a tensor functor that merely reflects epimorphisms between objects with duals.
\end{prop}

\begin{proof}
 Since $F$ is a tensor functor, the assumption implies that
\[
  \xymatrix{ A_i^{\vee} \otimes A_i^{\vee} \ar@<0.5ex>[r]^-{A_i^{\vee} \otimes \eta_i^{\vee}} \ar@<-0.5ex>[r]_-{\eta_i^{\vee} \otimes A_i^{\vee}} & A_i^{\vee} \ar[r]^-{\eta_i^{\vee}} & \U }
\]
 is a coequalizer diagram in $\ca{C}$ if and only if
\[
  \xymatrix@C=40pt{ FA_i^{\vee} \otimes FA_i^{\vee} \ar@<0.5ex>[r]^-{FA_i^{\vee} \otimes F\eta_i^{\vee}} \ar@<-0.5ex>[r]_-{F\eta_i^{\vee} \otimes FA_i^{\vee}} & FA_i^{\vee} \ar[r]^-{F\eta_i^{\vee}} & \U }
\]
 is a coequalizer diagram in $\ca{D}$.

 For the second claim it suffices to check that $F(\eta_i^{\vee})$ is an epimorphism for all $i$ in the indexing category. But tensor functors preserve duals, so this is the case if and only if $(F\eta_{i})^{\vee}$ is an epimorphism for all $i$.
\end{proof}

\subsection{Torsors for Adams Hopf algebras}\label{section:torsors_subsection}

 Let $G$ be a flat affine group scheme over $R$, and let $H$ be its corresponding Hopf algebra in $\Mod_R$. It is well known that the stack $BG$ classifies $G$-torsors: morphisms $X \rightarrow BG$ correspond to $G$-torsors on $X$. If $BG$ is an Adams stack (that is, if $H$ is an Adams Hopf algebra), then such morphisms in turn correspond to tensor functors $\QCoh(BG) \rightarrow \QCoh(X)$. We can ask if the category $\QCoh(BG)$ (which is equivalent to the category $\Rep(G)$ of algebraic representations of $G$ and to the category $\Comod(H)$ of comodules of $H$) has a corresponding universal property for general tensor categories---not just the ones of the form $\QCoh(X)$. 

 In order to state this, we first need to define torsors at this level of generality. Recall that in any cocomplete $R$-linear category $\ca{C}$, we have an \emph{external} tensor product: the external tensor product $C\odot M$ for $C \in \ca{C}$ and $M \in \Mod_R$ is characterized by a natural isomorphism $\ca{C}(C\odot M,-) \cong \Mod_R\bigl(M,\ca{C}(C,-)\bigr)$.

\begin{dfn}\label{dfn:torsor}
 Let $\ca{C}$ be a tensor category, and let $H$ be a commutative Hopf algebra in $\Mod_R$. An \emph{$H$-torsor} in $\ca{C}$ consists of a faithfully flat commutative algebra $A \in \ca{C}$, together with a morphism $\tau \colon A \rightarrow A \odot H$ of algebras such that the diagrams
\[
 \vcenter{\xymatrix{A \ar[r]^-{\tau} \ar[d]_{\tau} & A \odot H \ar[d]^{\tau \odot H} \\ A \odot H \ar[r]_-{A \odot \delta} & A \odot H \odot H}}
\quad \text{and}\quad
\vcenter{\xymatrix{A \ar[r]^-{\tau} \ar@{=}[rd] & A \odot H \ar[d]^{A \odot \varepsilon} \\ & A }}
\]
 commute, and the composite 
\[
 \xymatrix{A \otimes A \ar[r]^-{A \otimes \tau} & A \otimes A \odot H \ar[r]^-{\mu \odot H} & A \odot H}
\]
 is an isomorphism. A morphism of $H$-torsors $(A,\tau) \rightarrow (A^{\prime},\tau^{\prime})$ is a morphism $\varphi \colon A \rightarrow A^{\prime}$ of algebras such that the diagram
\[
 \xymatrix{A \ar[r]^-{\tau} \ar[d]_{\varphi} & A \odot H \ar[d]^{\varphi \odot H} \\ A^{\prime} \ar[r]_-{\tau^{\prime}} & A^{\prime} \odot H }
\]
 is commutative. The category of $H$-torsors in $\ca{C}$ is denoted by $\Tors_H(\ca{C})$.
\end{dfn}

\begin{example}\label{example:universal_torsor}
 Let $H$ be flat commutative Hopf algebra in $\Mod_R$. Let $\ubar{H}=(H,\delta)$ be the cofree comodule on $R$ (that is, the regular representation in $\Rep(G)$). Then $\tau \defl \delta \colon \ubar{H} \rightarrow \ubar{H} \odot H$ gives $\ubar{H}$ the structure of an $H$-torsor in $\Comod(H)$. Indeed, the underlying morphism in $\Mod_R$ of the composite
\[
 \xymatrix{ \ubar{H} \otimes \ubar{H} \ar[r]^-{\ubar{H} \otimes \delta} & \ubar{H} \otimes \ubar{H} \odot H \ar[r]^-{\mu \odot H} & \ubar{H} \odot H }
\]
 is one of the fusion operators of $H$. Its inverse is $\mu \otimes H \circ H\otimes s \otimes H \circ H \otimes \delta$, where $s$ denotes the antipode of $H$. Commutativity of the required diagrams is straightforward from the axioms of a Hopf algebra. The assumption that $H$ is flat implies that $\ubar{H}$ is faithfully flat.
\end{example}

\begin{example}
 If $X$ is a reasonable scheme or stack, then an $H$-torsor $(A,\tau)$ in $\QCoh(X)$ in the sense of Definition~\ref{dfn:torsor} corresponds to a scheme $Y=\Spec_X(A)$, affine and faithfully flat over $X$, together with an action $Y \times G \cong \Spec_X (A \odot H) \rightarrow Y$ of $G$ on $Y$ such that the morphism $G \times Y \rightarrow Y\times Y$ which represents $(g,y) \mapsto (gy,y)$ is an isomorphism. Thus $H$-torsors in $\QCoh(X)$ do indeed correspond to $G$-torsors on $X$ in the usual sense.
\end{example}

\begin{rmk}
 In \cite[\S 5.5]{BRANDENBURG_THESIS}, Brandenburg considered torsors of finite groups. In that case, both $H$ and $A$ have duals, and the definition of $H$-torsors can be rephrased in terms of $H^{\vee}$ and the coalgebra $A^{\vee}$. Thus, for \emph{finite} discrete groups, we arrive at the same notion of torsor as considered in \cite[Example~5.5.3]{BRANDENBURG_THESIS}. It is not clear if the two notions coincide for infinite groups.
\end{rmk}

 Recall that an \emph{Adams Hopf algebra} is a commutative Hopf algebra $H \in \Mod_R$ such that the category $\Comod(H)$ of comodules of $H$ is generated by objects with duals (equivalently, if and only if the classifying stack $BG$ of the affine group scheme corresponding to $H$ is an Adams stack).

\begin{lemma}\label{lemma:torsor_preserved}
 Let $\ca{C}$ be a tensor category with exact filtered colimits. If $H$ is an Adams Hopf algebra in $\Mod_R$, then any tensor functor $F \colon \Comod(H) \rightarrow \ca{C}$ preserves the torsor $\ubar{H}$ of Example~\ref{example:universal_torsor}.
\end{lemma}

\begin{proof}
 Since any tensor functor preserves colimits, we have a natural isomorphism $\psi \colon  F(\ubar{H}) \odot H \cong F(\ubar{H} \odot H)$. We claim that the composite $\psi F \tau$ endows $F\ubar{H}$ with the structure of an $H$-torsor. Naturality of $\psi$ implies that the required diagrams commute, and commutativity of the diagram
\[
 \xymatrix@C=20pt{&& F(\ubar{H})F(\ubar{H})H \ar[rd]^{\varphi^F_{\ubar{H},\ubar{H}} \odot H} \\
F(\ubar{H}) F(\ubar{H}) \ar[rd]_{\varphi^F_{\ubar{H},\ubar{H}}} \ar[r]^-{F\ubar{H} F\tau} & F(\ubar{H}) F(\ubar{H}H) \ar[rd]^{\varphi^F_{\ubar{H},\ubar{H}\odot H}} \ar[ru]^{F\ubar{H} \psi} & & F(\ubar{H}\ubar{H})H \ar[r]^-{F(\mu)H} & F(\ubar{H})H \\ 
& F(\ubar{H}\ubar{H}) \ar[r]_-{F(\ubar{H}\tau)} & F(\ubar{H}\ubar{H}H) \ar[ru]_{\psi} \ar[r]_-{F(\mu H)} & F(\ubar{H}H) \ar[ru]_{\psi} }
\]
 (where we have omitted the symbols $\otimes$ and $\odot$) implies that $\mu_{F\ubar{H}} \odot H \circ F\ubar{H} \otimes (\psi F\tau)$ is an isomorphism. It only remains to check that $F \ubar{H}$ is faithfully flat. This follows from the fact that it is the image of an Adams algebra under a tensor functor (see Proposition~\ref{prop:adams_strongly_stable}).
\end{proof}

 The following theorem shows that all $H$-torsors in reasonable tensor categories arise in this way.

\begin{thm}\label{thm:universal_torsor}
 Let $H \in \Mod_R$ be an Adams Hopf algebra. Let $\ca{C}$ be a tensor category with exact filtered colimits and with equalizers. Then the functor
\[
 \Fun_{c,\otimes}\bigl( \Comod(H),\ca{C} \bigr) \rightarrow \Tors_H(\ca{C})
\]
 which sends $F$ to $F \ubar{H}$ is an equivalence of categories.
\end{thm}

 As an immediate corollary, we find that the category of torsors is actually a groupoid.

\begin{cor}\label{cor:torsors_groupoid}
 Let $H \in \Mod_R$ be an Adams Hopf algebra and let $\ca{C}$ be a tensor category with equalizers and exact filtered colimits. Then $\Tors_{H}(\ca{C})$ is a groupoid.
\end{cor}

\begin{proof}
 By Theorem~\ref{thm:universal_torsor}, it suffices to check that any symmetric monoidal natural transformation between tensor functors $\Comod(H) \rightarrow \ca{C}$ is invertible. Since $\Comod(H)$ is generated by duals, this follows from the fact that ``duals invert,'' see for example \cite{DUALS_INVERT}.
\end{proof}

 The proof of the above theorem largely consists of showing that various diagrams in $\ca{C}$ are equalizers. The technique we usually use to do this is the following: given a faithfully flat algebra $A$ in $\ca{C}$ (in the sense of Definition~\ref{dfn:faithfully_flat}), the functor $A\otimes-$ preserves equalizers and reflects isomorphisms. Therefore it also \emph{detects} equalizers. It is often possible to show that the diagrams we care about are sent to a split equalizer by $A\otimes-$.

 In other words, we use basically use techniques of faithfully flat descent (equivalently, techniques used in the proof Beck's monadicity theorem) to prove Theorem~\ref{thm:universal_torsor}. Note, however, that the algebra $A$ will often change depending on the equalizer diagram in question, and the categories involved may not be abelian. Therefore the proof is not just a completely straightforward application of faithfully flat descent. For this reason we will provide full details below. We first introduce some notation.

\begin{notation}
 Let $H \in \Mod_R$ be a Hopf algebra. We write
\[
\ubar{H} \colon \Mod_R \rightarrow \Comod(H) 
\]
 for the functor which sends $M$ to $\ubar{H} \odot M =(H \otimes M, \delta \otimes M)$. The forgetful functor (its left adjoint) is denoted by $V \colon \Comod(H) \rightarrow \Mod_R$. We write $\rho \colon \id \Rightarrow \ubar{H}V$ for the unit of the adjunction. This is justified since its component at the comodule $\ubar{M}=(M,\rho)$ is $\rho \colon \ubar{M} \rightarrow \ubar{H} \odot M$.

 In particular, the component of $\rho \ubar{H} \colon \ubar{H} \rightarrow \ubar{H} V \ubar{H}$ at the trivial $R$-module $R$ is $\delta \colon \ubar{H} \rightarrow \ubar{H} \odot H$, that is, it is the structure morphism of the torsor of Example~\ref{example:universal_torsor}. 
\end{notation}

 Both $V$ and $\ubar{H}$ are symmetric lax monoidal functors. The following lemma shows that we can also talk about commutative algebras (and hence torsors) using the language of lax monoidal functors. Instead of translating back and forth, we use the lemma as a justification to work entirely in the world of symmetric lax monoidal functors. We denote the category of symmetric lax monoidal functors between two symmetric monoidal categories $\ca{C}$ and $\ca{D}$ by $\Fun_{\otimes,\lax}(\ca{C},\ca{D})$, and $\Fun_{c,\otimes,\lax}(\ca{C},\ca{D})$ for the full subcategory consisting of symmetric lax monoidal functors which are also left adjoints. Note that both $V$ and $\ubar{H}$ are of this latter kind.

\begin{lemma}\label{lemma:calg_as_monoidal}
 Let $\ca{C}$ be a tensor category. The functor which sends a symmetric lax monoidal left adjoint $F \colon \Mod_R \rightarrow \ca{C}$ to the commutative algebra $FR$ gives an equivalence between $\Fun_{c,\otimes,\lax}(\Mod_R,\ca{C})$ and the category of commutative algebras in $\ca{C}$. Its inverse sends a commutative algebra $A$ to the functor $ A\odot -$.
\end{lemma}
 
\begin{proof}
 This follows from the fact that all modules can be built via colimits from $R$ in a canonical way, and that colimits commute with tensor products both in $\Mod_R$ and $\ca{C}$. Details can be found in the proof of \cite[Theorem~5.1]{IM_KELLY}.
\end{proof}

 To construct an inverse to the functor from Theorem~\ref{thm:universal_torsor}, we need to construct tensor functors and symmetric monoidal natural transformations between them. We will use the following proposition to do this in both cases. Recall that a limit diagram $F \rightarrow F_i$ in a functor category is called \emph{pointwise} if $FC \rightarrow F_i C$ is a limit diagram in the target category for every object $C$ in the domain.

\begin{prop}\label{prop:pointwise_limits}
 The forgetful functor
\[
 \Fun_{\otimes, \lax} (\ca{C},\ca{D}) \rightarrow \Fun(\ca{C},\ca{D})
\]
 creates pointwise limits. Explicitly, given a pointwise limit diagram $\kappa^i \colon F \rightarrow F_i$ in $\Fun(\ca{C},\ca{D})$ such that all the $F_i$ are symmetric lax monoidal and all the $F_i \rightarrow F_j$ are symmetric monoidal natural transformations, there exists a unique symmetric lax monoidal structure on $F$ such that the $\kappa^i$ become symmetric monoidal, and that they exhibit the resulting $(F,\varphi^F,\varphi^F_0)$ as a limit of the $F_i$ in $\Fun_{\otimes, \lax}(\ca{C},\ca{D})$
\end{prop}

\begin{proof}
 The assumption implies that there exist unique morphisms $\varphi^F_{M,N}$ and $\varphi^F_0$ such that the diagrams
\[
 \vcenter{\xymatrix@C=40pt{FM \otimes FN \ar[d]_{\varphi^F_{M,N} } \ar[r]^-{\kappa^i_M \otimes \kappa^i_N} & F_i M \otimes F_i N \ar[d]^{\varphi^{F_i}_{M,N} } \\ F(M\otimes N) \ar[r]_-{\kappa^i_{M\otimes N}}  & F_i(M\otimes N)  }}
\quad \text{and}\quad
\vcenter{\xymatrix@C=20pt{ & \ar[ld]_{\varphi^F_0} \U \ar[rd]^{\varphi^{F_i}_0} \\ F\U \ar[rr]_{\kappa^i_{\U}} && F_i \U }}
\]
 are commutative for all $M,N \in \ca{C}$. Using the universal property for limits, one checks that these morphisms endow $F$ with the structure of a symmetric lax monoidal functor. Given a cone $\lambda^i \colon G \rightarrow F_i$ consisting of symmetric monoidal natural transformations, we get a unique natural transformation $\lambda \colon G \rightarrow F$ such that $\kappa^i \lambda=\lambda^i$. Again using the universal property of limits we find that $\lambda$ is symmetric monoidal.
\end{proof}

 The following lemma is the key to showing that the functor from Theorem~\ref{thm:universal_torsor} is fully faithful. It shows that certain equalizers in $\Comod(H)$ are preserved by \emph{any} tensor functor.

\begin{lemma}\label{lemma:pointwise_equalizer}
 Let $H \in \Mod_R$ be an Adams Hopf algebra, and let $\ca{C}$ be a tensor category with equalizers and exact filtered colimits. Let $ \ubar{M} \defl (M,\rho_M)$ be a comodule, and let $F \colon \Comod(H) \rightarrow \ca{C}$ be a tensor functor. Then
\begin{equation}\label{eqn:equalizer_image}
 \xymatrix@C=30pt{F\ubar{M} \ar[r]^-{F\rho_M} & F(\ubar{H} \odot M) \ar@<0.5ex>[r]^-{F(\delta \odot M)} \ar@<-0.5ex>[r]_-{F(\ubar{H} \odot \rho_M)} & F(\ubar{H} \odot H \odot M)}
\end{equation}
 is an equalizer diagram in $\ca{C}$. Moreover,
\[
 \xymatrix@C=30pt{F \ar[r]^-{F\rho} & F \ubar{H}V \ar@<0.5ex>[r]^-{F \rho \ubar{H} V } \ar@<-0.5ex>[r]_-{F \ubar{H} V\rho} & F\ubar{H} V \ubar{H} V}
\]
 is a (pointwise) equalizer diagram in $\Fun_{\otimes,\lax}\bigl(\Comod(H),\ca{C}\bigr)$.
\end{lemma}

\begin{proof}
 The second statement follows from the first by applying Proposition~\ref{prop:pointwise_limits} above.

 To see the first claim, note that it suffices to check that $F$ sends the diagram
\begin{equation}\label{eqn:tensored_equalizer}
 \xymatrix@C=40pt{\ubar{H} \otimes \ubar{M} \ar[r]^-{ \ubar{H} \otimes \rho_M} & \ubar{H} \otimes \ubar{H} \odot M \ar@<0.5ex>[r]^-{F(\delta \odot M)} \ar@<-0.5ex>[r]_-{\ubar{H} \otimes \ubar{H} \odot \rho_M} & \ubar{H} \otimes \ubar{H} \odot H \odot M}
\end{equation}
 to an equalizer diagram. Indeed, since $F$ is a tensor functor, the image of Diagram~\eqref{eqn:tensored_equalizer} under $F$ is isomorphic to the image of Diagram~\eqref{eqn:equalizer_image} under the functor $F \ubar{H} \otimes -$. Since $\ubar{H}$ is an Adams algebra, Proposition~\ref{prop:adams_strongly_stable} implies that $F \ubar{H}$ is faithfully flat. From the assumption that $\ca{C}$ has equalizers it follows that $F \ubar{H} \otimes -$ \emph{detects} equalizers, so it does suffice to check that the image of Diagram~\eqref{eqn:tensored_equalizer} under $F$ is an equalizer, as claimed.

 In fact, we will show that Diagram~\ref{eqn:tensored_equalizer} is a split equalizer, and hence absolute (that is, it is preserved by \emph{any} functor, not just $F$). Recall that the fact that the projection formula holds for the adjunction $V \dashv \ubar{H}$ means that for any comodule $\ubar{N}=(N,\rho_N)$, the natural morphism $\ubar{H} \otimes \ubar{N} \rightarrow \ubar{H} \odot N$ is an isomorphism (see \cite[\S 3]{FAUSK_HU_MAY} for the definition and \cite[Lemma~1.1.5]{HOVEY} for a proof). Thus Diagram~\eqref{eqn:tensored_equalizer} above is isomorphic to the image of the split equalizer diagram
\[
 \xymatrix{M \ar[r]^-{\rho} & H\otimes M \ar@<0.5ex>[r]^-{\delta \otimes M} \ar@<-0.5ex>[r]_-{H \otimes \rho_M} & H\otimes H \otimes M}
\]
 in the category $\Mod_R$ (with splitting given by $\varepsilon \otimes M$ and $ \varepsilon \otimes H \otimes M$) under the functor $\ubar{H} \colon \Mod_R \rightarrow \Comod(H)$.
\end{proof}

\begin{lemma}\label{lemma:torsor_fully_faithful}
 The functor from Theorem~\ref{thm:universal_torsor} is fully faithful.
\end{lemma}

\begin{proof}
 Since commutative algebras can be identified with symmetric lax monoidal functors by Lemma~\ref{lemma:calg_as_monoidal}, this boils down to the following statement: for tensor functors $F,G \colon \Comod(H) \rightarrow \ca{C}$ and any symmetric monoidal natural transformation $\psi \colon F \ubar{H} \rightarrow G \ubar{H}$ making the diagram
\[
 \xymatrix{ F \ubar{H} \ar[d]_{\psi} \ar[r]^-{F \rho \ubar{H}} & F\ubar{H} V \ubar{H} \ar[d]^{\psi V \ubar{H}} \\ G \ubar{H}  \ar[r]_-{G \rho \ubar{H}} & G \ubar{H} V \ubar{H}}
\]
 commutative, there exists a unique symmetric monoidal natural transformation $\varphi \colon F \rightarrow G$ such that $\psi =\varphi \ubar{H}$.

 The assumption on $\psi$ and naturality of $\psi$ imply that the solid arrow part of the diagram
\[
 \xymatrix{ F \ar@{..>}[d]_{\varphi} \ar[r]^-{F \rho} & F\ubar{H} V \ar[d]_{\psi V } \ar@<0.5ex>[r]^-{F \rho \ubar{H} V } \ar@<-0.5ex>[r]_-{F \ubar{H} V\rho} & F\ubar{H} V \ubar{H} V \ar[d]^{\psi V \ubar{H} V} \\ 
G  \ar[r]_-{G \rho} & G\ubar{H} V  \ar@<0.5ex>[r]^-{G \rho \ubar{H} V } \ar@<-0.5ex>[r]_-{G \ubar{H} V\rho} & G\ubar{H} V \ubar{H} V }
\]
 is commutative. Since both rows are pointwise equalizers in the functor category $\Fun_{\otimes,\lax}\bigr(\Comod(H),\ca{C}\bigl)$ (see Lemma~\ref{lemma:pointwise_equalizer}), there exists a unique symmetric monoidal natural transformation $\varphi$ making the above diagram commutative.

 The same lemma shows that $G\rho \ubar{H}$ is a monomorphism. Precomposing the left square with $\ubar{H}$ therefore shows that $\varphi \ubar{H}=\psi$. This shows that the functor of Theorem~\ref{thm:universal_torsor} is full.

 To see that it is faithful, let $\varphi^{\prime} \colon F\rightarrow G$ be a symmetric monoidal natural transformation with $\varphi^{\prime} \ubar{H}=\psi$. Then $\varphi^{\prime}$ makes the left square above commutative by naturality. Thus $\varphi^{\prime}=\varphi$, which shows that the functor of Theorem~\ref{thm:universal_torsor} is also faithful.
\end{proof}

 In what follows, given a commutative algebra $A \in \ca{C}$, we simply write $A$ for the corresponding symmetric lax monoidal functor $A\odot - \colon \Mod_R \rightarrow \ca{C}$. This is justified by Lemma~\ref{lemma:calg_as_monoidal}, and in accordance with our convention to write $\ubar{H}$ for $\ubar{H} \odot - \colon \Mod_R \rightarrow \Comod(H)$.

\begin{lemma}\label{lemma:construction_of_fa}
 Let $\ca{C}$ be a tensor category with equalizers and exact filtered colimits, and let $(A,\tau)$ be an $H$-torsor. Let
\[
 \xymatrix{F_A \ar[r]^-{\xi} & AV \ar@<0.5ex>[r]^-{\tau V} \ar@<-0.5ex>[r]_-{AV\rho} & AV\ubar{H}V }
\]
 be an equalizer diagram in $\Fun_{\otimes,\lax}\bigl(\Comod(H),\ca{C}\bigr)$. Then there exists a unique symmetric monoidal isomorphism $\alpha \colon F_A \ubar{H} \rightarrow A$ such that $\tau \alpha=\xi \ubar{H}$. Moreover, the diagram
\[
 \xymatrix@C=10pt{ F_A \ar[rd]_{\xi} \ar[rr]^-{F_A \rho} && F_A \ubar{H} V \ar[ld]^{\alpha V} \\ & AV}
\]
 is commutative.
\end{lemma}

\begin{proof}
 Since the equalizer is computed pointwise (see Proposition~\ref{prop:pointwise_limits}) and
\[
\ubar{H} \colon \Mod_R \rightarrow \Comod(H) 
\]
 is exact, the equalizer is preserved by precomposing with $\ubar{H}$. In order to show that the unique isomorphism $\alpha$ exists, it therefore suffices to check that
\[
\xymatrix{A \ar[r]^-{\tau} & AV\ubar{H} \ar@<0.5ex>[r]^-{\tau V \ubar{H}} \ar@<-0.5ex>[r]_-{AV\rho \ubar{H}} & AV\ubar{H}V \ubar{H} }
\]
 is also an equalizer diagram in $\Fun_{\otimes,\lax}(\Mod_R,\ca{C})$. This diagram is in fact a split equalizer, with splitting given by $AV\ubar{H} \varepsilon \colon AV\ubar{H}V \ubar{H} \rightarrow AV\ubar{H}$ and $A\varepsilon \colon AV\ubar{H} \rightarrow A$, where $\varepsilon$ denotes the counit of the adjunction $V \dashv \ubar{H}$.

 To see the second claim, note that $\tau V$ is a monomorphism. Indeed, since $A$ is an $H$-torsor, it is faithfully flat, and it suffices to show that $A\tau$ (and therefore $A \tau V$) is a split monomorphism. This follows from the fact that the composite $\mu \odot H \circ A \otimes \tau$ is an isomorphism (see Definition~\ref{dfn:torsor}). Commutativity of the diagram
\[
 \xymatrix{ & F_A \ubar{H} V \ar[rd]^{\xi \ubar{H} V} \ar[r]^-{\alpha V} & AV \ar[d]^{\tau V} \\ F_A \ar[r]_-{\xi} \ar[ru]^{F_A \rho} & AV \ar@<0.5ex>[r]^-{AV \rho} \ar@<-0.5ex>[r]_-{\tau V} & AV \ubar{H} V }
\]
 therefore implies that $\alpha V \circ F_A \rho=\xi$, as claimed.
\end{proof}

\begin{lemma}\label{lemma:torsor_eso}
 Let $\ca{C}$ be a tensor category with equalizers and exact filtered colimits, and let $(A,\tau)$ be an $H$-torsor in $\ca{C}$. Then the functor $F_A$ of Lemma~\ref{lemma:construction_of_fa} is a symmetric \emph{strong} monoidal left adjoint, and the isomorphism $\alpha \colon F_A \ubar{H} \rightarrow A$ is an isomorphism of $H$-torsors $(F_A \ubar{H},F_A \rho \ubar{H}) \rightarrow (A,\tau)$.
\end{lemma}

\begin{proof}
 Since $A$ is faithfully flat, it suffices to show that the composite
\[
 \xymatrix{\Comod(H) \ar[r]^-{F_A} & \ca{C} \ar[r]^-{A\otimes -} & \ca{C}_A}
\]
 is symmetric strong monoidal and cocontinuous to prove the same for $F_A$. Establishing these two properties implies the first claim, since any cocontinuous functor with domain $\Comod(H)$ is a left adjoint (this follows from the fact that $\Comod(H)$ is locally finitely presentable).

 We claim that the (a priori only symmetric lax monoidal) composite $AF_A$ is isomorphic to $A V$ as symmetric lax monoidal functors. More precisely, both these functors have natural lifts to the category $\ca{C}_A$ of $A$-modules (which we do not distinguish in our notation). The lift of $AV$ to $\ca{C}_A$ is symmetric strong monoidal. Thus it suffices to check that the lifts of $AF_A$ and $AV$ to $\ca{C}_A$ are isomorphic as lax monoidal functors.

 To see this, let $\overline{\tau} \defl \mu V\ubar{H} \circ A\tau \colon AA \rightarrow AV \ubar{H}$ be the morphism of $A$-algebras corresponding to $\tau \colon A \rightarrow AV \ubar{H}$. Since $(A,\tau)$ is a torsor, $\overline{\tau}$ is an isomorphism (see Definition~\ref{dfn:torsor}). The solid arrow part of the diagram
\[
 \xymatrix{AF_A \ar@{..>}[d]_{\beta} \ar[r]^-{A\xi} & AAV \ar@<0.5ex>[r]^-{A\tau V} \ar@<-0.5ex>[r]_-{AAV\rho} \ar[d]_{\overline{\tau} V} & AAV\ubar{H} V \ar[d]^{\overline{\tau} V\ubar{H} V} \\
 AV \ar[r]_-{AV\rho} & AV\ubar{H}V \ar@<0.5ex>[r]^-{AV\rho \ubar{H}V} \ar@<-0.5ex>[r]_-{AV\ubar{H}V \rho} & AV\ubar{H} V \ubar{H} V}
\]
 in the category $\Fun_{\otimes, \lax}\bigl(\Comod(H), \ca{C}_A \bigr)$ is commutative, so it suffices to check that both rows are equalizer diagrams to get the desired isomorphism $\beta$. This is immediate for the top row since $A$ is flat, and the bottom row is a split equalizer, with splitting given by $A\varepsilon V$ and $A\varepsilon V \ubar{H} V$, where $\varepsilon$ denotes the counit of the adjunction $V \dashv \ubar{H}$.

 It remains to check that $\alpha$ is an isomorphism of torsors. This follows from commutativity of the diagram
\[
 \xymatrix{F_A \ubar{H} \ar[r]^-{F_A \rho \ubar{H}} \ar[rd]^{\xi \ubar{H}} \ar[d]_{\alpha} & F_A \ubar{H} V\ubar{H} \ar[d]^{\alpha V\ubar{H}} \\ A \ar[r]_-{\tau} & AV\ubar{H}}
\]
 (see Lemma~\ref{lemma:construction_of_fa}).
\end{proof}

\begin{proof}[Proof of Theorem~\ref{thm:universal_torsor}]
 Let $H$ be an Adams Hopf algebra, and let $\ca{C}$ be a tensor category with equalizers and exact filtered colimits. The functor
\[
 \Fun_{c,\otimes}\bigl(\Comod(H), \ca{C} \bigr) \rightarrow \Tors_H(\ca{C})
\]
 which sends $F$ to $F\ubar{H}$ is well-defined by Lemma~\ref{lemma:torsor_preserved}. It is fully faithful by Lemma~\ref{lemma:torsor_fully_faithful} and essentially surjective by Lemma~\ref{lemma:torsor_eso}.
\end{proof}

\subsection{The case of the general linear group}\label{section:locally_free_objects}

 The goal of this section is to analyze the structure of torsors over the general linear group in general tensor categories. Let $H_d=R[x_{ij}\vert i,j=1,\ldots,d]_{\mathrm{det}}$ be the Hopf algebra of regular functions on the affine group scheme $\mathrm{GL}_d \defl \mathrm{GL}_d(R)$.

 It is well known that $\mathrm{BGL}_d$ classifies locally free objects of rank $d$ on schemes and algebraic stacks $X$. As we will see, this is true in a much more general context. In order to do this we will first give a definition of locally free objects in an arbitrary tensor category.

\begin{dfn}\label{dfn:locally_free_objects}
 Let $\ca{C}$ be a tensor category with exact filtered colimits. An object $M \in \ca{C}$ with a dual is called \emph{locally free of rank $d$} if there exists a faithfully flat commutative algebra $B \in \ca{C}$ such that $B \otimes M \cong B^{\oplus d}$ in the category $\ca{C}_B$ of $B$-modules in $\ca{C}$. We denote the groupoid of locally free objects of rank $d$ in $\ca{C}$ and isomorphisms between them by $\LF^{\rk d}_{\iso}(\ca{C})$.
\end{dfn}

\begin{rmk}
 Ideally, we would like to have a more intrinsic characterization of locally free objects of rank $d$. In characteristic zero, such a characterization has been given by Brandenburg, see \cite[\S 4.9]{BRANDENBURG_THESIS}. In positive characteristic this remains an open question.
\end{rmk}

 Instead of studying the above mentioned problem, we use the theory of torsors developed in \S \ref{section:torsors_subsection} to give a different characterization of the groupoid $\LF^{\rk d}_{\iso}(\ca{C})$: we will show that there is an equivalence
\[
 \Fun_{c,\otimes}\bigl(\Comod(H_d), \ca{C}) \rightarrow \LF^{\rk d}_{\iso}(\ca{C})
\]
 under a mild restriction on the tensor category $\ca{C}$ (see Theorem~\ref{thm:universal_property_rep_gld} below for a precise statement).

 In this section, we will write $M_B$ for the free $B$-module $B\otimes M$ on an object $M \in \ca{C}$, and we will write $M^d$ for the direct sum $M^{\oplus d}$ of $d$ copies of $M$. 

\begin{dfn}\label{dfn:framing}
 Let $\ca{C}$ be a tensor category, $M \in \ca{C}$ an object, and $B \in \ca{C}$ a commutative algebra. A \emph{$d$-framing} of $M$ (over $B$) is a morphism $\xi \colon M \rightarrow B^d$ such that the induced morphism $\overline{\xi} \colon M_B \rightarrow B^d$ is an isomorphism.
\end{dfn}
 
\begin{lemma}\label{lemma:framings_functor}
 If $\xi$ is a $d$-framing of $M$ over $B$ and $\varphi \colon B \rightarrow B^{\prime}$ is a morphism of commutative algebras, then $\varphi^d \circ \sigma$ is a $d$-framing of $M$ over $B^{\prime}$. Thus $d$-framings on $M$ define a functor
\[
 \mathrm{Fr}^d_M \colon \CAlg_{\ca{C}} \rightarrow \Set
\]
 on the category of commutative algebras on $\ca{C}$.
\end{lemma}

\begin{proof}
 This follows from the fact that $\overline{\varphi^d \xi}$ is (up to isomorphism) given by the image of $\overline{\xi}$ under the tensor functor $B^{\prime} \ten{B} - \colon \ca{C}_B \rightarrow \ca{C}_{B^{\prime}}$.
\end{proof}

 In order to state the next proposition, we need to fix some notation. Recall that $\Sym(C)=\bigoplus_{n \in \mathbb{N}} \Sym^{n}(C)$ denotes the free commutative algbera on the object $C \in \ca{C}$. Here $\Sym^n(C)$ is the quotient of $C^{\otimes n}$ by the action of the symmetric group $\Sigma_n$. As a left adjoint, $\Sym(-)$ sends finite direct sums to finite coproducts of commutative algebras, that is, to tensor products.

 Now fix an object $M \in \ca{C}$ with a dual. We write $M^d=\bigoplus_{i=1}^d M_i$ to keep track of the different copies of $M$ in the direct sum. Using the fact about direct sums mentioned above and the isomorphism $\Sym^1(C)\cong C$, we get natural inclusions of $M_i^{\vee} \otimes M_j$ in $\Sym\bigl((M^{\vee})^d \oplus M^d \bigr )$. Using the (twisted) coevaluation $\U \rightarrow M\otimes M^{\vee} \cong M^{\vee}\otimes M$ on $M$ and the universal property of $\Sym(\U)$ we get morphisms
\[
 \alpha_{ij} \colon \Sym(\U) \rightarrow \Sym\bigl((M^{\vee})^d \oplus M^d \bigr )
\]
 whose composite with the inclusion $\U \rightarrow \Sym(\U)$ is given by the twisted coevaluation on $M_i^{\vee} \otimes M_j$.

 Given $i, j \in \{1,\ldots, d\}$, we write $\delta_{ij} \colon \Sym(\U) \rightarrow \U$ for the unique morphism of algebras which restricts to the identity of $\U$ if $i=j$, and to the zero morphism if $i \neq j$.

 We write $\beta$ for the morphism
\[
 \Sym(M^{\vee} \otimes M) \rightarrow \Sym\bigl((M^{\vee})^d \oplus M^d \bigr )
\]
 induced by the composite of the diagonal $\colon M^{\vee} \otimes M \rightarrow  \bigoplus_i M_i^{\vee} \otimes M_i$ and the inclusion of $\bigoplus_i M_i^{\vee} \otimes M_i$ in $\Sym^2\bigl((M^{\vee})^d \oplus M^d \bigr )$.

 Finally, we let $\overline{\ev} \colon \Sym(M^{\vee}\otimes M) \rightarrow \U$ denote the unique morphism of commutative algebras which restricts to the evaluation $\ev \colon M^{\vee} \otimes M \rightarrow \U$.

\begin{lemma}\label{lemma:universal_framing}
 Let $M \in \ca{C}$ be an object with a dual. Let $A_M$ be the colimit of the diagram
\[
 \xymatrix@!C=20pt{& \Sym(\U) \ar[rd]^{\alpha_{ij}} \ar[ld]_{\delta_{ij}} && \Sym(M^{\vee} \otimes M) \ar[ld]_{\beta} \ar[rd]^{\overline{\ev}} \\ \U && \Sym\bigl((M^{\vee})^d \oplus M^d \bigr ) && \U}
\]
 in $\CAlg_{\ca{C}}$. Denote the composite
\[
 \xymatrix{M=M_i \ar[r] & \Sym(M_i)  \ar[r]^-{\mathrm{incl}_i} & \Sym\bigl((M^{\vee})^d \oplus M^d \bigr ) \ar[r] & A_M}
\]
 by $\sigma_i \colon M \rightarrow A_M$. Then $\sigma=(\sigma_i)_{i=1}^d \colon M \rightarrow A_M^d$ gives a $d$-framing of $M$ over $A_M$. Moreover, this $d$-framing is \emph{universal}: given any other $d$-framing $\xi \colon M \rightarrow B$, there exists a unique morphism of commutative algebras $\varphi \colon A_M \rightarrow B$ such that the diagram
\[
 \xymatrix@C=10pt{ & M \ar[rd]^{\xi} \ar[ld]_{\sigma} \\ A_M^d \ar[rr]_-{\varphi^d} && B^d }
\]
 is commutative. Thus, for objects $M$ with dual, the functor $\mathrm{Fr}_M^d \colon \CAlg_{\ca{C}} \rightarrow \Set$ is represented by $A_M$.
\end{lemma}

\begin{proof}
 Since $\U$ is the initial algebra, giving a morphism $A_M \rightarrow B$ simply amounts to giving a morphism $\Sym\bigl((M^{\vee})^d \oplus M^d \bigr ) \rightarrow B$ subject to various restrictions involving the evaluation and coevaluation of $M$. Explicitly, such a morphism amounts to giving morphisms $\tau_i \colon M^{\vee}_B \rightarrow B$ and $\sigma_j \colon M_B \rightarrow B$ of $B$-modules, $i,j=1,\ldots,d$, subject to the following equations. Writing $\overline{\tau}_i$ for the dual of $\tau_i$, the equations coming from the $\alpha_{ij}$ say that $\sigma_j \overline{\tau}_i$ is the identity for $i=j$ and zero if $i \neq j$, and the equation coming from $\beta$ says that $\sum_{i=1}^d  \tau_i \sigma_i$ is the identity on $M_B$. Thus the morphisms $(\sigma_i)_{i=1}^d \colon M_B \rightarrow B^d$ and $(\overline{\tau}_i)_{i=1}^d \colon B^d \rightarrow M_B$ are mutually inverse. This shows that $\sigma \colon M \rightarrow A_M^d$ is indeed a universal $d$-framing.
\end{proof}

\begin{example}\label{example:framing_on_free_module}
 The universal $d$-framing on the module $R^d \in \Mod_R$ is given by $A_{R^d} =H_d=R[x_{ij}]_{\mathrm{det}}$, where $\sigma \colon R^d \rightarrow (H_d)^{\oplus d}$ sends the basis vector $e_i$ to $\sum x_{ij}e_j$. Indeed, a $d$-framing $R^d \rightarrow B^d$ simply amounts to picking an invertible matrix in the commutative $R$-algebra $B$.
\end{example}

\begin{lemma}\label{lemma:framings_preserved}
 Tensor functors preserve universal $d$-framings on objects with duals.
\end{lemma}

\begin{proof}
 Any tensor functor induces a left adjoint between categories of commutative algebras. Since universal $d$-framings are given by the colimit in the category of commutative algebras described in Lemma~\ref{lemma:universal_framing}, they are preserved by tensor functors.
\end{proof}

 The following proposition gives a characterization of locally free objects of rank $d$ in a tensor category with exact filtered colimits (hence in particular in an lfp tensor category).

\begin{prop}\label{prop:locally_free_objects}
 Let $\ca{C}$ be a tensor category with exact filtered colimits and let $M \in \ca{C}$ be an object with dual. Then the following are equivalent:
\begin{itemize}
 \item[(i)] The object $M$ is locally free of rank $d$ (see Definition~\ref{dfn:locally_free_objects});
\item[(ii)] The universal $d$-framing $A_M$ is a faithfully flat algebra;

\item[(iii)] There exists a tensor category $\ca{D}$ with exact filtered colimits and a conservative and left exact tensor functor $F \colon \ca{C} \rightarrow \ca{D}$ such that $FM \cong \U^d$, where $\U$ denotes the unit of $\ca{D}$.
\end{itemize}
\end{prop}

\begin{proof}
 Using the fact that $\ca{C}_B$ has exact filtered colimits if $B$ does, it is clear that (ii) implies (i) and (i) implies (iii). To see that (iii) implies (ii), let $F$ be a tensor functor as in (iii).

 Since tensor functors preserve universal $d$-framings (see Lemma~\ref{lemma:framings_preserved} above), we know that $F(A_M)$ is a universal $d$-framing of $FM \cong \U^d$. On the other hand, $\U^d$ is also the image of $R^d$ under the tensor functor $\U \odot - \colon \Mod_R \rightarrow \ca{D}$, so the universal $d$-framing of $FM\cong \U^d$ is also isomorphic to $\U \odot H_d$. This algebra is faitfhfully flat as image of an Adams algebra under a tensor functor (see Proposition~\ref{prop:adams_strongly_stable}). This shows that $F(A_M)$ is faithfully flat, and since $F$ is conservative and preserves finite limits, it follows that $A_M$ is faithfully flat as well.
\end{proof}

\begin{lemma}\label{lemma:universal_locally_free_object}
 Let $(\ubar{R}^d,\rho)$ be the standard representation of $\mathrm{GL}_d$ given by the coaction $\rho \colon \ubar{R}^d \rightarrow \ubar{H}_d \odot R^d$ with $\rho(e_i)=\sum x_{ij}\otimes e_j$. Then $\rho$ exhibits $\ubar{H}_d$ as universal $d$-framing of $\ubar{R}^d$ in $\Comod(H_d)$. Moreover, if $\ca{C}$ is a tensor category with filtered exact colimits and $F \colon \Comod(H_d)\rightarrow \ca{C}$ is a tensor functor, then $F\ubar{R}^d$ is locally free of rank $d$. 
\end{lemma}

\begin{proof}
 Since $V \colon \Comod(H_d) \rightarrow \Mod_R$ preserves tensor products and creates colimits, it suffices to check that $V\sigma$ exhibits $H_d$ as universal framing of $R^d$, which is precisely the content of Example~\ref{example:framing_on_free_module}.

 To see the second claim, note that $F\sigma$ exhibits $F\ubar{H}_d$ as universal $d$-framing of $F\ubar{R}^d$ since tensor functors preserve universal $d$-framings (see Lemma~\ref{lemma:framings_preserved}). But $F\ubar{H}_d$ is faithfully flat as an image of an Adams algebra (see Proposition~\ref{prop:adams_strongly_stable}). Thus the second condition of Proposition~\ref{prop:locally_free_objects} is satisfied, which shows that $F\ubar{R}^d$ is indeed locally free.
\end{proof}

 We are now ready to state the main theorem of this section.

\begin{thm}\label{thm:universal_property_rep_gld}
 Let $\ca{C}$ be a tensor category with equalizers and exact filtered colimits. Then the functor
\[
 \Fun_{c,\otimes}\bigl(\Comod(H_d),\ca{C}\bigr) \rightarrow \LF^{\rk d}_{\iso}(\ca{C})
\]
 which sends $F$ to $F \ubar{R}^d$ is an equivalence of categories.
\end{thm}

 We will use the equivalence between $\Fun_{c,\otimes}\bigl(\Comod(H_d),\ca{C}\bigr)$ and $\Tors_{H_d}(\ca{C})$ (see Theorem~\ref{thm:universal_torsor}) to prove this. The difficult part is to construct a torsor from a locally free object $M$. We can do this using the algebra with a universal $d$-framing on the object $M$.

\begin{lemma}\label{lemma:torsor_from_locally_free}
 Let $M$ be a locally free object of rank $d$ in the tensor category $\ca{C}$. Then there exists a unique morphism of algebras $\tau \colon A_M \rightarrow A_M \odot H_d$ such that the diagram
\[
 \xymatrix{M \ar[r]^-{\sigma} \ar[d]_{\sigma} & A_M \odot R^d \ar[d]^{A_M \odot \rho} \\ A_M \odot R^d \ar[r]_-{\tau \odot R^d} & A_M \odot H_d \odot R^d}
\]
 is commutative. The pair $(A_M,\tau)$ is an $H_d$-torsor, and the diagram
\begin{equation}\label{eqn:equalizer_from_locally_free}
 \xymatrix{M \ar[r]^-{\sigma} & A_M \odot R^d \ar@<0.5ex>[r]^-{A_M \rho} \ar@<-0.5ex>[r]_-{\tau R^d} & A_M \odot H_d \odot R^d}
\end{equation}
is an equalizer diagram in $\ca{C}$.

 Moreover, this construction is functorial in $M$: given an isomorphism $M \rightarrow M^{\prime}$ of locally free objects, the induced isomorphism $A_M \rightarrow A_{M^{\prime}}$ is an isomorphism of torsors.
\end{lemma}

\begin{proof}
 The last statement about functoriality follows immediately from the fact that all the morphisms involved in the diagram are defined by the universal property of $\sigma$.

 Thus we only need to show that, for a fixed $M$, the morphism of algebras $\tau$ exists and that $(A_M,\tau)$ is a torsor. Troughout the proof we write $A=A_M$, $H=H_d$, and we omit the symbols $\otimes, \odot$.

 Commutativity of the diagram
\[
 \xymatrix@!C=30pt{AHM \ar[rr]^-{AH\sigma} \ar[d]_{(12)} & & AHAR^d \ar[ld]_{(12)} \ar[rd]^{(23)} \ar[rrr]^-{AHA\rho} &&& AHAHR^d \ar[d]^{(23)} \\ 
 HAM \ar[r]^-{HA \sigma} \ar[rd]_{\cong} & HAAR^d \ar[d]^{H \mu R^d} && AAHR^d \ar[dd]_{\mu HR^d} \ar[rr]^{AAH\rho} && AAHHR^d \ar[ld]_{\mu HH R^d} \ar[dd]^{\mu\mu R^d} \\
& HAR^d \ar[rrd]_{(12)} &&& AHHR^d \ar[rd]^{A\mu R^d} \\
& & & AHR^d \ar[rr]_{\cong} \ar[ru]^{AH \rho} && AHR^d}
\]
 shows that $A\rho \circ \sigma$ is a $d$-framing of $M$ over $AH$. Thus we do indeed get a unique morphism $\tau \colon A \rightarrow AH$ of algebras making the desired diagram commutative.

 Commutativity of the defining diagrams of a torsor is straightforward to check using the universal property of $\sigma$. To complete the proof that $(A,\tau)$ is a torsor, it only remains to show that the composite morphism $\mu H \circ A\tau$ of algebras is an isomorphism. The top horizontal composite in the commutative diagram
\[
 \xymatrix{AM \ar[r]^-{A \sigma} \ar[d]_{A \sigma} & AAR^d \ar[d]^{AA \rho} \ar[r]^{\mu R^d} & AR^d \ar[d]^{A \rho} \\
AAR^d \ar[r]_-{A \tau R^d} & AAHR^d \ar[r]_-{\mu H R^d} & AHR^d }
\]
 of $A$-modules is an isomorphism since $\sigma$ is a $d$-framing over $A$. But both $(AM,A\sigma)$ and $(AR^d,A\rho)$ are universal $d$-framings in the category $\ca{C}_A$ of $A$-algebras. Thus $\mu H \circ A\tau$ is an isomorphism, as claimed.

 To conclude the proof we need to check that Diagram~\eqref{eqn:equalizer_from_locally_free} is an equalizer diagram. Since $A$ is faithfully flat, it suffices to check that the diagram obtained by tensoring it with $A$ is an equalizer diagram. Let $\overline{\sigma} \colon AV \rightarrow AR^d$ and $\overline{\tau} \colon AA \rightarrow AH$ denote the morphisms of $A$-modules corresponding to $\sigma$ and $\tau$ respectively. By definition of $d$-framings, $\overline{\sigma}$ is an isomorphism, and---as we have just observed above---so is $\overline{\tau}$. The commutative diagram
\[
 \xymatrix{AM \ar[d]_{\overline{\sigma}} \ar[r]^-{A\sigma} & AAR^d \ar[d]_{\overline{\tau} R^d} \ar@<0.5ex>[r]^-{AA\rho} \ar@<-0.5ex>[r]_-{A\tau R^d} & AAHR^d \ar[d]^{\overline{\tau} HR^n} \\ 
AR^d  \ar[r]^-{A\rho} & AHR^d \ar@<0.5ex>[r]^-{AH\rho} \ar@<-0.5ex>[r]_-{A \delta R^d} & AHHR^d}
\]
 therefore reduces the problem to checking that the bottom row is an equalizer. It is in fact a split equalizer, with splitting given by the counit of $H$.
\end{proof}

\begin{proof}[Proof of Theorem~\ref{thm:universal_property_rep_gld}]
 Using the equivalence 
\[
\Fun_{c,\otimes}\bigl(\Comod(H_d), \ca{C}\bigr) \rightarrow \Tors_{H_d}(\ca{C}) 
\]
 of Theorem~\ref{thm:universal_torsor} we can reduce the problem to showing that the construction of Lemma~\ref{lemma:torsor_from_locally_free} gives an equivalence $\LF^{\rk d}_{\iso}(\ca{C}) \rightarrow \Tors_{H_d}(\ca{C}) $. Given an $H_d$-torsor $(A,\tau)$, let $M^A \defl F_A(\ubar{R}^d)$, where $F_A$ denotes the tensor functor defined in Lemmas~\ref{lemma:construction_of_fa} and \ref{lemma:torsor_eso}, that is, $M^A$ is defined by the equalizer diagram
\[
\xymatrix{M^A \ar[r]^-{\xi} & A \odot R^d \ar@<0.5ex>[r]^-{A \odot \rho} \ar@<-0.5ex>[r]_-{\tau \odot R^d} & A\odot H_d \odot R^d} 
\]
 in $\ca{C}$. The object $M^A$ is thus locally free of rank $d$ by Lemma~\ref{lemma:universal_locally_free_object}. From the fact that Diagram~\eqref{eqn:equalizer_from_locally_free} is an equalizer diagram we immediately get a natural isomorphism $M \cong M^{A_M}$. 

 It remains to check that $A \cong A_{M^A}$. To see this, it suffices to check that $\xi$ exhibits $A$ as universal $d$-framing of $M^A$.

 Since tensor functors preserve universal $d$-framings, $F_A\rho$ exhibits $F_A \ubar{H}_d$ as universal $d$-framing of $M^A$. The claim therefore follows from the existence of the isomorphism $\alpha \colon F_A \ubar{H}_d \cong A$ of algebras with $\alpha \odot R^d \circ F_A \rho =\xi$ (see Lemma~\ref{lemma:construction_of_fa}).
\end{proof}
\section{Universal geometric tensor categories}\label{section:limits} 

\subsection{Generalized Tannaka duality}\label{section:generalized_td}

 As mentioned in the introduction, we will use the generalized Tannaka duality developed in \cite{SCHAEPPI_STACKS, SCHAEPPI_INDABELIAN, SCHAEPPI_GEOMETRIC} to prove that the 2-category $\ca{AS}$ of Adams stacks over $R$ is (bicategorically) cocomplete. This duality gives a contravariant biequivalence between $\ca{AS}$ and a certain 2-category of right exact symmetric monoidal categories, the \emph{weakly Tannakian categories}. In order to recall the definition of these, we first need to introduce some terminology.

\begin{dfn}\label{dfn:rex_monoidal}
 Let $\ca{A}$ be an essentially small $R$-linear with finite colimits (that is, it is additive and has cokernels). Then $\ca{A}$ is called \emph{right exact symmetric monoidal} if it is symmetric monoidal, and for all $A \in \ca{A}$, the functor $A \otimes - \colon \ca{A} \rightarrow \ca{A}$ is right exact.

 The 2-category of right exact symmetric monoidal categories, right exact symmetric strong monoidal $R$-linear functors, and symmetric monoidal natural transformations is denoted by $\ca{RM}$.
\end{dfn}

 We also need the concept of an ind-abelian category: an essentially small $R$-linear category $\ca{A}$ is \emph{ind-abelian} if and only if $\Ind(\ca{A})$ is abelian. An intrinsic definition of ind-abelian categories is given in \cite[Definition~1.1]{SCHAEPPI_INDABELIAN}. The theorem that the intrinsic definition coincides with the one given here is proved using the notion of ind-class introduced in \cite[\S 2]{SCHAEPPI_INDABELIAN} (see also Definition~\ref{dfn:ind_class}). More precisely, a finitely cocomplete $R$-linear category $\ca{A}$ is ind-abelian if and only if the class of \emph{all} cokernel diagrams in $\ca{A}$ forms an ind-class. The author recently learned that this was also proved by Breitsprecher in \cite{BREITSPRECHER} (using different methods).

\begin{dfn}\label{dfn:weakly_tannakian}
 Let $\ca{A}$ be an abelian right exact symmetric monoidal $R$-linear category, and $B$ a commutative $R$-algebra. A symmetric strong monoidal $R$-linear functor
\[
 w \colon \ca{A} \rightarrow \Mod_B
\]
 is called a \emph{fiber functor} or \emph{flat covering} if it is faithful, flat\footnote{A functor is flat if and only if its extension to ind-objects is left exact. For equivalent characterizations see for example \cite[\S 3.1]{SCHAEPPI_INDABELIAN}.}, and right exact.

 The category $\ca{A}$ is called \emph{weakly Tannakian} if:
\begin{enumerate}
 \item[(i)] There exists a fiber functor $w \colon \ca{A} \rightarrow \Mod_B$ for some commutative $R$-algebra $B$;
\item[(ii)] For all objects $A \in \ca{A}$ there exists an epimorphism $A^{\prime} \rightarrow A$ such that $A^{\prime}$ has a dual. 
\end{enumerate}
\end{dfn}

 The two main results about weakly Tannakian categories from \cite{SCHAEPPI_STACKS, SCHAEPPI_INDABELIAN} give the following theorem.

\begin{thm}\label{thm:equivalence_adams_weakly_tannakian}
 The pseudofunctor
\[
 \QCoh_{\fp}(-) \colon \ca{AS}^{\op} \rightarrow \ca{RM}
\]
 is 2-fully faithful and it induces a contravariant biequivalence between $\ca{AS}$ and the full sub-2-category of $\ca{RM}$ consisting of weakly Tannakian categories.
\end{thm}

\begin{proof}
 The first claim is the content of \cite[Theorem~1.2.1]{SCHAEPPI_STACKS} and the second claim follows from \cite[Theorem~1.6]{SCHAEPPI_INDABELIAN}, which shows that $\ca{A}$ is weakly Tannakian if and only if there exists an Adams stack $X$ and an equivalence $\ca{A} \simeq \QCoh_{\fp}(X)$ in $\ca{RM}$. 
\end{proof}

 The main aim of \S \ref{section:affine_coverings} is to provide necessary and sufficient conditions for the existence of fiber functors.

 For Tannakian categories over a field of characteristic zero, this problem was studied by Deligne, see \cite[\S 7]{DELIGNE}. In \cite{SCHAEPPI_GEOMETRIC}, Deligne's argument was generalized to the context of geometric categories (respectively weakly Tannakian categories) in the case where the ground ring $R$ is a $\mathbb{Q}$-algebra. In this section we give variants of this result that work for arbitrary base rings. Although the conditions are not as tractable in positive characteristic as they are in characteristic zero, they can be used to show that limits of weakly Tannakian categories---and hence colimits of Adams stacks---exist over arbitrary base rings.
 
 The argument we present also allows us to to generalize \cite[Theorem~1.4]{SCHAEPPI_GEOMETRIC}: we show that one of the necessary conditions for the existence of a fiber functors listed there is implied by the others (see Theorem~\ref{thm:qalg_geometric_characterization} below).

 As already observed in \cite{SCHAEPPI_GEOMETRIC}, it is convenient to shift our perspective from right exact symmetric monoidal categories to their tensor categories of ind-objects.

 We call a tensor functor $F \colon \ca{C} \rightarrow \ca{D}$ a \emph{covering} (or a \emph{flat covering}) if $F$ is conservative (that is, it reflects isomorphisms) and it preserves finite limits. In this situation we say that $\ca{D}$ covers $\ca{C}$ or that $F$ is a covering of $\ca{C}$ by $\ca{D}$. This reversal stems from the fact that the passage from geometric objects to their associated tensor categories is contravariant.

 An example of a covering is given by the base change functor
\[
 (-)_A \colon \ca{C} \rightarrow \ca{C}_A
\]
 for a faithfully flat commutative algebra $A \in \ca{C}$.

\begin{dfn}\label{dfn:geometric_tensor_category}
 Let $\ca{C}$ be a tensor category over $R$. We call $\ca{C}$ \emph{pre-geometric} if it is lfp, abelian, and generated by objects with duals. We call it \emph{geometric} if, in addition, there exists a covering
\[
 \ca{C} \rightarrow \Mod_B
\]
 for some commutative $R$-algebra $B$.
\end{dfn}

 The following proposition shows that geometric tensor categories are in a precise sense equivalent to weakly Tannakian categories.

\begin{prop}\label{prop:geometric_weakly_tannakian}
 A right exact symmetric monoidal category $\ca{A}$ is weakly Tannakian if and only if $\Ind(\ca{A})$ is geometric. An lfp tensor category $\ca{C}$ is geometric if and only if $\ca{C}_{\fp}$ is weakly Tannakian. These constructions define mutually inverse 2-equivalences between the 2-category of weakly Tannakian categories and the 2-category of geometric tensor categories and tensor functors between them.
\end{prop}

\begin{proof}
 This is \cite[Proposition~2.8]{SCHAEPPI_GEOMETRIC}.
\end{proof}

 This allows us to state Theorem~\ref{thm:equivalence_adams_weakly_tannakian} in terms of geometric tensor categories.

\begin{thm}\label{thm:equivalence_adams_geometric}
 The pseudofunctor which sends an Adams stack $X$ to the tensor category $\QCoh(X)$ gives a biequivalence between the 2-category $\ca{AS}$ of Adams stacks and the 2-category of geometric tensor categories and tensor functors between them.
\end{thm}

\begin{proof}
 This follows from Theorem~\ref{thm:equivalence_adams_weakly_tannakian} and Proposition~\ref{prop:geometric_weakly_tannakian}, see also \cite[Theorem~2.9]{SCHAEPPI_GEOMETRIC}.
\end{proof}

 In some sense this shows that we have simply introduced another language to talk about the same thing. It turns out that this language is better suited for constructing coverings. The main reason is that we will construct coverings of $\ca{C}$ by constructing faithfully flat commutative algebras in $\ca{C}$, and the underlying objects of these will rarely be finitely presentable.

\subsection{Existence results for coverings by module categories}\label{section:affine_coverings}

 The aim of this section is to provide necessary and sufficient conditions for the existence of a covering
\[
 \ca{C} \rightarrow \Mod_B
\]
 on a pre-geometric category $\ca{C}$ (respectively of fiber functors on $\ca{C}_{\fp}$, the category of finitely presentable objects in $\ca{C}$).

 The key ingredients are the notion of locally free object of (constant) finite rank in a general tensor category (see Definition~\ref{dfn:locally_free_objects}), and the notion of locally split epimorphism defined analogously below.

\begin{dfn}\label{dfn:locally_split_epi}
 Let $\ca{C}$ be a tensor category. A morphism $p \colon M \rightarrow N$ is called a \emph{locally split epimorphism} (respectively \emph{locally split monomorphism}) if there exists a faithfully flat commutative algebra $A \in \ca{C}$ such that $p_A \colon M_A \rightarrow N_A$ is a split epimorphism (respectively a split monomorphism) in $\ca{C}_A$. 
\end{dfn}

 Note that the fact that $A$ is faithfully flat implies that a locally split epimorphism is in particular an epimorphism.

\begin{example}\label{example:locally_split_in_geom_cat}
 Let $\ca{C}$ be a geometric category, that is, $\ca{C}\simeq \QCoh(X)$ for some Adams stack $X$. Then all epimorphisms in $\ca{C}$ whose target is a dual are locally split. Indeed, since $X$ is an Adams stack there exists an affine covering $f \colon \Spec(B) \rightarrow X$, and $f^{\ast} p$ is a split epimorphism since $f^{\ast}$ sends objects with duals to finitely generated projective $B$-modules (here we use the equivalence $\QCoh\bigl(\Spec(B)\bigr) \simeq \Mod_B$). The claim follows since there is an equivalence $\Mod_B \simeq \ca{C}_{f_\ast B}$ which is compatible with $f^{\ast}$ and $(-)_{f_\ast B}$ (see \cite[Proposition~3.9]{SCHAEPPI_GEOMETRIC}), and $f_\ast B$ is faithfully flat since it is an Adams algebra, see Example~\ref{example:adams_hopf_alg_is_adams_alg}.
\end{example}

 In fact, slightly more is true: the locally split epimorphisms of the above example are preserved by \emph{any} tensor functor under mild conditions on the target.

\begin{lemma}\label{lemma:split_epi_preserved}
 Let $\ca{C}$ be a geometric tensor category, $p \colon M \rightarrow N$ an epimorphism in $\ca{C}$, and $F \colon \ca{C} \rightarrow \ca{D}$ a tensor functor. If $N$ has a dual and filtered colimits in $\ca{D}$ are exact, then $Fp$ is a locally split epimorphism.
\end{lemma}

\begin{proof}
 By Example~\ref{example:locally_split_in_geom_cat} above, there exists an Adams algebra $A \in \ca{C}$ such that $p_A$ is split. By extending the tensor functor $F$ to $A$-modules, we obtain a diagram
\[
 \xymatrix{ \ca{C} \ar[d]_{(-)_A} \ar[r]^-{F} & \ca{D} \ar[d]^{(-)_{FA}} \\
 \ca{C}_A \ar[r]_-{\overline{F}} & \ca{D}_{FA}  }
\]
 which commutes up to natural isomorphism. Thus $(Fp)_{FA}$ is a split epimorphism. The claim follows since $F$ sends Adams algebras to faithfully flat algebras by Proposition~\ref{prop:adams_strongly_stable}.
\end{proof}

 The following results contain the purely structural aspects of the version of Deligne's argument given in \cite{SCHAEPPI_GEOMETRIC}. These are in particular independent of the characteristic of the ground ring $R$. We first recall the following lemma.

\begin{lemma}\label{lemma:filtered_ffl}
 Let $\ca{C}$ be a pre-geometric tensor category. Then filtered colimits of faithfully flat algebras are again faithfully flat.
\end{lemma}

\begin{proof}
 This follows from the fact that $\ca{C}$ has enough flat resolutions, see \cite[Lemma~5.7]{SCHAEPPI_GEOMETRIC}.
\end{proof}

\begin{prop}\label{prop:locally_split_implies_projective}
 Let $\ca{C}$ be a pre-geometric tensor category. Let $\ca{G}$ be a generating set of $\ca{C}$, closed under finite direct sums. If all epimorphisms $M \rightarrow \U$ with $M \in \ca{G}$ are locally split, then there exists a faithfully flat algebra $B \in \ca{C}$ such that $B$ is projective as a $B$-module.
\end{prop}

\begin{proof}
 The proof given in \cite[Proposition~5.8]{SCHAEPPI_GEOMETRIC} does not depend on the fact that $R$ is a $\mathbb{Q}$-algebra. Since this argument is central for our existence results for coverings by module categores, we supply an outline below. More details can be found in the proof of \cite[Proposition~5.8]{SCHAEPPI_GEOMETRIC}.

 By assumption, for every epimorphism $p\colon M \rightarrow \U$ with $M \in \ca{G}$ there exists a faithfully flat algebra $B_p$ such that $p_{B_p}$ is a split epimorphism. Let $B$ be the filtered colimit of the finite tensor products of the these algebras $B_p$. Since finite tensor products of faithfully flat algebras are faithfully flat, Lemma~\ref{lemma:filtered_ffl} implies that $B$ is faithfully flat. By construction, $p_B$ is a split epimorphism for all $p \colon M \rightarrow \U$ with $M \in \ca{G}$. We claim that $B$ is projective as a $B$-module.

 First note that---by using pullbacks---it suffices to show that every epimorphism $q \colon N \rightarrow B$ in $\ca{C}_B$ has a section. As in the proof of \cite[Proposition~5.8]{SCHAEPPI_GEOMETRIC} (see also Lemma~\ref{lemma:epis_in_lfp_abelian}), we can find $M \in \ca{G}$ and an epimorphism $p \colon M \rightarrow \U$ in $\ca{C}$ such that the diagram
\[
 \xymatrix{M \ar[r] \ar[d]_{p} & N \ar[d]^{q} \\ \U \ar[r]_-{\eta} & B }
\]
 is commutative. Under the free-forgetful adjunction $(-)_B \dashv U \colon \ca{C} \rightarrow \ca{C}_B$, the above diagram corresponds to the commutative diagram
\[
 \xymatrix{M_B \ar[r] \ar[d]_{p_B} & N \ar[d]^{q} \\ B \ar@{=}[r] & B } 
\]
 in $\ca{C}_B$. But $p_B$ has a splitting by construction of $B$. This shows that $B \in \ca{C}_B$ is indeed projective.
\end{proof}

\begin{prop}\label{prop:locally_free_implies_generator}
 Let $\ca{C}$ be a pre-geometric tensor category. If $\ca{C}$ is generated by locally free objects of constant finite rank, then there exists a faithfully flat algebra $A \in \ca{C}$ such that $A$ is a generator of $\ca{C}_A$.
\end{prop}

\begin{proof}
 Let $\ca{G}$ be a generating set of locally free objects of finite rank. Thus for each $M \in \ca{G}$ there exists a faithfully flat algebra $A$ and $d \in \mathbb{N}$ such that $M_A \cong A^d$. As in the proof of Proposition~\ref{prop:locally_split_implies_projective}, we can find a single faithfully flat algebra $A$ which has this property for all $M \in \ca{G}$ simultaneously (where the rank $d=d_M$ will of course depend on the object $M$). The claim follows from the fact that $\{M_A\vert M \in \ca{G} \}$ is a generator of $\ca{C}_A$.
\end{proof}

\begin{cor}\label{cor:geometric_structural_proof}
 Let $\ca{C}$ be an lfp abelian tensor category which has a generating set $\ca{G}$, closed under finite direct sums, which consists of locally free objects of finite rank. If each epimorphism $M \rightarrow \U$ with $M \in \ca{G}$ is locally split, there exists a covering
\[
 \ca{C} \rightarrow \Mod_B
\]
 for some commutative $R$-algebra $B$. Thus $\ca{C}$ is geometric and there exists an Adams stack $X$ and an equivalence $\ca{C} \simeq \QCoh(X)$ of tensor categories.
\end{cor}

\begin{proof}
 Note that $\ca{C}$ is pre-geometric since locally free objects have duals. From Propositions~\ref{prop:locally_split_implies_projective} and \ref{prop:locally_free_implies_generator} it follows that there exists a faithfully flat algebra $A$ such that $A$ is a projective generator of $\ca{C}_A$. Thus $\ca{C}_A \simeq \Mod_B$ for some commutative algebra $B$ (in the terminology of \cite{SCHAEPPI_GEOMETRIC}: $A$ is an affine algebra). The composite of the functor
\[
 (-)_A \colon \ca{C} \rightarrow \ca{C}_A
\]
 with this equivalence gives the desired covering. The second claim follows from Theorem~\ref{thm:equivalence_adams_geometric}.
\end{proof}

 Note that the above results also have implications for tensor categories that are not geometric.

\begin{cor}\label{cor:hopf_monoidal_comonad}
 Let $\ca{C}$ be a pre-geometric tensor category such that all epimorphisms $M \rightarrow \U$ where $M$ has a dual are locally split. Then there exists a small symmetric monoidal $R$-linear category $\ca{B}$ with duals and an exact cocontinuous symmetric Hopf monoidal comonad
\[
 H \colon \Prs{B} \rightarrow \Prs{B}
\]
 (where the presheaf category $\Prs{B}$ is endowed with the Day convolution symmetric monoidal structure) such that $\ca{C}$ is equivalent to the category $\Comod(H)$ of $H$-comodules in $\Prs{B}$.
\end{cor}

\begin{proof}
 We claim that for any faithfully flat algebra $B \in \ca{C}$, the functor
\[
(-)_B \colon \ca{C} \rightarrow \ca{C}_B 
\]
 is comonadic, and the comonad on $\ca{C}_B$ is exact, cocontinuous, and symmetric Hopf monoidal. Exactness and comonadicity is immediate from the fact that $B$ is faithfully flat and from (the dual of) Beck's monadicity theorem. The comonad
\[
 \xymatrix{\ca{C}_B \ar[r]^-{U} & \ca{C} \ar[r]^-{(-)_B} & \ca{C}_B}
\]
 (where $U$ denotes the forgetful functor) is clearly cocontinuous and symmetric (lax) monoidal. It remains to check that it is Hopf monoidal.

 The adjunction $(-)_B \dashv B$ satisfies the projection formula (it is \emph{strong coclosed} in the sense of \cite{CHIKHLADZE_LACK_STREET}) by \cite[Proposition~3.8]{SCHAEPPI_GEOMETRIC}. Thus the comonad is Hopf monoidal by \cite[Proposition~4.4]{CHIKHLADZE_LACK_STREET}.

 To conclude the proof, it suffices to show that---for a suitably chosen faithfully flat algebra $B$---the category $\ca{C}_B$ is equivalent to $\Prs{B}$ as tensor category, where $\ca{B}$ is a small symmetric monoidal $R$-linear category $\ca{B}$ with duals. 

 Let $B$ be a faithfully flat algebra such that $B$ is projective as a $B$-module (such an algebra exists by Proposition~\ref{prop:locally_split_implies_projective}). Let $\ca{B} \subseteq \ca{C}_B$ be the full subcategory consisting of objects with a dual. We claim that the objects in $\ca{B}$ are \emph{small projective} in the sense of \cite[\S 5.5]{KELLY_BASIC}, that is, the hom-functors $\ca{C}_B(M,-)$ are cocontinuous for all $M \in \ca{B}$. Since the objects have duals, it suffices to show this for $M=B$, the unit object. The hom-functor $\ca{C}_B(B,-)\cong \ca{C}(\U,-)$ preserves filtered colimtis (recall that the unit object $\U$ of $\ca{C}$ is finitely presentable by assumption). Since $B$ is projective as a $B$-module, the left exact functor $\ca{C}_B(B,-)$ is in fact exact. Since any exact functor which preserves all filtered colimits is cocontinuous, it follows that $B$ (and hence all objects with duals in $\ca{C}_B$) are small projective.

 From the proof of \cite[Theorem~5.26]{KELLY_BASIC} it follows that the left Kan extension of the inclusion $K \colon \ca{B} \rightarrow \ca{C}_B$ induces an equivalence of categories $\Lan_Y K \colon \Prs{B} \rightarrow \ca{C}_B$. But this functor is symmetric strong monoidal by Theorem~\ref{thm:day_convolution_universal}, so we do indeed have the desired equivalence of tensor categories.
\end{proof}

\subsection{Recognizing locally split epimorphisms}\label{section:locally_split_epis}

 In order to use the corollaries of \S \ref{section:affine_coverings}, we need to find good criteria for epimorphisms $p\colon M \rightarrow \U$ to be locally split. In Deligne's original argument for the case of Tannakian categories in characteristic zero, the following construction plays a crucial role.

\begin{dfn}\label{dfn:universal_splitting_algebra}
 Let $\ca{C}$ be a tensor category, let $M \in \ca{C}$ be an object with a dual, and let $p \colon M \rightarrow \U$ be an epimorphism. The morphism
\[
 \xymatrix{\U \cong \U^{\vee} \ar[r]^-{p^{\vee}} & M^{\vee}=\Sym^1(M^{\vee}) \ar[r] & \Sym(M^{\vee})}
\]
 induces a homomorphism of commutative algebras $\Sym(\U) \rightarrow \Sym(M^{\vee})$ in $\ca{C}$. Let
\[
 \xymatrix{\Sym(\U) \ar[r] \ar[d] & \Sym(M^{\vee}) \ar[d] \\ \U \ar[r] & A_p }
\]
 be a pushout diagram in the category of commutative algebras in $\ca{C}$, where the morphism $\Sym(\U) \rightarrow \U$ corresponds to the identity on $\U$. The composite
\[
 \xymatrix{M^{\vee} \ar[r] & \Sym(M^{\vee}) \ar[r] & A_p}
\]
 corresponds to a morphism $\tau \colon M^{\vee}_{A_p} \rightarrow A_p$ of $A_p$-modules. We let $\sigma \colon \U \rightarrow M_{A_p}$ denote the morphism in $\ca{C}$ which corresponds to the dual $\tau^{\vee} \colon A_p \rightarrow M_{A_p}$ of $\tau$.
\end{dfn}

 This construction has the following universal property.

\begin{prop}\label{prop:universal_splitting_algebra}
 The algebra $A_p$ in Definition~\ref{dfn:universal_splitting_algebra} is the universal algebra endowed with a splitting of $p_{A_p}$. More precisely: given a commutative algebra $B$ and a morphism $s \colon \U \rightarrow M_B$ which corresponds to a splitting $ \overline{s} \colon B \rightarrow M_B$ of $p_B \colon M_B \rightarrow B$, there exists a unique ring homomorphism $\varphi \colon A_p \rightarrow B$ such that $s=\varphi \otimes M \circ \sigma$.
\end{prop}

\begin{proof}
 By construction, to give a morphism $\varphi \colon A_p \rightarrow B$ of commutative algebras amounts to giving a morphism $\overline{s}^{\vee} \colon M^{\vee}_B \rightarrow B$ in $\ca{C}_B$ such that $\overline{s}^{\vee} p_B^{\vee}=\id_B$, and the conclusion follows by taking duals.
\end{proof}

 The following lemma gives a more concrete description of the algebra $A_p$. It can essentially be found in \cite[\S 7]{DELIGNE} and the proof of \cite[Proposition~5.6]{SCHAEPPI_GEOMETRIC}.

\begin{lemma}\label{lemma:adams_presentation}
 Let $\ca{C}$ be a tensor category, $M \in \ca{C}$ an object with dual, and $p \colon M \rightarrow \U$ a morphism. Let $A_p$ denote the algebra from Definition~\ref{dfn:universal_splitting_algebra}.

 For $i \in \mathbb{N}$, let $\eta_i$ denote the composite
\[
 \xymatrix@C=50pt{\U \cong \Sym^{i}(\U^{\vee}) \ar[r]^-{\Sym^{i}(p^{\vee})} & \Sym^i(M^{\vee}) }
\]
 in $\ca{C}$ and let $f_i \colon \Sym^{i}(M^{\vee}) \rightarrow \Sym(M^{\vee}) \rightarrow A_p$ denote the canonical morphism which exists by Definition of $A_p$. Let $f_{i,i+1}$ be the morphism making the diagram
\[
 \xymatrix{\Sym^{i} (M^{\vee}) \ar[d]_{\cong} \ar[r]^-{f_{i,i+1}} & \Sym^{i+1}(M^{\vee})\\
\U^{\vee} \otimes \Sym^{i}(M^{\vee}) \ar[r]_-{p^{\vee} \otimes \id} & M^{\vee} \otimes \Sym^{i}(M^{\vee}) \ar[u]_{\mu} }
\]
 in $\ca{C}$ commutative, where $\mu$ is the multiplication of $\Sym(M^{\vee})$. Then $f_{i,i+1} \circ \eta_i=\eta_{i+1}$, and the $f_i$ exhibit $A_p$ as colimit of the chain
\begin{equation}\label{eqn:splitting_algebra_chain}
 \xymatrix{ \U \ar[r]^-{f_{01}} & M^{\vee} \ar[r]^-{f_{12}} & \Sym^2(M^{\vee}) \ar[r] & \ldots }
\end{equation}
 in $\ca{C}$.
\end{lemma}

\begin{proof}
 That the $f_i$ exhibit $A_p$ as colimit of the chain of Diagram~\eqref{example:locally_split_in_geom_cat} was already used in \cite[\S 7]{DELIGNE}, see also \cite[Lemma~5.14]{SCHAEPPI_GEOMETRIC} for details.

 The composite $f_{i-1,i} \circ \ldots \circ f_{01}$ is given by
\[
 \xymatrix@C=40pt{\U \cong (\U^{\vee})^{\otimes i} \ar[r]^-{(p^{\vee})^{\otimes i}} & (M^{\vee})^{\otimes i} \ar[r]^-{\pi_i} & \Sym^i(M^{\vee}) } \smash{\rlap{,}}
\]
 where $\pi_i$ denotes the canonical projection. Since the $\pi_i$ give a morphism of algebras from the tensor algebra to the symmetric algebra, this composite is equal to
\[
 \xymatrix@C=40pt{\U \cong {(\U^{\vee})}^{\otimes i} \ar[r]^{\cong} & \Sym^i(\U^{\vee})   \ar[r]^-{\Sym^i(p^{\vee})} & \Sym^i(M^{\vee}) } \smash{\rlap{,}}
\]
 which shows that $f_{i-1,i} \circ \ldots \circ f_{01}=\eta_{i}$. The desired equality $f_{i,i+1} \circ \eta_i =\eta_{i+1}$ follows easily from this.
\end{proof}

 With this lemma in hand, we can prove the following proposition, which gives useful criteria for an epimorphism $p \colon M \rightarrow \U$ whose domain has a dual to be locally split.

\begin{prop}\label{prop:locally_split_criteria}
 Let $\ca{C}$ be an lfp abelian tensor category. Let $M \in \ca{C}$ be an object such that $\Sym^i(M^{\vee})$ has a dual for all $i \in \mathbb{N}$. Then the following are equivalent for an epimorphism $p \colon M \rightarrow \U$:
\begin{enumerate}
 \item[(i)] There exists an Adams algebra $A$ such that $p_A \colon M_A \rightarrow A$ is a split epimorphism in $\ca{C}_A$;
\item[(ii)] The morphism $p$ is a locally split epimorphism;
\item[(iii)] There exists a tensor functor $F \colon \ca{C} \rightarrow \ca{D}$ which reflects epimorphisms between objects with duals such that $F(p)$ is a split epimorphism;
\item[(iv)] The duals of the morphisms $\Sym^i(p^{\vee}) \colon \Sym^i(\U^{\vee}) \rightarrow \Sym^{i}(M^{\vee})$ are epimorphisms;
\item[(v)] The algebra $A_p$ from Definition~\ref{dfn:universal_splitting_algebra} is an Adams algebra;
\item[(vi)] For all $i \in \mathbb{N}$, the morphism
\[
 (p^{\otimes i})^{\Sigma_i} \colon (M^{\otimes i})^{\Sigma_i} \rightarrow (\U^{\otimes i})^{\Sigma_i}
\]
 induced by $p$ on fixed points of the symmetric group actions is an epimorphism. 
\end{enumerate}
\end{prop}

\begin{proof}
 We will show the implications
\[
 \text{(i)} \Rightarrow \text{(ii)} \Rightarrow \text{(iii)} \Rightarrow \text{(iv)} \Rightarrow \text{(v)} \Rightarrow \text{(i)}
\]
 and the equivalence $\text{(iv)} \Leftrightarrow \text{(vi)}$.

 Since Adams algebras are faithfull flat (see Proposition~\ref{prop:adams_faithfully_flat}), we have $\text{(i)} \Rightarrow \text{(ii)}$. To see that $\text{(ii)} \Rightarrow \text{(iii)}$, note that any faithfully flat algebra $B$ such that $p_B$ is a split epimorphism induces a functor $F=(-)_B \colon \ca{C} \rightarrow \ca{C}_B$ with the desired properties. For the implication $\text{(iii)} \Rightarrow \text{(iv)}$ it suffices to show that the images under $F$ of the morphisms in question are epimorphisms. Since $F$ commutes with duals and symmetric powers, this reduces the problem to checking that $\Sym^{i}\bigl(F(p)^{\vee}\bigr)^{\vee}$ is an epimorphism. But $F(p)$ is a split epimorphism, and any functor---in particular $\Sym^i\bigl((-)^{\vee}\bigr)^{\vee}$---preserves split epimorphisms. Condition~(iv) and Lemma~\ref{lemma:adams_presentation} show that the morphisms
\[
 \eta_i \colon \U \cong \Sym^i(\U^{\vee}) \rightarrow \Sym^i(M^{\vee})
\]
 defined in Lemma~\ref{lemma:adams_presentation} exhibit $A_p$ as Adams algebra. This shows that $\text{(iv)} \Rightarrow \text{(v)}$, and $\text{(v)} \Rightarrow \text{(i)}$ follows from the fact that $p_{A_p}$ is a split epimorphism (see Proposition~\ref{prop:universal_splitting_algebra}).

 It remains to check that $\text{(iv)} \Leftrightarrow \text{(vi)}$. For any $N \in \ca{C}$ (with a dual) we have isomorphisms
\[
 \Sym^i(N)^{\vee}=[\Sym^{i}(N),\U]=[N^{\otimes i}\slash {\Sigma_i}, \U] \cong[N^{\otimes i},\U]^{\Sigma_i} \cong \bigl((N^{\vee})^{\otimes i} \bigr)^{\Sigma_i}
\]
 which are natural in $N$. This shows that the duals of the morphisms in Condition~(iv) are isomorphic to the morphisms of Condition~(vi), hence that one of these sets consists entirely of epimorphisms if and only if the other one does.
\end{proof}

\begin{cor}\label{cor:qalg_epis_locally_split}
 Let $R$ be a $\mathbb{Q}$-algebra, let $\ca{C}$ be an lfp abelian tensor category over $R$, and let $M \in \ca{C}$ be an object with a dual. Then any epimorphism $p \colon M \rightarrow \U$ is a locally split epimorphism.
\end{cor}

\begin{proof}
 The statement is evident if $R=0$, hence we may assume that $\mathbb{Q} \subseteq R$. In that case, we can divide by $i!$, so $\Sym^{i}(M^{\vee})$ is  a direct summand of $(M^{\vee})^{\otimes i}$ and thus has a dual. Therefore we can apply Proposition~\ref{prop:locally_split_criteria} above, and it suffices to check that the morphisms
\[
 (p^{\otimes i})^{\Sigma_i} \colon (M^{\otimes i})^{\Sigma_i} \rightarrow (\U^{\otimes i})^{\Sigma_i} 
\]
 are epimorphisms. This follows from the fact that the morphisms
\[
 N^{\otimes i} \slash {\Sigma_i} \rightarrow (N^{\otimes i})^{\Sigma_i}
\]
 induced by $\sum_{\sigma \in \Sigma_i} \sigma \colon N^{\otimes i} \rightarrow N^{\otimes i}$ are inverse to the composite
\[
 (N^{\otimes i})^{\Sigma_i} \rightarrow N^{\otimes i} \rightarrow N^{\otimes i} \slash {\Sigma_i}
\]
 up to a factor of $i! \in \mathbb{Q} \subseteq R$.
\end{proof}

\begin{cor}\label{cor:qalg_free_generators}
 Let $R$ be a $\mathbb{Q}$-algebra and let $\ca{C}$ be an lfp abelian tensor category over $R$. If $\ca{C}$ is generated by locally free objects of constant rank, then $\ca{C}$ is geometric.
\end{cor}

\begin{proof}
 This follows from Corollary~\ref{cor:geometric_structural_proof}, taking into account that the relevant epimorphisms are locally split by Corollary~\ref{cor:qalg_epis_locally_split} above.
\end{proof}

 We can further specialize this to tensor categories generated by line bundles.

\begin{dfn}\label{dfn:line_bundle}
 Let $\ca{C}$ be a tensor category. An invertible object $L \in \ca{C}$ is called a \emph{line bundle} if the symmetry $s_{L,L} \colon L\otimes L \rightarrow L \otimes L$ is equal to the identity.
\end{dfn}

\begin{cor}\label{cor:qalg_line_bundle_generator}
 Let $R$ be a $\mathbb{Q}$-algebra and let $\ca{C}$ be an lfp abelian tensor category over $R$. If $\ca{C}$ is generated by line bundles, then $\ca{C}$ is geometric.
\end{cor}

\begin{proof}
 It suffices to check that any line bundle $L \in \ca{C}$ is locally free of rank one. The algebra $A=\bigoplus_{i \in \mathbb{Z}} L^{\otimes i}$ is commutative since the $\Sigma_i$-action on $L^{\otimes i}$ and its inverse are trivial. It is flat as direct sum of objects with duals, and thus faithfully flat since the unit $\U \rightarrow A$ is a split monomorphism. Moreover, we have $L_A \cong A$ as $A$-modules, that is, $L$ is locally free of rank one (see also \cite[Proposition~4.9.9]{BRANDENBURG_THESIS}).
\end{proof}

 Combining this with one of the main results of \cite{SCHAEPPI_GEOMETRIC}, we also get the following intrinsic characterization of geometric tensor categories over $\mathbb{Q}$-algebras. In order to state it, we need the concepts of rank and exterior powers in general tensor categories. The \emph{rank} of an object with a dual is the trace of the identity morphism, that is, the composite
\[
 \xymatrix{\U \ar[r] & X \otimes X^{\vee} \ar[r]^{\cong} & X^{\vee} \otimes X \ar[r] & \U }
\]
 of the coevaluation and the evaluation of $X$ in the endomorphism ring $\ca{C}(\U,\U)$ of the unit object. The $i$-th \emph{exterior power} $\Lambda^i(X)$ of $X$ is the splitting of the idempotent
\[
 \frac{1}{i!} \sum_{\sigma \in \Sigma_i} \sgn(\sigma) \sigma \colon X^{\otimes i} \rightarrow X^{\otimes i}
\]
 on $X^{\otimes i}$.

\begin{thm}\label{thm:qalg_geometric_characterization}
 Let $R$ be a $\mathbb{Q}$-algebra and let $\ca{C}$ be an lfp abelian tensor category over $R$. Then $\ca{C}$ is geometric if and only if the following conditions hold:
\begin{enumerate}
 \item[(i)] The category $\ca{C}$ is generated by duals;
\item[(ii)] If $X$ has a dual and $\rk(X)=0$, then $X\cong 0$;
\item[(iii)] For all objects $X$ with a dual there exists an $i \in \mathbb{N}$ such that $\Lambda^i(X) \cong 0$.
\end{enumerate}
\end{thm}

\begin{proof}
 By Corollary~\ref{cor:qalg_free_generators}, it suffices to check that $\ca{C}$ is generated by locally free objects if Conditions~(i)-(iii) hold. This is precisely the content of the proof of \cite[Proposition~5.12]{SCHAEPPI_GEOMETRIC}.
\end{proof}

 If we translate this to the language of weakly Tannakian categories we obtain a proof of Theorem~\ref{thm:existence_of_fiber_functor}.

\begin{proof}[Proof of Theorem~\ref{thm:existence_of_fiber_functor}]
 Using the fact that $\ca{A}$ is weakly Tannakian if and only if $\Ind(\ca{A})$ is geometric (see Proposition~\ref{prop:geometric_weakly_tannakian}), Theorem~\ref{thm:existence_of_fiber_functor} is an immediate consequence of Theorem~\ref{thm:qalg_geometric_characterization}.
\end{proof}

 If $R$ is not a $\mathbb{Q}$-algebra, it is not clear that the functor which takes fixed points of a group action preserves epimorphisms. In particular, it is not clear that the conclusion of Corollary~\ref{cor:qalg_epis_locally_split} holds in general. Therefore we need to include an additional condition in the following description result for geometric categories over arbitrary ground rings. Despite being less tractable than the above theorem and corollaries, it suffices for the universal constructions of geometric categories that we describe in the next section.

\begin{thm}\label{thm:description}
 Let $R$ be a an arbitrary commutative ring and let $\ca{C}$ be an lfp abelian tensor category over $R$. Then $\ca{C}$ is geometric if and only if the following hold:
\begin{enumerate}
 \item[(i)] There exists a generating set $\ca{G}$ of $\ca{C}$ consisting of locally free objects of constant rank.
 \item[(ii)] If $p \colon M \rightarrow \U$ is an epimorphism and $M$ is a finite direct sum of objects in the generating set $\ca{G}$, then the induced morphisms
\[
 (p^{\otimes i})^{\Sigma_i} \colon (M^{\otimes i})^{\Sigma_i} \rightarrow (\U^{\otimes i})^{\Sigma_i}
\]
 on the $\Sigma_i$-fixed points are epimorphisms for all $i \in \mathbb{N}$.
\end{enumerate}
\end{thm}

\begin{proof}
 Note that locally free objects of constant rank are closed under finite direct sums. To apply Proposition~\ref{prop:locally_split_criteria}, it therefore suffices to check that the symmetric powers of a locally free objects of constant rank $d$ have duals. But any such object is the image of the standard representation $\ubar{R}^d \in \Rep(\mathrm{GL}_d)$ under some tensor functor $\Rep(\mathrm{GL}_d) \rightarrow \ca{C}$ (see Theorem~\ref{thm:universal_property_rep_gld}). This reduces the problem to checking that $\Sym^{i}(\ubar{R}^d)$ has a dual, which follows from the fact that its underlying $R$-module is finitely generated and free. Thus we can indeed apply Proposition~\ref{prop:locally_split_criteria}. Part~(vi) of that proposition and the first assumption show that the conditions of Corollary~\ref{cor:geometric_structural_proof} are satisfied.
\end{proof}

 We conclude this section with a version of Proposition~\ref{prop:locally_split_criteria} that works for tensor categories with exact filtered colimits which need not be abelian. This lemma will be used in a sequel in order to give a more tractable description of certain universal weakly Tannakian categories.

\begin{lemma}\label{lemma:detecting_locally_split}
 Let $\ca{C}$ be a tensor category with exact filtered colimits. Let
\[
 p \colon M \rightarrow \U 
\]
 be a morphism whose domain is a retract of a locally free object of constant finite rank. Let $\ca{D}$ be a tensor category and $F \colon \ca{C} \rightarrow \ca{D}$ a tensor functor which detects those coequalizer diagrams whose entries have duals.

 In this situation, if $Fp$ is a locally split epimorphism, there exists an Adams algebra $A \in \ca{C}$ such that $p_A$ is a split epimorphism. 
\end{lemma}

\begin{proof}
 By Proposition~\ref{prop:universal_splitting_algebra} it suffices to check that the algebra $A_p$ defined there is an Adams algebra (any Adams algebra in $\ca{C}$ is faithfully flat since filtered colimits in $\ca{C}$ are exact, see Proposition~\ref{prop:adams_strongly_stable}).

 The object $\Sym^i(M^{\vee})$ has a dual: it is a retract of $\Sym^i(N)$ for some locally free object $N$ of constant finite rank, and $\Sym^i(N)$ is the image of some some locally free object of constant rank in $\Rep(\mathrm{GL}_d)$ by Theorem~\ref{thm:universal_property_rep_gld}. Thus the objects $A_i \defl \Sym^i(M^{\vee})$ and the morphisms $\eta_i \colon \U \rightarrow A_i$ given by the composite
\[
 \xymatrix@C=50pt{\U \cong \Sym^{i}(\U^{\vee}) \ar[r]^-{\Sym^{i}(p^{\vee})} & \Sym^i(M^{\vee})} 
\]
 give a ``candidate'' collection for exhibiting $A_p$ as Adams algebra (see Lemma~\ref{lemma:adams_presentation}).

 It only remains to check that
\begin{equation}\label{eqn:adams_coequalizer}
   \xymatrix{ A_i^{\vee} \otimes A_i^{\vee} \ar@<0.5ex>[r]^-{A_i^{\vee} \otimes \eta_i^{\vee}} \ar@<-0.5ex>[r]_-{\eta_i^{\vee} \otimes A_i^{\vee}} & A_i^{\vee} \ar[r]^-{\eta_i^{\vee}} & \U }
\end{equation}
 is a coequalizer diagram in $\ca{C}$. 

 By assumption, it suffices to check that the image of this diagram under $F$ is a coequalizer diagram in $\ca{D}$. Moreover, by composing $F \colon \ca{C} \rightarrow \ca{D}$ with any faithfully flat base change functor $(-)_B \colon \ca{D} \rightarrow \ca{D}_B$, we obtain a new tensor functor with the same properties as $F$. Therefore we can assume that $Fp$ is a split epimorphism. It follows that $F\eta_i^{\vee} \cong \Sym^i(Fp^{\vee})^{\vee}$ has a section $s \colon F\U \rightarrow FA_i$.

 Since $s$ and $FA_i \otimes s$ exhibit the image of Diagram~\eqref{eqn:adams_coequalizer} as a split coequalizer\footnote{Split coequalizers are the dual notion to split equalizers.}, the assumption on $F$ implies that Diagram~\eqref{eqn:adams_coequalizer} is a coequalizer diagram, as claimed.
\end{proof}

\subsection{The construction}\label{section:universal_geometric_cats}

 In this section we prove that any lfp tensor category $\ca{C}$ (which need \emph{not} be abelian) has an associated geometric tensor category $\ca{G}$, together with a tensor functor $\ca{G} \rightarrow \ca{C}$ which is universal: any other tensor functor $\ca{G}^{\prime} \rightarrow \ca{C}$ whose domain $\ca{G}^{\prime}$ is a geometric tensor category factors essentially uniquely through $\ca{G} \rightarrow \ca{C}$.

 In a general $R$-linear category with cokernels, the only sensible definition of a right exact sequence is a pair of morphisms $p \colon L \rightarrow M$ and $q \colon M \rightarrow N$ such that $q$ is a cokernel of $p$. If the category also has kernels, then $q$ is easily seen to be the cokernel of its kernel $K \rightarrow M$. However, it need \emph{not} be true in general that the induced morphism $L \rightarrow K$ is itself a cokernel (also known as \emph{regular} epimorphism). Indeed, taking $p$ to be any epimorphism which is not regular and $N=0$ gives an example. In this context, there is thus some ambiguity: we could demand that a right exact sequence is a sequence
\[
 \xymatrix{L \ar[r]^{p} & M \ar[r]^{q} & N \ar[r] & 0}
\]
 such that the induced morphism from $L$ to the kernel $K$ of $q$ is a regular epimorphism. We will therefore say that the above sequence is a \emph{cokernel diagram} if $q$ is a cokernel of $p$, and that it is \emph{right exact} if, in addition, the induced morphism $L \rightarrow K$ is a regular epimorphism. That being said, most of the cokernel diagrams we consider in this section are right exact.

 \begin{dfn}\label{dfn:split_rex_sequence}
 Let $\ca{C}$ be an $R$-linear with kernels and cokernels. A sequence
\[
 \xymatrix{L \ar[r]^{p} & M \ar[r]^{q} & N \ar[r] & 0}
\]
 in $\ca{C}$ with $pq=0$ is called \emph{split right exact} if both $q$ and the induced morphism from $L$ to the kernel of $q$ are split epimorphisms. If $\ca{C}$ is an lfp tensor category over $R$, then a sequence as above with $pq=0$ is called \emph{locally split right exact} if there exists a faithfully flat commutative algebra $B \in \ca{C}$ such that the sequence
\[
 \xymatrix{L_B \ar[r]^{p} & M_B \ar[r]^{q} & N_B \ar[r] & 0} 
\]
 in $\ca{C}_B$ is split right exact.
 \end{dfn}

 The names are justified since such sequences are right exact. This is immediate for split right exact sequences, and for locally split right exact sequences it follows from the facts that $B$ is faithfully flat and that an epimorphism is regular if and only if it is the cokernel of its kernel.

\begin{example}\label{example:duals_locally_split}
 Let $\ca{C}=\QCoh(X)$ for some Adams stack $X$. Then any right exact sequence
\[
 \xymatrix{L \ar[r] & M \ar[r] & N \ar[r] & 0}
\]
 in $\ca{C}$ where both $M$ and $N$ have a dual is locally split right exact. Indeed, for $f \colon \Spec(A) \rightarrow X$ any affine covering and $B=f_\ast (A)$, we have $\ca{C}_B \simeq \Mod_A$. Thus both $M_B$ and $N_B$ are projective. It follows that the kernel $K$ of the split epimorphism $M_B \rightarrow N_B$ is projective too, hence that the epimorphism $L_B \rightarrow K$ is split.
\end{example}

\begin{rmk}\label{rmk:locally_split_rex_stable_under_tensoring}
 Locally split right exact sequences in an lfp tensor category $\ca{C}$ are preserved by $M \otimes - \colon \ca{C} \rightarrow \ca{C}$ for any object $M \in \ca{C}$. Using the fact that $(-)_B$ creates colimits for any faithfully flat $B$, we can immediately reduce this to showing that $M \otimes -$ preserves split right exact sequences. This follows since any $R$-linear functor, in particular $M \otimes -$, preserves split epimorphisms and their kernels.
\end{rmk}

\begin{dfn}\label{dfn:associated_geometric_cat}
 Let $\ca{C}$ be an lfp tensor category over an arbitrary commutative ring $R$ (not necessarily abelian), and let $\ca{A} \subseteq \ca{C}$ be a full subcategory that consists of locally free objects of constant finite rank and that is closed under finite tensor products. Let $\Sigma$ be the set\footnote{Note that locally free objects of constant finite rank are in particular finitely presentable, hence $\ca{A}$ is essentially small.} of all the cokernel diagrams in $\ca{A}$ which are locally split right exact sequences in $\ca{C}$. Let
\[
 G(\ca{A},\ca{C}) \defl \Lex_{\Sigma}[\ca{A}^{\op},\Mod_R]
\]
 be the category of $R$-linear presheaves which send all the cokernel diagrams in $\Sigma$ to kernel diagrams (that is, left exact sequences in $\Mod_R$). If there is no potential for confusion we simply write $G(\ca{A})$ for $G(\ca{A},\ca{C})$.
\end{dfn}

\begin{rmk}\label{rmk:universal_tensor_functor}
 Let $\ca{A} \subseteq \ca{C}$ be as in Definition~\ref{dfn:associated_geometric_cat} above. Since locally split right exact sequences are stable under tensoring, the category $G(\ca{A})$ inherits the structure of an lfp tensor category over $R$ via Day reflection (see Proposition~\ref{prop:day_reflection}), and the inclusion $K \colon \ca{A} \rightarrow \ca{C}$ induces a tensor functor
\[
 E \colon G(\ca{A}) \rightarrow \ca{C}
\]
 whose restriction along the Yoneda embedding $\ca{A} \rightarrow G(\ca{A})$ is isomorphic to $K$ (see Theorem~\ref{thm:lex_sigma_universal}).
\end{rmk}

 The main result of this section is the following theorem.

\begin{thm}\label{thm:universal_geometric_category}
 Let $\ca{C}$ be an lfp tensor category over an arbitrary commutative ring $R$ and let $\ca{A} \subseteq \ca{C}$ be a full subcategory consisting of locally free objects of constant finite rank. Suppose the following conditions hold:
\begin{enumerate}
 \item[(i)] The category $\ca{A}$ is closed under finite direct sums, finite tensor products, and duals;
 \item[(ii)] If $M \rightarrow N$ is a locally split epimorphism in $\ca{C}$, $N$ is locally free of constant rank, and $M \in \ca{A}$, then $N \in \ca{A}$ as well.
\end{enumerate}
 Then the category $G(\ca{A})$ of Definition~\ref{dfn:associated_geometric_cat} is a geometric tensor category.
\end{thm}

 To prove this, we will use the description result of the previous section which works over arbitrary commutative rings (Theorem~\ref{thm:description}). We face three difficulties: we have to show that $G(\ca{A})$ is abelian, that it is generated by locally free objects, and that a technical condition about the interaction of epimorphisms and certain finite limits is satisfied. For the first fact, we will use the notion of ind-class introduced in \cite[\S 2]{SCHAEPPI_INDABELIAN}, see also Definition~\ref{dfn:ind_class} and Proposition~\ref{prop:ind_class_implies_abelian}. To check the condition about epimorphisms, we have to delve deeper into the details of the proof given in \cite[\S 2]{SCHAEPPI_INDABELIAN}: from this it will follow that $G(\ca{A})$ is a category of $R$-linear sheaves for an (enriched) Grothendieck topology. In the unenriched case, there are convenient characterizations of epimorphisms in categories of sheaves, and in Appendix~\ref{appendix:sheaves} we will see that similar characterizations are possible in the enriched case.

 To start, we will show that $G(\ca{A})$ is generated by locally free objects of finite rank by showing that the representable presheaves are locally free of constant rank. Note that this is not a tautology: although the objects in $\ca{A}$ are locally free \emph{in the category $\ca{C}$}, the same need a priori not be true in the category $G(\ca{A})$, since the tensor functor $E \colon G(\ca{A}) \rightarrow \ca{C}$ of Remark~\ref{rmk:universal_tensor_functor} need not be conservative nor exact. Thus we will need to use the machinery developed in \S \ref{section:torsors} to prove this.

\begin{lemma}\label{lemma:factorization_through_F}
 Let $\ca{C}$ be an lfp tensor category over $R$ and let $\ca{A} \subseteq \ca{C}$ be a subcategory as in Theorem~\ref{thm:universal_geometric_category}. Let $\ca{G}$ be a geometric tensor category with generating set $\ca{G}_0$ that consists of locally free objects of constant finite rank and that is closed under finite direct sums and tensor products. Let $F \colon \ca{G} \rightarrow \ca{C}$ be a tensor functor with the property that $F(M)$ lies in $\ca{A}$ for all $M \in \ca{G}_0$.

 Then there exists a tensor functor $F^{\prime} \colon \ca{G} \rightarrow G(\ca{A})$ such that $EF^{\prime} \cong F$, where $E \colon G(\ca{A}) \rightarrow \ca{C}$ denotes the tensor functor defined in Remark~\ref{rmk:universal_tensor_functor}. Moreover, the diagram
\[
 \xymatrix{\ca{G}_0 \ar[rr]^-{F^{\prime}\vert_{\ca{G}_0}} \ar[rd]_{F\vert_{\ca{G}_0}}  & & G(\ca{A}) \\ & \ca{A} \ar[ru]_{Y} }
\]
 commutes up to isomorphism, where $Y$ denotes the Yoneda embedding.
\end{lemma}

\begin{proof}
 Let $\Sigma_0$ be the set of all cokernel diagrams in $\ca{G}_0$ which are right exact sequences in $\ca{G}$. From Theorem~\ref{thm:adams_vb_universal} we know that restriction along the inclusion $\ca{G}_0 \rightarrow \ca{G}$ induces an equivalence
\[
 \Fun_{c,\otimes}(\ca{G},\ca{D}) \rightarrow \Fun_{\Sigma_0,\otimes}(\ca{G}_0,\ca{D})
\]
 of categories for all tensor categories $\ca{D}$.

 Recall from Remark~\ref{rmk:universal_tensor_functor} that the restriction of $E$ along the Yoneda embedding $Y \colon \ca{A} \rightarrow G(\ca{A})$ is isomorphic to the inclusion $\ca{A} \rightarrow \ca{C}$  Thus, to prove the claim, we only need to check that the composite
\[
 \xymatrix{\ca{G}_0 \ar[r]^-{F\vert_{\ca{G}_0}} & \ca{A} \ar[r]^-{Y} & G(\ca{A}) }
\]
 sends cokernel diagrams in $\Sigma_0$ to cokernel diagrams in $G(\ca{A})$. By Definition~\ref{dfn:associated_geometric_cat}, $G(\ca{A})$ is the category $\Lex_{\Sigma}[\ca{A}^{\op},\Mod_R]$, where $\Sigma$ denotes the set of all cokernel diagrams in $\ca{A}$ which are locally split right exact sequences in $\ca{C}$. From the Yoneda lemma it follows that Yoneda embedding $\ca{A} \rightarrow G(\ca{A})$ sends cokernel diagrams in $\Sigma$ to cokernel diagrams. This reduces the problem to checking that $F$ sends cokernel diagrams in $\Sigma_0$ to locally split right exact sequences in $\ca{C}$.

 Let $B \in \ca{G}$ be an Adams algebra such that $B$ is a projective generator of $\ca{G}_B$ (such an algebra exists by Proposition~\ref{prop:adams_affine_algebra}). As we have seen in Example~\ref{example:locally_split_in_geom_cat}, the base change functor $(-)_B \colon \ca{G} \rightarrow \ca{G}_B$ sends all colimit diagrams in $\Sigma_0$ to split right exact sequences. The diagram
\[
 \xymatrix{\ca{G} \ar[r]^{F} \ar[d]_{(-)_B} & \ca{C} \ar[d]^{(-)_{FB}} \\ \ca{G}_B \ar[r]_{\overline{F}} & \ca{C}_{FB} }
\]
 (where $\overline{F}$ denotes the tensor functor induced by $F$ on the category of $B$-modules) commutes up to isomorphism, hence $(-)_{FB}$ sends the image of a cokernel diagram in $\Sigma_0$ to a split right exact sequence. Moreover, $FB$ is faithfully flat since $B$ is an Adams algebra and filtered colimits in the lfp category $\ca{C}$ are exact (see Proposition~\ref{prop:adams_strongly_stable}). Thus $F$ sends cokernel diagrams in $\Sigma_0$ to locally split right exact sequences in $\ca{C}$, as claimed.
\end{proof} 

\begin{lemma}\label{lemma:generator_of_rep_gl_d}
 Let $\ca{G}_0 \subseteq \Rep(\mathrm{GL}_d)$ be the smallest subcategory which contains the standard representation $V=\ubar{R}^d$, is closed under finite tensor products, finite direct sums, duals, and quotients which are locally free of finite constant rank. Then $\ca{G}_0$ is a generator of $\Rep(\mathrm{GL}_d)$.
\end{lemma}

\begin{proof}
 We use the equivalence between $\Rep(\mathrm{GL}_d)$ and the category of comodules of the Hopf $H=R[x_{ij}]_{\mathrm{det}}$ of regular functions on $\mathrm{GL}_d$. We write $\ubar{H}$ for the $H$-comodule $(H,\delta)$. By the proof of \cite[Proposition~1.4.4]{HOVEY}, it suffices to check that $\ubar{H}$ is a filtered colimit of objects $G_i \in \ca{G}_0$, for then the $G_i^{\vee}$ generate $\Rep(\mathrm{GL}_d)$. It suffices to prove this for $R=\mathbb{Z}$ since all the constructions under which $\ca{G}_0$ is assumed to be closed are stable under base change.

 Let $\ca{I}$ be the poset of subcomodules of $\ubar{H}$ consisting of images of homomorphisms of comodules
\[
 \textstyle\bigoplus\nolimits_{i=1}^k V^{\otimes n_i} \rightarrow \ubar{H}
\]
 where $n_i \in \mathbb{Z}$ and $V^n \defl (V^{\vee})^{-n}$ for $n<0$. Since $H$ is flat, these images are torsion free, hence their underlying $\mathbb{Z}$-modules are finitely generated free. Thus all the elements of $\ca{I}$ are contained in $\ca{G}_0$ (since $\ca{G}_0$ is closed under quotients which are locally free of constant finite rank). The poset $\ca{I}$ is clearly directed, so we are done if we can show that its union is all of $\ubar{H}$. This follows from the fact that there is an epimorphism
\[
 \Sym\bigl((V^{\vee})^{d} \oplus V^{d}\bigr) \rightarrow \ubar{H}
\]
 (see Lemmas~\ref{lemma:universal_framing} and \ref{lemma:universal_locally_free_object}). 
 \end{proof}

\begin{lemma}\label{lemma:generated_by_locally_free}
 Let $\ca{A}$, $\ca{C}$ be as in Theorem~\ref{thm:universal_geometric_category}, and let $A \in \ca{A}$ be an object which is locally free of rank $d$ as an object of $\ca{C}$. Then the representable functor $\ca{A}(-,A)$ is locally free of rank $d \in \mathbb{N}$ as an object of the tensor category $G(\ca{A})$. In particular, $G(\ca{A})$ is generated by locally free objects of constant finite rank.
\end{lemma}

\begin{proof}
 By Theorem~\ref{thm:universal_property_rep_gld}, there exists a tensor functor $F \colon \Rep(\mathrm{GL}_d) \rightarrow \ca{C}$ such that $F\ubar{R}^d \cong A$, where $\ubar{R}^d$ denotes the standard representation. We claim that this tensor functor sends all the objects of the generating set $\ca{G}_0$ defined in Lemma~\ref{lemma:generator_of_rep_gl_d} to objects in $\ca{A}$ (at least up to isomorphism). To see this, first recall that all epimorphisms between objects with duals in a geometric tensor category are locally split by an Adams algebra (see Example~\ref{example:locally_split_in_geom_cat}). Since $\ca{A}$ is closed under finite direct sums, tensor products, duals, and \emph{locally split} quotients which are locally free of constant finite rank, it follows that the image of $\ca{G}_0$ under $F$ is indeed contained in $\ca{A}$.

 Thus we can apply Lemma~\ref{lemma:factorization_through_F}, which shows that $F$ factors as
\[
 \xymatrix{\Rep(\mathrm{GL}_d) \ar[r]^-{F^{\prime}} & G(\ca{A}) \ar[r]^-{E} & \ca{C} }
\]
 for some tensor functor $F^{\prime}$ with $F^{\prime}\ubar{R}^d \cong \ca{A}(-,A)$.

 Applying Theorem~\ref{thm:universal_property_rep_gld} again, this shows that $\ca{A}(-,A)$ is indeed locally free of rank $d$ (the theorem is applicable since filtered colimits in $G(\ca{A})$ are exact, they are computed as in $\Prs{A}$).
\end{proof}

 As a next step towards our proof of Theorem~\ref{thm:universal_geometric_category}, we will show that $G(\ca{A})$ is an abelian category.

\begin{lemma}\label{lemma:locally_free_kernel}
 Let $\ca{C}$ be an lfp tensor category and let $M$, $N \in \ca{C}$ be locally free objects of rank $m$, $n \in \mathbb{N}$ respectively. Let $p \colon M \rightarrow N$ be a locally split epimorphism with kernel $K$. Then $K$ is locally free of rank $m-n$ and $K \rightarrow M$ is a locally split monomorphism.
\end{lemma}

\begin{proof}
 By definition of locally free objects and locally split epimorphism we can find a faithfully flat algebra $A \in \ca{C}$ such that $M_A \cong A^n$, $N_A \cong A^n$, and $p_A$ is a split epimorphism in the category $\ca{C}_A$ of $A$-modules. 

 Let $B=\ca{C}_A(A,A)$ be the $R$-algebra of endomorphisms of the $A$-module $A$. There exists a unique tensor functor $F \colon \Mod_B \rightarrow \ca{C}_A$ which sends $B$ to $A$ and induces the identity
\[
 B \cong \End_{\Mod_B}(B) \rightarrow \ca{C}_{A}(A,A)
\]
(this follows for example from Theorem~\ref{thm:day_convolution_universal}, but in this case it can also be seen much more directly). Since $p_A$ and its splitting are given by matrices with coefficients in $B=\ca{C}_A(A,A)$, it follows that $p_A$ is the image of a split epimorphism $q \colon B^m \rightarrow B^n$ in $\Mod_B$. Let $K^{\prime}$ denote its kernel. The fact that the sequence
\[
 \xymatrix{0 \ar[r] & K^{\prime} \ar[r] & B^m \ar[r]^{q} \ar[r] & B^n \ar[r] & 0}
\]
 is split implies that it is preserved by $F \colon \Mod_B \rightarrow \ca{C_A}$. Thus we have $FK^{\prime} \cong K_A$, which shows that $K \rightarrow M$ is a locally split monomorphism.

 It remains to show that $FK^{\prime}$ is locally free of rank $m-n$. But the $B$-module $K^{\prime}$ is locally free of rank $m-n$ in $\Mod_B$, and all tensor functors between lfp tensor categories preserve locally free objects of a constant finite rank (this follows from Theorem~\ref{thm:universal_property_rep_gld}).
\end{proof}

\begin{lemma}\label{lemma:pullback_of_locally_split}
 Let $\ca{A} \subseteq \ca{C}$ be as in Theorem~\ref{thm:universal_geometric_category} and let $p \colon M \rightarrow N$ be a locally split epimorphism in $\ca{C}$. Let
\[
 \xymatrix{M^{\prime} \ar[r]^{p^{\prime}} \ar[d]_{f^{\prime}} & N^{\prime} \ar[d]^{f} \\ M \ar[r]_{p} & N}
\]
 be a pullback diagram in $\ca{C}$ where $f$ is an arbitrary morphism. If the objects $M$, $N$, and $N^{\prime}$ lie in $\ca{A}$, then $M^{\prime}$ lies in $\ca{A}$ as well. Moreover, $p^{\prime}$ is a locally split epimorphism in $\ca{C}$.
\end{lemma}

\begin{proof}
 First note that pullbacks of locally split epimorphisms are locally split. It suffices to check that pullbacks of split epimorphisms are split, and this is true in \emph{any} category.

 We can reduce the problem to the case where $N^{\prime}=0$ (and thus where $M^{\prime}$ is the kernel $K$ of $p$). Indeed, the morphism $q \colon M \oplus N^{\prime} \rightarrow N$ which restricts to $p$ on $M$ and to $-f$ on $N^{\prime}$ is a locally split epimorphism (since $p$ is), and its kernel is the pullback $M^{\prime}$.

 The kernel $K$ of $p$ is locally free of constant finite rank and the morphism $K \rightarrow M$ is a locally split monomorphism by Lemma~\ref{lemma:locally_free_kernel}. Thus its dual $M^{\vee} \rightarrow K^{\vee}$ is a locally split \emph{epimorphism} whose target is locally free of constant finite rank. The object $M^{\vee}$ lies in $\ca{A}$ since $\ca{A}$ is closed under duals. From Condition~(ii) of Theorem~\ref{thm:universal_geometric_category} it follows that $K^{\vee}$ lies in $\ca{A}$. Using closure under duals again we find that $K \cong K^{\vee \vee}$ lies in $\ca{A}$, as claimed.
\end{proof}

 Let $\Sigma$ be a class of cokernel diagrams in an $R$-linear category $\ca{A}$. Recall that we write $R(\Sigma)$ for the class of all morphisms $q \colon N \rightarrow L$ in $\ca{A}$ for which there exists a morphism $p \colon M \rightarrow N$ in $\ca{A}$ such that $q$ is the cokernel of $p$ and the cokernel diagram
\[
 \xymatrix{M \ar[r]^{p} & N \ar[r]^{q} & L}
\]
lies in $\Sigma$.

\begin{lemma}\label{lemma:locally_split_rex_ind_class}
 Let $\ca{A} \subseteq \ca{C}$ be as in Theorem~\ref{thm:universal_geometric_category}. Let $\Sigma$ be the class of cokernel diagrams in $\ca{A}$ which are locally split right exact sequences in $\ca{C}$. Then $\Sigma$ is an ind-class (see \cite[Definition~2.3]{SCHAEPPI_INDABELIAN} and Definition~\ref{dfn:ind_class}) and the tensor category $G(\ca{A})$ of Definition~\ref{dfn:associated_geometric_cat} is abelian. The class $R(\Sigma)$ consists of all locally split epimorphisms in $\ca{C}$ whose domain and codomain both lie in $\ca{A}$.
\end{lemma}

\begin{proof}
 That the category $G(\ca{A})=\Lex_{\Sigma}[\ca{A}^{\op},\Mod_R]$ is abelian follows from the first claim and Proposition~\ref{prop:ind_class_implies_abelian}.

 We first prove the claim about $R(\Sigma)$. By definition of $\Sigma$ it is clear that all morphisms in $R(\Sigma)$ are locally split epimorphisms in $\ca{C}$. Conversely, if $p \colon M \rightarrow N$ is a locally split epimorphism in $\ca{C}$ with $M$, $N \in \ca{A}$, then its kernel $K$ lies in $\ca{A}$ by Lemma~\ref{lemma:pullback_of_locally_split}. Since the right exact sequence
\[
 \xymatrix{K \ar[r] & M \ar[r]^-{p} & N \ar[r] & 0}
\]
 is locally split, it lies in $\Sigma$. This shows that $p$ lies in $R(\Sigma)$, as claimed.

 To see that $\Sigma$ is an ind-class, note that
\[
 \xymatrix{N \ar[r]^{q} & L \ar[r] & 0 \ar[r] & 0}
\]
 is a locally split right exact sequence in $\ca{C}$ if $q$ is a locally split epimorphism in $\ca{C}$. This shows that $\Sigma$ satisfies Condition~(i) of Definition~\ref{dfn:ind_class}. To see that it also satisfies Condition~(ii), let
\[
 \xymatrix{M \ar[r]^p & N \ar[r]^{q} & L \ar[r] & 0}
\]
 be in $\Sigma$, and let $f \colon D \rightarrow N$ be a morphism in $\ca{A}$ with $qf=0$. Let $k \colon K \rightarrow N$ be the kernel of $q$ and let $p^{\prime} \colon M \rightarrow K$ be the factorization of $p$ through $k$. Then $p^{\prime}$ is a locally split epimorphism in $\ca{C}$ by definition of locally split right exact sequences (see Definition~\ref{dfn:split_rex_sequence}). The object $K$ lies in $\ca{A}$ by Lemma~\ref{lemma:pullback_of_locally_split}. Since $qf=0$ there exists a morphism $f^{\prime} \colon D \rightarrow K$ such that $kf^{\prime}=f$. Let
\[
 \xymatrix{E \ar[r]^{p^{\prime \prime}} \ar[d]_{f^{\prime\prime}} & D \ar[d]^{f^{\prime}} \\ M \ar[r]_{p^{\prime}} & K }
\]
 be a pullback diagram in $\ca{C}$. Since $\ca{A}$ is closed under pullbacks of locally split epimorphisms along arbitrary morphisms (see Lemma~\ref{lemma:pullback_of_locally_split}), the object $E$ lies in $\ca{A}$ and $p^{\prime\prime}$ is locally split. As we saw above, this implies that $p^{\prime \prime}$ lies in  $R(\Sigma)$. Thus $\Sigma$ does indeed satisfy Condition~(ii) of the definition of an ind-class.
\end{proof}

 At this point we have shown that $G(\ca{A})$ is generated by a set of locally free objects of constant finite rank (see Lemma~\ref{lemma:generated_by_locally_free}) and that $G(\ca{A})$ is abelian (see Lemma~\ref{lemma:locally_split_rex_ind_class} above). If $R$ is a $\mathbb{Q}$-algebra, then we already know that $G(\ca{A})$ is a geometric tensor category from Theorem~\ref{thm:qalg_geometric_characterization}. To see this for general commutative base rings $R$, we still need to check the condition about epimorphisms in Theorem~\ref{thm:description}. To do this, we will use the fact that $G(\ca{A})$ can also be described as a category of (enriched) sheaves for a Grothendieck topology on $\ca{A}$ (see \cite[Proposition~2.6]{SCHAEPPI_INDABELIAN}). In the unenriched case, there is a convenient characterization of epimorphisms in the category of sheaves. The same remains true in the enriched case, though the proofs have to be modified a bit. Therefore we defer the proof of the following proposition to Appendix~\ref{appendix:sheaves}.

\begin{prop}\label{prop:epi_characterization}
 Let $\ca{A} \subseteq \ca{C}$ be as in Theorem~\ref{thm:universal_geometric_category}. Let
\[
\ca{A}(-,f) \colon \ca{A}(-,A) \rightarrow \ca{A}(-,B) 
\]
 be a morphism between representable presheaves in $G(\ca{A})$. Then $\ca{A}(-,f)$ is an epimorphism in $G(\ca{A})$ if and only if $f \colon A \rightarrow B$ is a locally split epimorphism in $\ca{C}$.
\end{prop}

 With this in hand, we can finish the proof of Theorem~\ref{thm:universal_geometric_category}.

\begin{proof}[Proof of Theorem~\ref{thm:universal_geometric_category}]
 To show that $G(\ca{A})$ is geometric, we will check that the conditions of Theorem~\ref{thm:description} are satisfied. The lfp tensor category $G(\ca{A})$ is abelian by Lemma~\ref{lemma:locally_split_rex_ind_class} and it is generated by locally free objects of constant finite rank by Lemma~\ref{lemma:generated_by_locally_free}. The condition about epimorphisms follows if we can show the following claim: given an epimorphism $\ca{A}(-,f)\colon \ca{A}(-,A) \rightarrow \ca{A}(-,B)$ in $G(\ca{A})$, the induced morphism on fixed points
\[
\bigl( \ca{A}(-,f)^{\otimes i} \bigr)^{\Sigma_i} \colon \bigl( \ca{A}(-,A)^{\otimes i} \bigr)^{\Sigma_i} \rightarrow \bigl( \ca{A}(-,B)^{\otimes i} \bigr)^{\Sigma_i}
\]
 of the symmetric group action on the $i$-fold tensor product of $\ca{A}(-,f)$ is again an epimorphism.

 From Propostion~\ref{prop:epi_characterization} we know that $f \colon A \rightarrow B$ is a locally split epimorphism in $\ca{C}$. From this it follows that
\[
 ( f^{\otimes i})^{\Sigma_i} \colon (A^{\otimes i})^{\Sigma_i} \rightarrow (B^{\otimes i})^{\Sigma_i}
\]
 is a locally split epimorphism in $\ca{C}$ as well (it is sent to a split epimorphism by any base change functor which sends $f$ to a split epimorphism). Moreover, the objects $(A^{\otimes i})^{\Sigma_i}$ and $(B^{\otimes i})^{\Sigma_i}$ lie in $\ca{A}$. To see this it suffices to prove that their duals $\Sym^i(A^{\vee})$ and $\Sym^i(B^{\vee})$ lie in $\ca{A}$. But $\Sym^i(A^{\vee})$ is locally free of constant finite rank and $(A^{\vee})^{\otimes i} \rightarrow \Sym^i(A^{\vee})$ is a locally split epimorphism (this follows for example by Theorem~\ref{thm:universal_property_rep_gld} since the statement is clearly true for the standard representation of $\mathrm{GL}_d$). Thus $\Sym^i(A^{\vee})$ and similarly $\Sym^i(B^{\vee})$ lie in $\ca{A}$ by Condition~(ii) of Theorem~\ref{thm:universal_geometric_category}.

 By Proposition~\ref{prop:epi_characterization}, the locally split epimorphism $(f^{\otimes i})^{\Sigma_i}$ is sent to an epimorphism by the Yoneda embedding $\ca{A} \rightarrow G(\ca{A})$. The claim now follows from the two facts that the Yoneda embedding is symmetric strong monoidal and that it preserves any limits that exist in $\ca{A}$, in particular the fixed points of the symmetric group actions on $A^{\otimes i}$ and $B^{\otimes i}$ respectively.
\end{proof}

\subsection{The universal property}\label{section:basic_properties_universal_geometric}
 Having established that $G(\ca{A})$ is a geometric category, we can deduce a few more facts about its structure. For example, we can give a precise description of its subcategory of vector bundles. Recall that a \emph{vector bundle} in a geometric category is a locally free object of finite rank (not necessarily constant, though). This description will be crucial for establishing the desired universal property of $G(\ca{A})$.

\begin{prop}\label{prop:vector_bundles_in_GA}
 Let $\ca{A} \subseteq \ca{C}$ be as in Theorem~\ref{thm:universal_geometric_category}. Then the Yoneda embedding $\ca{A} \rightarrow G(\ca{A})$ gives an equivalence between $\ca{A}$ and the category of vector bundles in $G(\ca{A})$ of constant finite rank. The category of \emph{all} vector bundles in $G(\ca{A})$ is equivalent to the Karoubi envelope $\overline{\ca{A}}$ of $\ca{A}$.
\end{prop}

\begin{proof}
 The second statement follows from the first and the fact that any vector bundle in a geometric category is a direct summand of a vector bundle of constant rank (see Proposition~\ref{prop:vb_c_generator}).

 We know from Lemma~\ref{lemma:generated_by_locally_free} that all objects in the image of the Yoneda embedding are vector bundles of constant rank. Since the Yoneda embedding is fully faithful it only remains to show the converse: any vector bundle $M \in G(\ca{A})$ of constant rank lies in the image of the Yoneda embedding. 
 
 The object $M$ is finitely presentable since it has a dual (here we are using the fact that the unit of $G(\ca{A})$ is finitely presentable: it lies in the image of the Yoneda embedding). By Proposition~\ref{prop:ind_class_implies_abelian}, there exists a morphism $f \colon A \rightarrow A^{\prime}$ in $\ca{A}$ and a right exact sequence
\begin{equation}\label{eqn:presentation_of_vector_bundle}
  \xymatrix{\ca{A}(-,A) \ar[r]^-{\ca{A}(-,f)} & \ca{A}(-,A^{\prime}) \ar[r] & M \ar[r] & 0 } 
\end{equation}
 in $G(\ca{A})$. This sequence is locally split by an Adams algebra (see Example~\ref{example:locally_split_in_geom_cat}). It follows that it is sent to a locally split right exact sequence by \emph{any} tensor functor whose codomain has exact filtered colimits. Applying this to the tensor functor $E \colon G(\ca{A}) \rightarrow \ca{C}$ from Remark~\ref{rmk:universal_tensor_functor} we find that there exists a locally split right exact sequence
\begin{equation}\label{eqn:image_of_presentation}
 \xymatrix{A \ar[r]^-{f} & A^{\prime} \ar[r] & EM \ar[r] & 0}
\end{equation}
 in $\ca{C}$. The object $EM \in \ca{C}$ is locally free of constant rank by Theorem~\ref{thm:universal_property_rep_gld} (it is the image of the standard representation for some tensor functor from $\Rep(\mathrm{GL}_d)$ to $\ca{C}$). We have just argued that the epimorphism $A^{\prime} \rightarrow EM$ is locally split, so from Condition~(ii) of Theorem~\ref{thm:universal_geometric_category} it follows that $EM \in \ca{A}$.

 Moreover, since the right exact sequence of Diagram~\eqref{eqn:image_of_presentation} is locally split right exact, it lies in the class $\Sigma$ of Definition~\ref{dfn:associated_geometric_cat}. Thus it is preserved by the Yoneda embedding $\ca{A} \rightarrow \Lex_{\Sigma}[\ca{A}^{\op},\Mod_R]=G(\ca{A})$. Combining this with the right exact sequence in Diagram~\eqref{eqn:presentation_of_vector_bundle} we find that $M \cong \ca{A}(-,EM)$, as claimed.
\end{proof}

 As a consequence of this, we find that $G(\ca{A})$ has the following universal property among geometric tensor categories.

\begin{prop}\label{prop:universal_property_of_associated_geometric_category}
 Let $\ca{A} \subseteq \ca{C}$ be as in Theorem~\ref{thm:universal_geometric_category} and let $\ca{G}$ be a geometric tensor category over $R$ with generating set $\ca{G}_0$, closed under finite direct sums and tensor products, and consisting of vector bundles of constant rank. Then composition with the tensor functor $E \colon G(\ca{A}) \rightarrow \ca{C}$ of Remark~\ref{rmk:universal_tensor_functor} induces a fully faithful functor
\[
 \Fun_{c,\otimes }\bigr(\ca{G},G(\ca{A})\bigl) \rightarrow \Fun_{c,\otimes}(\ca{G},\ca{C})
\]
 whose essential image consists of all tensor functors $F \colon \ca{G} \rightarrow \ca{C}$ with $F(\ca{G}_0) \subseteq \ca{A}$.
\end{prop}

\begin{proof}
 We will first identify the essential image. That any such tensor functor $F$ lies in the essential image is a consequence of Lemma~\ref{lemma:factorization_through_F}. Conversely, if $F=EF^{\prime}$ for some tensor functor $F^{\prime} \colon \ca{G} \rightarrow G(\ca{A})$, we know that $F^{\prime}$ preserves locally free objects of constant rank (since there are Adams algebras exhibiting each such object as locally free). Thus $F^{\prime} \vert_{\ca{G}_0}$ factors through the Yoneda embedding $Y \colon \ca{A} \rightarrow G(\ca{A})$ by Proposition~\ref{prop:vector_bundles_in_GA}. It follows that $F \vert_{\ca{G}_0}$ factors through $EY$, which by Definition of $E$ is isomorphic to the inclusion $\ca{A} \rightarrow \ca{C}$ (see Remark~\ref{rmk:universal_tensor_functor}).

 It remains to check that composition with $E$ gives a fully faithful functor. Let $\Sigma_0$ be the set of cokernel diagrams in $\ca{G}_0$ which are right exact sequences in $\ca{G}$. Then restriction along the inclusion $\ca{G}_0 \rightarrow \ca{G}$ induces an equivalence
\[
 \Fun_{c,\otimes}(\ca{G},\ca{D}) \rightarrow \Fun_{\Sigma_0,\otimes}(\ca{G}_0,\ca{D})
\]
 for any tensor category $\ca{D}$ (see Corollary~\ref{cor:C_universal_tensor}). As we have just seen, any tensor functor in $\Fun_{\Sigma_0,\otimes}\bigr(\ca{G}_0,G(\ca{A})\bigl)$ factors through the Yoneda embedding. The claim therefore follows from the fact that $EY$ is fully faithful (it is isomorphic to the inclusion $\ca{A} \rightarrow \ca{C}$).
\end{proof}

 Specializing this to the case where $\ca{A}$ is the category of \emph{all} locally free objects of finite constant rank, we obtain the following theorem. It will be used in the next section to construct limits of weakly Tannakian categories and colimits of Adams stacks.

\begin{thm}\label{thm:geometric_categories_coreflective}
 Let $\ca{C}$ be an lfp tensor category over $R$ (not necessarily abelian) and let $\LF^c_{\ca{C}} \subseteq \ca{C}$ be the full subcategory of locally free objects of constant rank. Then for any geometric tensor category $\ca{G}$ over $R$, composition with the tensor functor $E \colon G(\LF^c_{\ca{C}}) \rightarrow \ca{C}$ of Remark~\ref{rmk:universal_tensor_functor} induces an equivalence
\[
 \Fun_{c,\otimes}\bigr(\ca{G},G(\LF^c_{\ca{C}})\bigl) \rightarrow \Fun_{c,\otimes}(\ca{G},\ca{C})
\]
 of categories. Thus the inclusion of geometric tensor categories in the 2-category of lfp tensor categories and tensor functors between them has a right biadjoint given by $G(\LF^c_{(-)})$.
\end{thm}

\begin{proof}
 First note that $\LF^c_{\ca{C}}$ satisfies the conditions of Theorem~\ref{thm:universal_geometric_category}. The various closure properties of Condition~(i) follow from the fact that for any finite set of objects in $\LF^c_{\ca{C}}$ we can find a single faithfully flat algebra which sends all of them to finite direct sums of the unit. Condition~(ii) is vacuously true for $\LF^c_{\ca{C}}$.
 
 The fact that $E$ induces an equivalence precisely means that the geometric tensor category $G(\LF^c_{\ca{C}})$ is a bicategorical coreflection of the lfp tensor category $\ca{C}$, that is, that the inclusion described in the statement of the theorem has the desired right biadjoint.

 We will apply Proposition~\ref{prop:universal_property_of_associated_geometric_category} to show that composition with $E$ gives an equivalence. Let $\ca{G}_0 \subseteq \ca{G}$ be the full subcategory of vector bundles of constant rank (that is, $\ca{G}_0=\LF^c_{\ca{G}}$). Since $\ca{G}$ is a geometric tensor category, the subcategory $\ca{G}_0$ forms a generator of $\ca{G}$, and it clearly satisfies the conditions of Proposition~\ref{prop:universal_property_of_associated_geometric_category}. To conclude the proof it only remains to show that for any tensor functor $F \colon \ca{G} \rightarrow \ca{C}$ we have $F(\ca{G}_0) \subseteq \LF^c_{\ca{C}}$, that is, that any such tensor functor preserves locally free objects of constant finite rank. This follows for example from Theorem~\ref{thm:universal_property_rep_gld}, or from the fact that in a geometric category we can more directly find an Adams algebra (not just a faithfully flat algebra) which sends a given locally free object of rank $d$ to the $d$-fold direct sum of the unit object. 
\end{proof}

 Another special case of the construction $G(\ca{A})$ is the case where the category $\ca{C}$ is already geometric, but $\ca{A}$ is a proper subcategory of the category of vector bundles of constant rank (for example, generated in a suitable sense by a single vector bundle).

\begin{dfn}\label{dfn:stack_generated_by_vector_bundles}
 Let $X$ be an Adams stack, and let $S$ be a set of vector bundles of constant rank on $X$. Let $\langle S \rangle$ be the smallest subcategory of $\QCoh(X)$ which is closed under finite direct sums, finite tensor products, duals, and quotients which are vector bundles of constant rank. Let $X^{S}$ be the (essentially) unique Adams stack with $\QCoh(X^{S}) \simeq G(\langle S \rangle)$.
\end{dfn}

\begin{rmk}
 The category $\langle S \rangle$ is precisely the smallest subcategory which satisfies the conditions of Theorem~\ref{thm:universal_geometric_category}. This follows since any epimorphism between vector bundles is locally split (see Example~\ref{example:duals_locally_split}). Thus $G(\langle S \rangle)$ is indeed geometric, so the above definition makes sense.
\end{rmk}

 Recall that the pseudofunctor which sends an Adams stack $X$ to the geometric tensor category $\QCoh(X)$ gives a contravariant equivalence between the 2-category of Adams stacks and the 2-category of geometric tensor categories and tensor functors (see Theorem~\ref{thm:equivalence_adams_geometric}). In particular, in the situation of Definition~\ref{dfn:stack_generated_by_vector_bundles}, there exists  a morphism $e \colon X \rightarrow X^{S}$ such that $e^{\ast} \cong E \colon G(\langle S \rangle) \rightarrow \QCoh(X)$. The universal property of Proposition~\ref{prop:universal_property_of_associated_geometric_category} implies that $X^S$ has a universal property in the 2-category of Adams stacks.

\begin{prop}\label{prop:stack_generated_by_vector_bundles}
 Let $X$ be an Adams stack and let $S$ be a set of vector bundles on $X$, each of constant finite rank, and let $\langle S \rangle \subseteq \QCoh(X)$ and $X^S$ be as in Definition~\ref{dfn:stack_generated_by_vector_bundles}. Then the functor $e^{\ast} \colon \QCoh(X^S) \rightarrow \QCoh(X)$ restricts to an equivalence of the category $\VB^c(X^S)$ of vector bundles of constant rank on $X^S$ and the full subcategory $\langle S \rangle$ of $\QCoh(X)$.

 Moreover, for any Adams stack $Y$, composition with $e \colon X \rightarrow X^S$ induces an equivalence between the category of morphisms $X^S \rightarrow Y$ and the full subcategory of morphisms $f \colon X \rightarrow Y$ for which $f^{\ast}\bigl(\VB^c(Y)\bigr) \subseteq \langle S \rangle$. 
\end{prop}

\begin{proof}
 Recall that $\QCoh(X^S)\simeq G(\langle S \rangle)$. The claim about vector bundles of constant rank on $X^S$ therefore follows from Proposition~\ref{prop:vector_bundles_in_GA}.

 The second part follows from the universal property of $G(\langle S \rangle)$ (see Proposition~\ref{prop:universal_property_of_associated_geometric_category}) and the contravariant equivalence between the 2-category of Adams stacks and the 2-category of geometric tensor categories (see Theorem~\ref{thm:equivalence_adams_geometric}).
\end{proof}

 There is a similar construction in the theory of Tannakian categories. Given a Tannakian category $\ca{T}$ and an object $V \in \ca{T}$, the category $\langle V \rangle$ is the full subcategory of subquotients of finite tensor products of copies of $V$ and its dual. Despite the slightly different description, this category coincides with the category of the same name in Definition~\ref{dfn:stack_generated_by_vector_bundles} (where $\QCoh(X)=\Ind(\ca{T})$ and $S$ is the set $\{ V \}$). This follows since all epimorphisms and monomorphisms between objects in $\ca{T}$ are locally split in $\Ind(\ca{T})$, and all objects of $\ca{T}$ are locally free of constant finite rank as objects of $\Ind(\ca{T})$. Moreover, in this case there is a simpler description of the geometric category $G(\langle S \rangle)$.

\begin{prop}
 Let $\ca{T}$ be a Tannakian category over a field $k$, and let $S$ be a set of objects of $\ca{T}$. Then the category $\langle S \rangle \subseteq \Ind(\ca{T})$ of Definition~\ref{dfn:stack_generated_by_vector_bundles} is Tannakian, the geometric category $G(\langle S \rangle)$ is equivalent to $\Ind(\langle S \rangle)$, and the tensor functor $E \colon \Ind(\langle S \rangle) \rightarrow \Ind(\ca{T})$ of Remark~\ref{rmk:universal_tensor_functor} is exact and fully faithful.
\end{prop}

\begin{proof}
 As in the case $S=\{ V \}$ outlined above we find that the category $\langle S \rangle$ consists precisely of the subquotients of finite tensor products of objects in $S$ and their duals. Thus $\langle S \rangle$ is an abelian subcategory of $\Ind(\ca{T})$, and the inclusion $\langle S \rangle \rightarrow \Ind(\ca{T})$ is exact.

 Let $w \colon \ca{T} \rightarrow \Vect_K$ be a fiber functor for some field extension $K \supseteq k$. The induced tensor functor $\Ind(\ca{T}) \rightarrow \Vect_K$ sends all right exact sequences whose objects lie in $\ca{T}$ to split exact sequences. Up to equivalence, this tensor functor is of the form $(-)_B$ for some Adams algebra $B \in \Ind(\ca{T})$ (see \cite[Proposition~3.9]{SCHAEPPI_GEOMETRIC} and Proposition~\ref{prop:adams_affine_algebra}). Thus the class $\Sigma$ of cokernel diagrams in $\langle S \rangle$ which are locally split right exact sequences in $\Ind(\ca{T})$ coincides with the class of \emph{all} right exact sequences in the abelian category $\langle S \rangle$. It follows that
\[
 G(\langle S \rangle)=\Lex_{\Sigma}[\langle S \rangle^{\op},\Vect_k]
\]
 is the category of ind-objects of $\langle S \rangle$. The induced tensor functor
\[
 E \colon \Ind(\langle S \rangle ) \rightarrow \Ind(\ca{T})
\]
 of Remark~\ref{rmk:universal_tensor_functor} is exact and fully faithful since the inclusion $\langle S \rangle \rightarrow \Ind(\ca{T})$ has the same properties.
\end{proof}

\subsection{Colimits of Adams stacks}\label{section:colimits_of_adams_stacks}

 In order to prove our result on the cocompleteness of the 2-category of Adams stacks, it is convenient to shift our perspective slightly. From Theorem~\ref{thm:geometric_categories_coreflective} we know that the 2-category of geometric tensor categories is a (bicategorically) coreflective sub-2-category of the 2-category of lfp tensor categories and tensor functors between them. If the latter had all bicategorical limits, we could deduce that the 2-category of geometric tensor categories has all bicategorical limits as well. However, there are some subtleties to take into account: while the limit of locally finitely presentable categories is again locally presentable, it need not be locally \emph{finitely} presentable in general. Moreover, it is unclear if $\Fun_{c,\otimes}(\ca{C},\ca{D})$ is essentially small for general lfp tensor categories $\ca{C}$, $\ca{D}$.

 For these reasons it makes sense to consider the 2-category of weakly Tannakian categories (see Definition~\ref{dfn:weakly_tannakian}) instead of the 2-category of geometric tensor categories.

 As we saw in Proposition~\ref{prop:geometric_weakly_tannakian}, these 2-categories are equivalent: if $\ca{C}$ is geometric, then the full subcategory $\ca{C}_{\fp}$ of finitely presentable objects is weakly Tannakian, and if $\ca{A}$ is weakly Tannakian, then $\Ind(\ca{A})$ is geometric. The 2-category of weakly Tannakian categories has the advantage that it is naturally a sub-2-category of the 2-category $\ca{RM}$, which does not have the above mentioned drawbacks.

 The hom-categories of $\ca{RM}$ are clearly essentially small since the objects of $\ca{RM}$ are essentially small categories. Moreover, $\ca{RM}$ has all bicategorical limits and colimits. This follows for example from the fact that $\ca{RM}$ is the category of ``commutative monoids'' (more precisely, symmetric pseudomonoids) in the bicategorically complete and cocomplete 2-category $\Rex$ of essentially small finitely cocomplete $R$-linear categories and right exact functors between them (see \cite[\S 5]{SCHAEPPI_INDABELIAN}). The usual arguments for existence of limits and colimits in the category of commutative algebras in an lfp tensor category can be ``categorified.''

 To any right exact symmetric monoidal category we can associate a weakly Tannakian category using the construction introduced in \S \ref{section:universal_geometric_cats}. In order to simplify the statement, we extend the definition of locally free objects and locally split epimorphism to right exact symmetric monoidal categories as follows.

\begin{dfn}\label{dfn:locally_free_in_RM}
 An object $A$ of a right exact symmetric monoidal category $\ca{A}$ is called \emph{locally free of rank $d \in \mathbb{N}$} (respectively locally free of constant finite rank) if $A$, considered as an object of the lfp tensor category $\Ind(\ca{A})$, is locally free of rank $d$ in the sense of Definition~\ref{dfn:locally_free_objects} (respectively if it is locally free of rank $d$ for some $d \in \mathbb{N}$). We write $\LF^c_{\ca{A}}$ for the full subcategory of $\ca{A}$ consisting of locally free objects of constant finite rank.

 A morphism in $\ca{A}$ is called a \emph{locally split epimorphism} if it is a locally split epimorphism in $\Ind(\ca{A})$ in the sense of Definition~\ref{dfn:locally_split_epi}.
\end{dfn}

\begin{rmk}
 Let $\ca{A}$ be a right exact symmetric monoidal category. Then the category $\LF^c_{\ca{A}}$ of locally free objects of constant finite rank in $\ca{A}$ is equivalent to the category $\LF^{c}_{\Ind(\ca{A})}$ of locally free objects of constant finite rank in the lfp tensor category $\Ind(\ca{A})$. It is immediate from the definition that the former category is a full subcategory of the latter. To see that this inclusion is an equivalence, note that objects of $\LF^c_{\Ind(\ca{A})}$ have duals, so they are finitely presentable in $\Ind(\ca{A})$ since the unit of $\Ind(\ca{A})$ is finitely presentable. The claim follows from the fact that the finitely presentable objects of $\Ind(\ca{A})$ are precisely the objects isomorphic to objects in the image of the embedding of $\ca{A}$ in $\Ind(\ca{A})$.
\end{rmk}

\begin{dfn}\label{dfn:universal_weakly_tannakian}
 For $\ca{A} \in \ca{RM}$, let 
\[
 T(\ca{A}) \defl G\bigl(\LF^c_{\ca{A}},\Ind(\ca{A})\bigr)_{\fp}
\]
 be the right exact symmetric monoidal category of finitely presentable objects in the geometric tensor category $G\bigl(\LF^c_{\ca{A}},\Ind(\ca{A})\bigr)$ associated to the full subcategory $\LF^c_{\ca{A}}$ of $\Ind(\ca{A})$ consisting of locally free objects of constant finite rank (see Theorem~\ref{thm:universal_geometric_category}). To keep the notation simpler, we will write $G(\LF^c_{\ca{A}})$ for $G\bigl(\LF^c_{\ca{A}},\Ind(\ca{A})\bigr)$ in the remainder of this section.
\end{dfn}

\begin{thm}\label{thm:weakly_tannakian_coreflective}
 The assignment which sends $\ca{A} \in \ca{RM}$ to $T(\ca{A})$ gives a right biadjoint to the inclusion of the full sub-2-category of weakly Tannakian categories in $\ca{RM}$. In particular, the 2-category of weakly Tannakian categories has all bicategorical limits, and it is closed under all bicategorical colimits. The limit of a diagram of weakly Tannakian categories $\ca{A}_i$ is given by $T( \lim \ca{A}_i)$, where $\lim \ca{A}_i$ denotes the limit in $\ca{RM}$. 
\end{thm}

\begin{proof}
 We use the fact that for any weakly Tannakian category $\ca{T}$ and any lfp tensor category $\ca{C}$, there is an equivalence
\begin{equation}\label{eqn:extend_to_ind_objects}
 \ca{RM}(\ca{T},\ca{C}_{\fp}) \simeq \Fun_{c,\otimes}\bigl(\Ind(\ca{T}), \ca{C}\bigr) 
\end{equation}
 of categories which is natural in $\ca{C}$ (this follows from Theorem~\ref{thm:day_convolution_universal} and the fact that any tensor functor whose domain is geometric preserves finitely presentable objects; see also \cite[Lemma~2.6]{SCHAEPPI_GEOMETRIC}).

 Thus for $\ca{A} \in \ca{RM}$ and $\ca{C}=\Ind(\ca{A})$ we have an equivalence
\[
 \ca{RM}(\ca{T},\ca{A}) \simeq \Fun_{c,\otimes}\bigl(\Ind(\ca{T}),\Ind(\ca{A})\bigr)
\]
 of categories which is natural in $\ca{A}$ and $\ca{T}$. The right hand side is naturally equivalent to the category
\[
 \Fun_{c,\otimes}\bigl(\Ind(\ca{T}),G(\LF^c_{\ca{A}})\bigr)
\]
 by Theorem~\ref{thm:geometric_categories_coreflective}. Applying the equivalence \eqref{eqn:extend_to_ind_objects} again we find that there is an equivalence
\[
 \ca{RM}(\ca{T},\ca{A}) \simeq \ca{RM}\bigr(\ca{T},T(\ca{A})\bigl)
\]
 which is natural in $\ca{T}$. This proves the first claim.

 The second claim is a general fact about a full sub-2-category $\ca{S}$ of a 2-category $\ca{K}$ whose inclusion $I \colon \ca{S} \rightarrow \ca{K}$ has a right biadjoint $T \colon \ca{K} \rightarrow \ca{S}$. The bicategorical limit of a diagram in $\ca{S}$ is given by applying $T$ to the limit in $\ca{K}$. To see that $\ca{S}$ is closed under bicategorical colimits, note that from the universal property of colimits we can construct a 1-cell
\[
 \colim S_i \rightarrow IT \colim S_i
\]
 whose composite with the counit $\varepsilon \colon IT \colim S_i \rightarrow \colim S_i$ is isomorphic to the identity. The claim that $\colim S_i \in \ca{S}$ follows from the fact that $\ca{S}$ is closed under such ``essential'' retracts (at least up to equivalence, which is all we need).
\end{proof}

 This theorem implies that the 2-category of Adams stacks has all bicategorical limits and colimits.

\begin{proof}[Proof of Theorem~\ref{thm:complete_cocomplete}]
 The 2-category $\ca{AS}$ of Adams stacks is contravariantly equivalent to the 2-category of weakly Tannakian categories by Theorem~\ref{thm:equivalence_adams_weakly_tannakian}. Thus $\ca{AS}$ is bicategorically complete and cocomplete by Theorem~\ref{thm:weakly_tannakian_coreflective}.
\end{proof}

 As an immediate consequence, we find that for reasonable stacks $X$, there exists a universal approximation $X \rightarrow X^{\prime}$ of $X$ by an Adams stack $X^{\prime}$.

\begin{thm}\label{thm:universal_adams_stack}
 Let $X$ be a \emph{small} stack on the $\fpqc$-site $\Aff_R$, that is, a stack that can be written as a bicategorical colimit of affine schemes for some diagram with small indexing category. Then there exists an Adams stack $X^{\prime}$ and a morphism of stacks $X \rightarrow X^{\prime}$ which induces an equivalence
\[
 \Hom(X^{prime},Y) \rightarrow \Hom(X,Y)
\]
 for all Adams stacks $Y$.
\end{thm}

\begin{proof}
 Write $X=\colim \Spec(A_i)$ in the 2-category of stacks on the $\fpqc$-site $\Aff_R$, and let $X^{\prime}$ be the colimit of the same diagram in the 2-category $\ca{AS}$ of Adams stacks. Then for any Adams-stack $Y$ we have equivalences
\begin{align*}
 \ca{AS}(X^{\prime},Y) &\simeq \lim \ca{AS}\bigl(\Spec(A_i),Y\bigr)\\
&= \lim \Hom \bigl(\Spec(A_i),Y\bigr) \\
&\simeq \Hom(X,Y)
\end{align*}
 which are natural in $Y$. By the Yoneda lemma, this equivalence is completely determined by the image of $\id_{X^{\prime}} \in \ca{AS}(X^{\prime},X^{\prime})$, that is, it is given by composition with a morphism $ X \rightarrow X^{\prime}$, as claimed.
\end{proof}

\appendix
\section{Enriched Grothendieck topologies and sheaves}\label{appendix:sheaves}

\subsection{Basic definitions}
 In this appendix we will establish some basic properties of categories of sheaves enriched in $R$-modules. Their counterparts for \emph{unenriched} sheaves are well-known. However, the proofs often differ slightly. The main reason for this is that certain convenient properties that hold for unenriched presheaves of sets do not hold for $R$-linear presheaves. For example, the category of $\Set$-valued presheaves has universal colimits and disjoint coproducts, neither of which is true for the category of $R$-linear presheaves. We therefore give a detailed account of this theory and develop some of the generalizations necessary to prove the results of \S \ref{section:limits}.

\begin{dfn}\label{dfn:coverage}
 Let $\ca{A}$ be an $R$-linear category. A class $\Theta$ of morphisms in $\ca{A}$ is called a \emph{singleton coverage} on $\ca{A}$ if it satisfies the following closure property: given $p \colon A \rightarrow B$ in $\Theta$ and an arbitrary morphism $f \colon C \rightarrow B$, there exists an object $D \in \ca{A}$ together with a morphism $p^{\prime} \colon D \rightarrow B$ in $\Theta$ and a morphism $f^{\prime} \colon D \rightarrow A$ such that the diagram
\[
 \xymatrix{D \ar[r]^-{p^{\prime}} \ar[d]_{f^{\prime}} & C \ar[d]^{f} \\ A \ar[r]_-{p} & B }
\]
 is commutative. The morphisms in $\Theta$ are called \emph{coverings}.
\end{dfn}

\begin{example}\label{example:ind_class_gives_covering}
 Let $\Sigma$ be an ind-class on $\ca{A}$. Then $R(\Sigma)$ (that is, the morphisms that appear as cokernels in some cokernel diagram in $\Sigma$) is a singleton coverage on $\ca{A}$. This follows if we apply Condition~(ii) to the right exact sequence in Condition~(i) of the definition of an ind-class (see  Definition~\ref{dfn:ind_class} and \cite[Definition~2.3]{SCHAEPPI_INDABELIAN}).
\end{example}

 As in the unenriched context, we can use a singleton coverage on $\ca{A}$ to define a category of sheaves on $\ca{A}$. Given an $R$-linear presheaf $F \colon \ca{A}^{\op} \rightarrow \Mod_R$, a morphism $f \colon A \rightarrow B$ in $\ca{A}$, and an element $b \in FB$, we write $b \cdot f$ for the element $Ff(b)$ of $FA$ (if the presheaf $F$ is understood from the context). This notation is often used for unenriched presheaves since the functor axioms correspond to the right action axioms: we have $(b \cdot f) \cdot g=b \cdot (f \circ g)$ and $b \cdot \id_B=b$ for all composable morphisms $f,g$ in $\ca{A}$.

\begin{dfn}\label{dfn:sheaf_for_coverage}
 Let $\Theta$ be a singleton coverage on the $R$-linear category $\ca{A}$. An $R$-linear presheaf $F \colon \ca{A}^{\op} \rightarrow \Mod_R$ on $\ca{A}$ is called an \emph{$R$-linear sheaf} for $\Theta$ if the following conditions hold for all coverings $p \colon A \rightarrow B$ in $\Theta$:
\begin{enumerate}
 \item[(i)] The morphism $Fp \colon FB \rightarrow FA$ is a monomorphism;
 \item[(ii)] An element $a \in FA$ is of the form $b \cdot p$ for some $b \in FB$ if and only if $a \cdot f=0$ for all morphisms $f \colon A^{\prime} \rightarrow A$ in $\ca{A}$ with $pf=0$.
\end{enumerate}
 The full subcategory of $\Prs{A}$ consisting of sheaves for $\Theta$ is denoted by $\Sh_{\Theta}(\ca{A})$.
\end{dfn}
 
 The main technical result of \cite[\S 2]{SCHAEPPI_INDABELIAN} is the following proposition.

\begin{prop}\label{prop:lex_sigma_equals_sheaves}
 Let $\Sigma$ be an ind-class on the $R$-linear category $\ca{A}$. Then an $R$-linear presheaf on $\ca{A}$ is a sheaf for the coverage $R(\Sigma)$ of Example~\ref{example:ind_class_gives_covering} if and only if it sends cokernel diagrams in $\Sigma$ to kernel diagrams. In other words, we have $\Lex_{\Sigma}[\ca{A}^{\op},\Mod_R]=\Sh_{R(\Sigma)}(\ca{A})$.
\end{prop}

\begin{proof}
 This is proved in \cite[Proposition~2.6]{SCHAEPPI_INDABELIAN}.
\end{proof}

 Categories of $R$-linear sheaves can also be described via enriched Grothendieck topologies in the sense of Borceux and Quinteiro \cite{BORCEUX_QUINTEIRO}. We use the fact that epimorphisms in such categories can be described very explicitly in order to check the condition about epimorphisms in Theorem~\ref{thm:description} for the tensor categories constructed in \S \ref{section:universal_geometric_cats}. This description is standard in the unenriched case. Since this is a key ingredient in the main result of this paper we will give a detailed proof that it holds for categories of $R$-linear sheaves as well. We start by recalling some terminology and results from \cite{BORCEUX_QUINTEIRO}.

\begin{dfn}\label{dfn:sieve}
 An \emph{$R$-linear sieve} on an object $A \in \ca{A}$ is an $R$-linear presheaf $S \colon \ca{A}^{\op} \rightarrow \Mod_R$ on $\ca{A}$ which is a subobject of the representable presheaf $\ca{A}(-,A)$.
\end{dfn}

 Note that $R$-linear sieves differ from ordinary (unenriched) sieves: for all objects $B \in \ca{A}$, $SB$ is a \emph{submodule} of $\ca{A}(B,A)$ instead of a mere subset. Thus we cannot generate an $R$-linear sieve from a collection of morphisms with target $A$ in the same way as in the unenriched setting: in order to get a submodule we would have to take finite linear combinations. This problem does however not arise in the case of singleton coverings, which is the only case we will need.

\begin{example}\label{example:S_p}
 Let $q \colon A \rightarrow B$ be a morphism in $\ca{A}$. We write $S_q$ for the image of
\[
 \ca{A}(-,q) \colon \ca{A}(-,A) \rightarrow \ca{A}(-,B)
\]
 in the category $\Prs{A}$ of $R$-linear presheaves on $\ca{A}$. This is an $R$-linear sieve on $A$, and $S_q(C)$ consists of all morphisms $C \rightarrow B$ which factor through $q$.
\end{example}

 Note that the sieve $S_q$ is constructed as the corresponding sieve in the unenriched case. For this reason, we encounter no serious difficulties when generalizing classical results to the $R$-linear case. If we wanted to allow coverages whose coverings are not singletons, we would have to find a way to deal with the fact that coproducts in $\Prs{A}$ are neither disjoint nor universal. The following definition gives a special case of an enriched Grothendieck topology in the sense of Borceux and Quinteiro (see \cite[\S 1]{BORCEUX_QUINTEIRO}).

\begin{dfn}\label{dfn:grothendieck_topology}
 An \emph{$R$-linear Grothendieck topology} $\tau$ on $\ca{A}$ consists of a collection $\tau_A$ of $R$-linear sieves on $A$ for all objects $A \in \ca{A}$, subject to the following closure conditions:
\begin{enumerate}
 \item[(i)] For all $A \in \ca{A}$, $\id_{\ca{A}(-,A)}$ lies in $\tau_A$;
\item[(ii)] If $f \colon A \rightarrow B$ is a morphism in $\ca{A}$, $S \in \tau_B$, and
\[
 \xymatrix{S^{\prime} \ar[r] \ar[d] & S \ar[d] \\ \ca{A}(-,A) \ar[r]_-{\ca{A}(-,f)} & \ca{A}(-,B)}
\]
 is a pullback diagram in $\Prs{A}$, then $S^{\prime} \in \tau_A$;
\item[(iii)] If $S$ is any sieve on $B$, $T \in \tau_B$, and for all pullback diagrams
\[
 \xymatrix{f^{-1} S \ar[r] \ar[d] & S \ar[d] \\ \ca{A}(-,A) \ar[r]_-{\ca{A}(-,f)} & \ca{A}(-,B)} 
\]
 with $f \in TA$, the $R$-linear sieve $f^{-1} S$ lies in $\tau_A$, then $S$ lies in $\tau_B$.
\end{enumerate}
\end{dfn}

\begin{dfn}\label{dfn:sheaf_for_grothendieck_topology}
 Let $\ca{A}$ be an $R$-linear category and let $S$ be an $R$-linear sieve on $A \in \ca{A}$. Let $i \colon S \rightarrow \ca{A}(-,A)$ denote the inclusion of $S$ in the representable presheaf. We say that an $R$-linear presheaf $F \colon \ca{A}^{\op} \rightarrow \Mod_R$ is a \emph{sheaf for $S$} if
\[
 \Prs{A}(i,F) \colon \Prs{A}\bigl(\ca{A}(-,A),F\bigr) \rightarrow \Prs{A}(S,F)
\]
 is an bijection. Equivalently, if for all solid arrow diagrams
\[
 \vcenter{
 \xymatrix{ S \ar[d]_i \ar[r] & F \\ \ca{A}(-,A) \ar@{-->}[ru]}
 }
\]
 in $\Prs{A}$ there exists a unique morphism $\ca{A}(-,A) \rightarrow F$ making the above diagram commutative.

 If $\tau$ is an $R$-linear Grothendieck topology on $\ca{A}$, then we say that $F$ is a \emph{sheaf for $\tau$} if $F$ is a sheaf for all $R$-linear sieves $S \in \tau_A$ and all $A \in \ca{A}$. We write $\Sh_{\tau}(\ca{A})$ for the full subcategory of $\Prs{A}$ consisting of sheaves for $\tau$.
\end{dfn}

 These correspond with the definitions of enriched Grothendieck topologies and enriched sheaves given in \cite{BORCEUX_QUINTEIRO} (by \cite[Lemma~1.6]{BORCEUX_QUINTEIRO}, applied to the strong generator $\{R\}$ of $\Mod_R$). 

\subsection{The topology induced by a coverage}
 We will see below that any singleton coverage $\Theta$ induces an $R$-linear Grothendieck topology $\tau(\Theta)$ in such a way that $F$ is a sheaf for $\Theta$ in the sense of Definition~\ref{dfn:sheaf_for_coverage} if and only if it is a sheaf for $\tau(\Theta)$ in the sense of Definition~\ref{dfn:sheaf_for_grothendieck_topology}. The proof mostly follows the proof of the corresponding fact in the unenriched case, and the result was implicitly used in \cite[\S 2]{SCHAEPPI_INDABELIAN}. Since it is a key ingredient in the proof of the main result of this article, we provide full details below.

\begin{lemma}\label{lemma:quasi_pullback}
 Let $\Theta$ be a singleton coverage on the $R$-linear category $\ca{A}$, and let $q \colon B^{\prime} \rightarrow B$ be a composite of n morphisms in $\Theta$ for some $n \in \mathbb{N}$. Then for any morphism $f \colon A \rightarrow B$ in $\ca{A}$ there exists a composite $q^{\prime} \colon A^{\prime} \rightarrow A$ of $n$ arrows in $\Theta$ and a morphism $S_{q^{\prime}} \rightarrow S_q$ in $\Prs{A}$ such that the diagram
\begin{equation}\label{eqn:quasi_pullback}
 \vcenter{
 \xymatrix{ S_{q^{\prime}} \ar[r] \ar[d] & S_q \ar[d] \\ \ca{A}(-,A) \ar[r]_-{\ca{A}(-,f)} & \ca{A}(-,B)
}}
\end{equation}
 is commutative.
\end{lemma}

\begin{proof}
 Applying the defining property of singleton coverages $n$ times we find that there exists a commutative diagram
\[
 \xymatrix{A^{\prime} \ar[r]^-{f^{\prime}} \ar[d]_{q^{\prime}} &  B^{\prime} \ar[d]^{q} \\ A \ar[r]_-{f} & B }
\]
 such that $q^{\prime}$ is a composite of $n$ morphisms in $\Theta$. The claim follows by taking images of the vertical morphisms of the induced diagram in $\Prs{A}$.
\end{proof}

\begin{dfn}\label{dfn:topology_from_coverage}
 Let $\Theta$ be a singleton coverage on the $R$-linear category $\ca{A}$. For $A\in \ca{A}$, let $\tau(\Theta)_A$ be the set of $R$-linear sieves on $A$ given by the union of $\{ \id \colon \ca{A}(-,A) \rightarrow \ca{A}(-,A)\}$ and the set of all $R$-linear sieves $S$ on $A$ for which there exists a finite composite $q$ of morphisms in $\Theta$ (with target $A$) such that $S_q \subseteq S$.
\end{dfn}

\begin{prop}\label{prop:topology_form_coverage}
 Let $\Theta$ be a singleton coverage on the $R$-linear category $\ca{A}$. Then $\tau(\Theta)$ is an $R$-linear Grothendieck coverage on $\ca{A}$.
\end{prop}

\begin{proof}
 Condition~(i) of Definition~\ref{dfn:grothendieck_topology} is satisfied by definition, and Condition~(ii) follows from the existence of Diagram~\eqref{eqn:quasi_pullback} in Lemma~\ref{lemma:quasi_pullback}.

 To see that Condition~(iii) holds, let $S$ be an $R$-linear sieve on $B$, let $T \in \tau(\Theta)_B$, and assume that the pullback $f^{-1} S$ of $S$ along $\ca{A}(-,f) \colon \ca{A}(-,A) \rightarrow \ca{A}(-,B)$ lies in $\tau(\Theta)_A$ for all $f \in TA$ (and all $A \in \ca{A}$). If $\id_B \in TB$, it follows immediately that $S \in \tau(\Theta)_B$. 

 It remains to check that the same is true if we only know that $T$ contains the $R$-linear sieve $S_q$ for some finite composite $q \colon A \rightarrow B$ of morphisms in $\Theta$. In this case we have $q \in TA$, so by assumption, the pullback $q^{-1} S$ lies in $\tau(\Theta)_A$. There are now two cases: either $q^{-1}S$ is equal to the identity on $\ca{A}(-,A)$, or there exists a finite composite $q^{\prime}$ of morphisms in $\Theta$ with target $A$ such that $S_{q^{\prime}} \subseteq q^{-1} S$. In the latter case, $S$ contains $S_{qq^{\prime}}$ (any composite $qq^{\prime}f$ is the image of an element of $S_{q^{\prime}}$ under the morphism $\ca{A}(-,q) \colon \ca{A}(-,A) \rightarrow \ca{A}(-,B)$), hence $S$ lies in $\tau(\Theta)_B$.

 In the former case, that is, if $q^{-1} S=\ca{A}(-,A)$, the pair of pullback diagrams
\[
 \xymatrix{\ca{A}(-,A) \ar@{=}[d] \ar[r] & S \cap S_q \ar[d] \ar[r] & S \ar[d] \\ \ca{A}(-,A) \ar@{->>}[r] & S_q \ar[r] & \ca{A}(-,B)}
\]
 shows that the inclusion $S \cap S_q \rightarrow S_q$ is an epimorphism, hence that $S_q \subseteq S$. This concludes the proof that $\tau(\Theta)$ is an $R$-linear Grothendieck topology.
\end{proof}

\begin{lemma}\label{lemma:sheaves_for_coverage}
 Let $\Theta$ be a singleton coverage on the essentially small $R$-linear category $\ca{A}$ and let $F \colon \ca{A}^{\op} \rightarrow \Mod_R$ be an $R$-linear presheaf. Then $F$ is a sheaf for $\Theta$ in the sense of Definition~\ref{dfn:sheaf_for_coverage} if and only if $F$ is a sheaf for all sieves $S_q$ with $q \in \Theta$.
\end{lemma}

\begin{proof}
 For all morphisms $q \colon A \rightarrow B$ in $\ca{A}$ we have a right exact sequence
\[
 \xymatrix{\bigoplus_{\{f \colon A^{\prime}\rightarrow A \vert qf=0 \}} \ca{A}(-,A^{\prime}) \ar[r] & \ca{A}(-,A) \ar@{->>}[r]^-{\ca{A}(-,q)} & S_q \ar[r] & 0}
\]
 in $\Prs{A}$ (since the representable presheaves generate $\Prs{A}$). Using the fact that $\Prs{A}(-,F)$ sends right exact sequences to left exact sequences we find that $F$ is a sheaf for $\Theta$ if and only if
\[
 \Prs{A}\bigl(\ca{A}(-,B),F\bigr) \rightarrow \Prs{A}(S_q,F)
\]
 is an isomorphism for all $q \in \Theta$.
\end{proof}

\begin{prop}\label{prop:sheaves_coincide}
 Let $\Theta$ be a singleton coverage on the $R$-linear category $\ca{A}$ and let $F \colon \ca{A}^{\op} \rightarrow \Mod_R$ be an $R$-linear presheaf. Then $F$ is a sheaf for the singleton coverage $\Theta$ in the sense of Definition~\ref{dfn:sheaf_for_coverage} if and only if $F$ is a sheaf for the $R$-linear Grothendieck topology $\tau(\Theta)$ in the sense of Definition~\ref{dfn:sheaf_for_grothendieck_topology}. 
\end{prop}

\begin{proof}
 It follows from Lemma~\ref{lemma:sheaves_for_coverage} that sheaves for $\tau(\Theta)$ are sheaves for $\Theta$, so it remains to check the converse.

 The sheaf condition for the identity sieve is vacuous, so we only need to check that
\[
 \Prs{A}\bigl(\ca{A}(-,B),F \bigr) \rightarrow \Prs{A}(S,F)
\]
 is an isomorphism if $S$ contains $S_q$ for some finite composite $q \colon A \rightarrow B$ of morphisms in $\Theta$. We prove this by induction on the length of the composite $q$.

 From Lemma~\ref{lemma:sheaves_for_coverage} we know that $F$ is a sheaf for the $R$-linear sieve $S_q$ on $A$ if $q$ is a morphism in $\Theta$ (that is, the composite of a single morphism in $\Theta$). Assume by induction that the same is true for all sieves $S_q$ on $A$ where $q$ is a composite of at most $n$ morphisms in $\Theta$. We will first show that $F$ is a sheaf for all $R$-linear sieves $S$ containing such a sieve $S_q$, and we then use this fact to show that $F$ is a sheaf for all $S_q$ where $q$ is a composite of at most $n+1$ morphisms in $\Theta$. By induction on $n$, these two facts imply that $F$ is a sheaf for $\tau(\Theta)$ (see Definition~\ref{dfn:topology_from_coverage}).

 Thus let $q$ be a composite of at most $n$ morphisms in $\Theta$ and let $S \subseteq \ca{A}(-,B)$ be an $R$-linear sieve containing $S_q$. Since $F$ is a sheaf for $S_q$, for any $\lambda \colon S \rightarrow F$ there exists a unique morphism $\mu \colon \ca{A}(-,B) \rightarrow F$ such that the diagram
\begin{equation}\label{eqn:intermediate_sieve}
 \vcenter{
\xymatrix{ S_q \ar[rd]^{\lambda i} \ar[d]_{i} \\ S \ar[d]_{j} & F \\ \ca{A}(-,B) \ar@{-->}[ru]_{\mu}
}}
\end{equation}
 is commutative. It only remains to check that $\mu j=\lambda$. Since the representable presheaves generate $\Prs{A}$, it suffices to check that $\mu j k=\lambda k$ for all $k \colon \ca{A}(-,A) \rightarrow S$. The composite $jk$ is of the form $\ca{A}(-,f)$ for some $f \colon A \rightarrow B$ by the Yoneda lemma. By Lemma~\ref{lemma:quasi_pullback}, there exists a composite $q^{\prime}$ of at most $n$ morphisms in $\Theta$ and a morphism $S_{q^{\prime}} \rightarrow S_q$ such that the diagram
\[
 \xymatrix{S_{q^{\prime}} \ar[d]_{i^{\prime}}\ar[r] & S_q \ar[d]^{ji} \\ \ca{A}(-,A) \ar[r]_-{\ca{A}(-,f)} & \ca{A}(-,B) }
\]
 is commutative. Since $j$ is a monomorphism, the square
\[
 \xymatrix{S_{q^{\prime}} \ar[d]_{i^{\prime}}\ar[r] & S_q \ar[d]^{i} \\ \ca{A}(-,A) \ar[r]_-{k} & S}
\]
 is commutative as well. The assumption that $F$ is a sheaf for $S_{q^{\prime}}$ implies that it suffices to check that the equation $\mu j k i^{\prime}=\lambda k i^{\prime}$ holds, which follows from commutativity of the above square and of Diagram~\eqref{eqn:intermediate_sieve}. This concludes the proof that $F$ is a sheaf for all $R$-linear sieves that contain an $R$-linear sieve $S_q$ where $q \colon A \rightarrow B$ is a composite of at most $n$ morphisms in $\Theta$.

 It remains to show that for any such morphism $q$ and any morphism $p \colon A^{\prime} \rightarrow A$ in $\Theta$, $F$ is a sheaf for $S_{qp}$ (as mentioned above, the claim then follows by induction on $n$). Let
\begin{equation}\label{eqn:composite_pullback}
\vcenter{
 \xymatrix{S \ar[r] \ar@{->>}[d] & \ca{A}(-,A) \ar@{->>}[d] \\ S_{qp} \ar[r] & S_q}
}
\end{equation}
 be the pullback of the epimorphism $\ca{A}(-,A) \rightarrow S_q$ along the inclusion $S_{qp} \rightarrow S_q$. From the commutative diagram
\[
 \xymatrix{S_p \ar[r] \ar[d] & \ca{A}(-,A) \ar[d] \\ S_{qp} \ar[r] & S_q}
\]
 (with vertical morphisms given by composition with $q$ and horizontal morphisms given by inclusions) we deduce that $S_p \subseteq S$. From the fact we proved above it follows that $F$ is a sheaf for the $R$-linear sieve $S$ on $A$. Diagram~\eqref{eqn:composite_pullback} is a pullback of an epimorphism along a monomorphism in the abelian category $\Prs{A}$, so it is also a pushout diagram, and therefore sent to a pullback diagram by the contravariant hom-functor $\Prs{A}(-,F)$. Since isomorphisms are stable under pullback it follows that
\[
 \Prs{A}(S_q, F) \rightarrow \Prs{A}(S_{qp},F)
\]
 is an isomorphism. But $F$ is a sheaf for $S_q$ by assumption, hence the composite
\[
 \Prs{A}\bigl(\ca{A}(-,B), F\bigr) \rightarrow \Prs{A}(S_q, F) \rightarrow \Prs{A}(S_{qp},F)
\]
 is an isomorphism as well. This concludes the proof that $F$ is a sheaf for $S_{qp}$ and---by induction on $n$---that $F$ is a sheaf for $\tau(\Theta)$.
\end{proof}

 One of the main results of \cite{BORCEUX_QUINTEIRO} specializes to the following theorem about $R$-linear sheaves for $R$-linear Grothendieck topologies.

\begin{thm}\label{thm:borceux_quinteiro}
 Let $\tau$ be an $R$-linear Grothendieck topology on the $R$-linear category $\ca{A}$. Then the inclusion of $\Sh_\tau(\ca{A})$ in $\Prs{A}$ has an exact $R$-linear left adjoint. In particular, $\Sh_\tau(\ca{A})$ is an abelian category.
\end{thm}

\begin{proof}
 This is proved in \cite[Theorem~4.4]{BORCEUX_QUINTEIRO} for general enriched Grothendieck topologies.
\end{proof}

\subsection{Epimorphisms in the category of sheaves}
 As in the unenriched case, coverages may satisfy various closure properties that give better control over categorical concepts---such as epimorphisms---in the category of $R$-linear sheaves.

\begin{dfn}\label{dfn:closure_properties}
 Let $\Theta$ be a singleton coverage on the $R$-linear category $\ca{A}$. A \emph{closure property} for $\Theta$ is a property of the following list:
\begin{enumerate}
 \item[(i)] If $p$ and $q$ are composable morphisms in $\Theta$, then $pq \in \Theta$;
 \item[(ii)] If $p$ and $f$ are composable morphisms and $pf \in \Theta$, then $p \in \Theta$;
 \item[(iii)] If $p \colon A \rightarrow B$ and $p^{\prime} \colon A^{\prime} \rightarrow B$ are morphisms in $\Theta$, there exist morphisms $f \colon A^{\prime\prime} \rightarrow A$ and $f^{\prime} \colon A^{\prime \prime} \rightarrow A^{\prime}$ in $\ca{A}$ such that the diagram
\[
 \xymatrix{A^{\prime \prime} \ar[r]^-{f^{\prime}} \ar[d]_f & A^{\prime} \ar[d]^{p^{\prime}} \\ A \ar[r]_-{p} & B }
\]
 is commutative and the composite $A^{\prime \prime} \rightarrow B$ lies in $\Theta$;
 \item[(iv)] The category $\ca{A}$ has pullbacks of morphisms in $\Theta$ along arbitrary morphisms in $\ca{A}$ and $\Theta$ is stable under pullbacks, that is, the pullback of any morphism in $\Theta$ lies in $\Theta$.
\end{enumerate}
\end{dfn}

 Note that these closure properties are not independent, for example, Property~(i) and (iv) together imply Property~(iii) (and also that $\Theta$ is a singleton coverage).

\begin{example}\label{example:coverage_of_interest}
 Let $\ca{A}$ be a full subcategory of the lfp tensor category $\ca{C}$ over $R$ which satisfies the conditions of Theorem~\ref{thm:universal_geometric_category}, that is, it consists of locally free objects of constant finite rank, and the following conditions hold:
\begin{enumerate}
 \item[(i)] The subcategory $\ca{A}$ is closed under finite direct sums, finite tensor products, and duals;
 \item[(ii)] If $C \in \ca{C}$ is locally free of constant finite rank and $p \colon A \rightarrow C$ is a locally split epimorphism in $\ca{C}$ with $A \in \ca{A}$, then $C \in \ca{A}$. 
\end{enumerate}
 Let $\Theta$ be the class of all morphisms in $\ca{A}$ which are locally split epimorphisms in $\ca{C}$. Then $\Theta$ is a coverage which satisfies all the closure properties of Definition~\ref{dfn:closure_properties}. Indeed, Conditions~(i) and (ii) follow directly from the fact that split epimorphisms satisfy the corresponding properties, and Condition~(iv) follows from Lemma~\ref{lemma:pullback_of_locally_split}.
\end{example}

\begin{lemma}\label{lemma:generated_subsheaf}
 Let $\Theta$ be a singleton coverage on the $R$-linear category $\ca{A}$ which satisfies Conditions~(i) and (iii) of Definition~\ref{dfn:closure_properties}. Let $F$ be a sheaf for $\Theta$ and let $G \subseteq F$ be a subobject in $\Prs{A}$ (so $G$ need \emph{not} be a sheaf for $\Theta$). Let
\[
 HB \defl \{ x \in FB \mid \exists p \colon A \rightarrow B \in \Theta \enspace \text{such that} \enspace x \cdot p \in GA \} \smash{\rlap{.}}
\]
 Then $H$ is a subobject of $F$ in $\Sh_{\Theta}(\ca{A})$.
\end{lemma}

\begin{proof}
 The existence of the commutative diagram in Condition~(iii) of Definition~\ref{dfn:closure_properties} implies that $HB$ is a submodule of $FB$ for all $B \in \ca{A}$. Given an arbitrary morphism $f \colon A \rightarrow B$ in $\ca{A}$, the defining property of a singleton coverage shows that $Ff(x) \in HA$ for all $x \in HB$. Thus $H$ is indeed a subfunctor of $F$. It follows in particular that $H$ is $R$-linear, that is, $H$ is a subobject of $F$ in $\Prs{A}$.

 It only remains to show that $H$ is an $R$-linear sheaf for $\Theta$. By Lemma~\ref{lemma:sheaves_for_coverage}, it suffices to check that
\[
 \Prs{A}\bigl(\ca{A}(-,A), H \bigr) \rightarrow \Prs{A}(S_p,H)
\]
 is an isomorphism for all $p \colon A \rightarrow B$ in $\Theta$. Since $F$ is a sheaf for $\Theta$ and the inclusion $H \rightarrow F$ is a monomorphism, $H$ is \emph{separated}, that is, an extension
\[
 \xymatrix{S_p \ar[r] \ar[d] & H \\ \ca{A}(-,B) \ar@{-->}[ru]}
\]
 is unique if it exists. Moreover, there exists a unique morphism $g$ such that the diagram
\[
 \xymatrix{S_p \ar[r] \ar[d] & H \ar[d] \\ \ca{A}(-,B) \ar[r]_-{g} & F }
\]
 is commutative. It only remains to check that $g$ factors through the inclusion $H \rightarrow F$. By Yoneda, $g$ corresponds to the unique element $x=g_B(\id_B) \in FB$, so it suffices to check that $x \in HB$. Precomposing the above square with the projection $\ca{A}(-,A) \rightarrow S_p$ we find that $g \circ \ca{A}(-,p)$ factors through $H$, that is, $x \cdot p \in HA$. By definition of $H$, this implies that there exists a morphism $q \colon A^{\prime} \rightarrow A$ in $\Theta$ such that $x \cdot p \cdot q \in GA^{\prime}$. Since $\Theta$ is closed under composition (Condition~(i) of Definition~\ref{dfn:closure_properties}) this shows that $x \in HB$.
\end{proof}

 Recall that a (singleton) coverage is called \emph{subcanonical} if all representable presheaves are sheaves. The following lemma gives a sufficent condition for a singleton coverage to be subcanonical.

\begin{lemma}\label{lemma:subcanonical}
 Let $\Theta$ be a singleton coverage on the $R$-linear category $\ca{A}$ such that all morphisms in $\Theta$ are cokernels in $\ca{A}$. Then $\Theta$ is subcanonical. 
\end{lemma}

\begin{proof}
 Fix an object $X \in \ca{A}$. We will check directly that $\ca{A}(-,X)$ satisfies the conditions of Definition~\ref{dfn:sheaf_for_coverage}. By assumption, for any morphism $q \colon B \rightarrow C$ in $\Theta$ there exists a morphism $p \colon A \rightarrow B$ in $\ca{A}$ such that $q$ is the cokernel of $p$. Since $\ca{A}(-,X)$ sends cokernels to kernels, $\ca{A}(q,X)$ is a monomorphism. Moreover, if $g \in \ca{A}(B,X)$ is a morphism such that $gf=0$ for all $f \colon A^{\prime} \rightarrow B$ in $\ca{A}$ with $qf=0$, then we find in particular that $gp=0$. Thus there exists a morphism $g^{\prime} \in \ca{A}(C,X)$ such that $g^{\prime}q=g$. These two facts imply that $\ca{A}(-,X)$ is a sheaf for $\Theta$.
\end{proof}

\begin{prop}\label{prop:epimorphisms_of_representable_sheaves}
 Let $\Theta$ be a subcanonical singleton coverage on the $R$-linear category $\ca{A}$ and let $f \colon A \rightarrow B$ be a morphism in $\ca{A}$. If $\Theta$ satisfies Conditions~(i), (ii), and (iii) of Definition~\ref{dfn:closure_properties}, then
\[
 \ca{A}(-,f) \colon \ca{A}(-,A) \rightarrow \ca{A}(-,B)
\]
 is an epimorphism in $\Sh_{\Theta}(\ca{A})$ if and only if $f$ lies in $\Theta$.
\end{prop}

\begin{proof}
 If $f$ lies in $\Theta$, then $Ff$ is a monomorphism for any $R$-linear sheaf $F$ (see Definition~\ref{dfn:sheaf_for_coverage}). By Yoneda, this is equivalent to the fact that 
\[
\Prs{A}\bigl(\ca{A}(-,f),F \bigr) \colon \Prs{A}\bigl(\ca{A}(-,B),F \bigr) \rightarrow \Prs{A}\bigl(\ca{A}(-,A),F \bigr) 
\]
 is a monomorphism for all $F \in \Sh_{\Theta}(\ca{A})$, so $\ca{A}(-,f)$ is indeed an epimorphism.

 Conversely, suppose that $\ca{A}(-,f)$ is an epimorphism in $\Sh_{\Theta}(\ca{A})$. Let $S_f$ be the image of $\ca{A}(-,f)$ in $\ca{A}(-,B)$, computed in the category $\Prs{A}$. By Lemma~\ref{lemma:generated_subsheaf}, the $R$-linear presheaf $S^{\prime}_f$ given by
\[
 S^{\prime}_f C= \{ g \colon C \rightarrow B \mid \exists q \colon C^{\prime} \rightarrow C \in \Theta \enspace \text{such that} \enspace gq \in S_f(C^{\prime}) \}
\]
 is an $R$-linear sheaf for $\Theta$. By construction, the epimorphism $\ca{A}(-,f)$ factors through the inclusion $S^{\prime}_f \rightarrow \ca{A}(-,B)$, so this inclusion is both an epimorphism and a monomorphism in $\Sh_{\Theta}(\ca{A})$. Since $\Sh_{\Theta}(\ca{A})$ is abelian (see Theorem~\ref{thm:borceux_quinteiro} and Proposition~\ref{prop:sheaves_coincide}) this implies that the inclusion is an isomorphism in $\Sh_{\Theta}(\ca{A})$ (and thus in particular in $\Prs{A}$). Therefore $\id_B$ lies in $S^{\prime}_f (B)$, that is, there exists $q \colon A^{\prime} \rightarrow B $ in $\Theta$ such that $\id_B \circ q \in S_f$. This shows that $q=fh$ for some $h \colon A^{\prime} \rightarrow A$, and it follows from Condition~(ii) of Definition~\ref{dfn:closure_properties} that $f \in \Theta$.
\end{proof}

 Proposition~\ref{prop:epi_characterization} is a straightforward consequence of the above proposition.

\begin{proof}[Proof of Proposition~\ref{prop:epi_characterization}]
 Let $\ca{C}$ be an lfp tensor category over $R$ and let $\ca{A} \subseteq \ca{C}$ be a full subcategory satisfying the conditions of Theorem~\ref{thm:universal_geometric_category} (see also Example~\ref{example:coverage_of_interest}). We need to show that $\ca{A}(-,f) \colon \ca{A}(-,A) \rightarrow \ca{A}(-,B)$ is an epimorphism in $G(\ca{A})=\Lex_{\Sigma}[\ca{A}^{\op},\Mod_R]$ if and only if $f$ is a locally split epimorphism in $\ca{C}$, where $\Sigma$ is the class of all cokernel diagrams in $\ca{A}$ which are locally split right exact sequences in $\ca{C}$.

 From Proposition~\ref{prop:lex_sigma_equals_sheaves} we know that $\Lex_{\Sigma}[\ca{A}^{\op},\Mod_R]=\Sh_{R(\Sigma)}(\ca{A})$, where $R(\Sigma)$ is the singleton coverage associated to $\Sigma$ (see Example~\ref{example:ind_class_gives_covering}). Moreover, from Lemma~\ref{lemma:pullback_of_locally_split} we know that $R(\Sigma)$ is the singleton coverage of morphisms in $\ca{A}$ which are locally split epimorphisms in $\ca{C}$. This coverage satisfies all the closure properties of Definition~\ref{dfn:closure_properties} by Example~\ref{example:coverage_of_interest}. Thus Proposition~\ref{prop:epimorphisms_of_representable_sheaves} is applicable, and it follows that $\ca{A}(-,f)$ is an epimorphism in $\Sh_{R(\Sigma)}(\ca{A})=\Lex_{\Sigma}[\ca{A}^{\op},\Mod_R]$ if and only if $f$ lies in $R(\Sigma)$, as claimed. 
\end{proof}

\bibliographystyle{amsalpha}
\bibliography{colimits}

\end{document}